\documentclass[twoside,11pt]{article}
\usepackage{jmlr2e}
\usepackage{amsfonts,amssymb,amsmath}
\usepackage[noabbrev,nameinlink, capitalize]{cleveref}
\usepackage{mathtools, mathrsfs}
\usepackage{bm}
\usepackage{graphicx}
\usepackage{graphbox}
\usepackage{tabularx}
\usepackage[ruled,vlined, linesnumbered]{algorithm2e}
\usepackage{enumitem}
\usepackage{array}
\usepackage{multirow}
\usepackage{makecell}
\usepackage{appendix}
\usepackage{empheq}
\usepackage{cases}
\usepackage{textcomp}
\usepackage{gensymb}
\usepackage{subcaption}
\usepackage{float}
\usepackage{regexpatch}
\usepackage{tcolorbox} 
\usepackage{xcolor} 
\usepackage{algorithmic}
\usepackage{eqparbox}
\usepackage{pifont}
\usepackage{hhline}
\usepackage{caption}
\usepackage{enumitem}
\usepackage{diagbox}

\usepackage{tikz}
\usetikzlibrary{calc,patterns,decorations.pathreplacing,calligraphy}
\usetikzlibrary{shapes,positioning}
\usetikzlibrary{shapes.geometric}
\tikzstyle{bag} = [align=center]

\newcommand{\R}{\mathbb{R}}

\newcommand{\N}{\mathbb{N}}
\newcommand{\Z}{\mathbb{Z}}
\newcommand{\norm}[1]{\left\lVert #1 \right\rVert}
\newcommand{\lap}{\Delta}
\newcommand{\grad}{\nabla}

\DeclareMathOperator*{\argmin}{arg\,min}

\DeclareMathOperator*{\argmax}{arg\,max}

\newtheorem{assumption}{Assumption}
\crefname{assumption}{Assumption}{Assumptions}
\crefname{algorithm}{Algorithm}{Algorithms}
\crefname{definition}{Definition}{Definitions}
\crefname{remark}{Remark}{Remarks}
\crefname{equation}{}{}
\crefname{appendix}{Appendix}{Appendices}
\crefname{lemma}{Lemma}{Lemmas}
\crefname{proposition}{Proposition}{Propositions}
\crefname{theorem}{Theorem}{Theorems}

\newcommand{\bdmc}[1]{\boldsymbol{\mathcal{#1}}}

\newcommand{\dd}{\,\text{\normalfont\ignorespaces d}}

\newcommand{\bzero}{\bm{0}}

\newcommand{\be}{\bm{e}}
\newcommand{\bff}{\bm{f}}
\newcommand{\bg}{\bm{g}}

\newcommand{\bmm}{\bm{m}}
\newcommand{\bn}{\bm{n}}

\newcommand{\bp}{\bm{p}}

\newcommand{\bu}{\bm{u}}
\newcommand{\bv}{\bm{v}}
\newcommand{\bw}{\bm{w}}
\newcommand{\bx}{\bm{x}}
\newcommand{\by}{\bm{y}}

\newcommand{\bA}{\bm{A}}
\newcommand{\bB}{\bm{B}}
\newcommand{\bC}{\bm{C}}
\newcommand{\bD}{\bm{D}}

\newcommand{\bF}{\bm{F}}

\newcommand{\bH}{\bm{H}}
\newcommand{\bI}{\bm{I}}
\newcommand{\bJ}{\bm{J}}

\newcommand{\bM}{\bm{M}}

\newcommand{\bP}{\bm{P}}

\newcommand{\bV}{\bm{V}}

\newcommand{\bxi}{{\bm{\xi}}}

\newcommand{\calA}{{\mathcal{A}}}
\newcommand{\calB}{{\mathcal{B}}}
\newcommand{\calC}{{\mathcal{C}}}
\newcommand{\calD}{{\mathcal{D}}}
\newcommand{\calE}{{\mathcal{E}}}
\newcommand{\calF}{{\mathcal{F}}}

\newcommand{\calH}{{\mathcal{H}}}
\newcommand{\calI}{{\mathcal{I}}}

\newcommand{\calK}{{\mathcal{K}}}
\newcommand{\calL}{{\mathcal{L}}}
\newcommand{\calM}{{\mathcal{M}}}
\newcommand{\calN}{{\mathcal{N}}}

\newcommand{\calP}{{\mathcal{P}}}
\newcommand{\calQ}{{\mathcal{Q}}}
\newcommand{\calR}{{\mathcal{R}}}
\newcommand{\calS}{{\mathcal{S}}}
\newcommand{\calT}{{\mathcal{T}}}
\newcommand{\calU}{{\mathcal{U}}}
\newcommand{\calV}{{\mathcal{V}}}
\newcommand{\calW}{{\mathcal{W}}}

\newcommand{\scrA}{{\mathscr{A}}}

\newcommand{\scrE}{{\mathscr{E}}}

\newcommand{\scrH}{{\mathscr{H}}}

\newcommand{\scrM}{{\mathscr{M}}}
\newcommand{\scrN}{{\mathscr{N}}}

\newcommand{\scrP}{{\mathscr{P}}}

\newcommand{\scrU}{{\mathscr{U}}}
\newcommand{\scrV}{{\mathscr{V}}}

\newcommand{\scrX}{{\mathscr{X}}}
\newcommand{\scrY}{{\mathscr{Y}}}

\newcolumntype{C}{>{$\displaystyle}c<{$}}
\newcolumntype{V}{ >{\centering\arraybackslash} m{1cm} }

\newcommand{\cpr}{\mathcal{C}_{\text{\normalfont\ignorespaces pr}}}
\newcommand{\cpo}{\mathcal{C}_{\text{\normalfont\ignorespaces post}}}

\newcommand{\sder}{D_{\scrH_{\mu}}}
\newcommand{\nn}{\bff_{\text{\normalfont\ignorespaces NN}}}

\newcommand{\iid}{\stackrel{\text{\normalfont\ignorespaces i.i.d.}}{\sim}}
\newcommand{\meas}{\text{\normalfont\ignorespaces d}}

\newcommand{\HS}{\text{\normalfont\ignorespaces HS}}
\newcommand{\mmala}{\text{\normalfont\ignorespaces mMALA}}
\newcommand{\disbasis}[1]{\psi^{\text{\normalfont\ignorespaces DIS}}_{#1}}
\newcommand{\disev}[1]{\lambda^{\text{\normalfont\ignorespaces DIS}}_{#1}}
\newcommand{\klebasis}[1]{\psi^{\text{\normalfont\ignorespaces KLE}}_{#1}}
\newcommand{\kleev}[1]{\lambda^{\text{\normalfont\ignorespaces KLE}}_{#1}}

\newcommand*{\vertbar}{\rule[-1.5ex]{0.5pt}{3ex}}
\newcommand*{\horzbar}{\rule[.5ex]{2.5ex}{0.5pt}}
\newcommand{\cmark}{\ding{51}}%
\newcommand{\xmark}{\ding{55}}%

\newcommand{\multiline}[1]{%
    \begin{tabularx}{\dimexpr\linewidth-\ALG@thistlm}[t]{@{}X@{}}
        #1
    \end{tabularx}
}

\makeatletter
\renewcommand*{\thanks}[1]{%
  \footnotemark
  \protected@xdef\@thanks{\@thanks
    \protect\footnotetext[\arabic{footnote}]{#1}}%
}
\makeatother

\usepackage{lastpage}
\jmlrheading{xx}{xxxx}{x-xx}{xxxxx}{xxx}{xxxx}{xxxxx}

\ShortHeadings{DINO Acceleration of Geometric MCMC}{Cao, O'Leary-Roseberry, and Ghattas}
\firstpageno{1}

\begin{document}

\title{Derivative-Informed Neural Operator Acceleration of Geometric MCMC for Infinite-Dimensional Bayesian Inverse Problems}

\author{\name Lianghao Cao\thanks{Corresponding author} \email lianghao@caltech.edu \\
       \addr Department of Computing and Mathematical Sciences\\
       California Institute of Technology\\
       Pasadena, CA 91125, USA.
       \AND
       \name Thomas O'Leary-Roseberry$^\dagger$ \email tom.olearyroseberry@utexas.edu \\
        \name Omar Ghattas$^{\dagger,\ddagger}$ \email omar@oden.utexas.edu\\
       \addr $^\dagger$Oden Institute for Computational Engineering and Sciences\\
       \addr $^\ddagger$Waker Department of Mechanical Engineering\\
       The University of Texas at Austin\\
       Austin, TX 78712, USA.
       }
\editor{ }
\maketitle
\begin{abstract}%
We propose an operator learning approach to accelerate geometric Markov chain Monte Carlo (MCMC) for solving infinite-dimensional Bayesian inverse problems (BIPs). While geometric MCMC employs high-quality proposals that adapt to posterior local geometry, it requires repeated computations of gradients and Hessians of the log-likelihood, which becomes prohibitive when the parameter-to-observable (PtO) map is defined through expensive-to-solve parametric partial differential equations (PDEs). We consider a delayed-acceptance geometric MCMC method driven by a neural operator surrogate of the PtO map, where the proposal exploits fast surrogate predictions of the log-likelihood and, simultaneously, its gradient and Hessian. To achieve a substantial speedup, the surrogate must accurately approximate the PtO map and its Jacobian, which often demands a prohibitively large number of PtO map samples via conventional operator learning methods. In this work, we present an extension of derivative-informed operator learning [O'Leary-Roseberry et al., \textit{J.\ Comput.\ Phys.}, 496 (2024)] that uses joint samples of the PtO map and its Jacobian. This leads to derivative-informed neural operator (DINO) surrogates that accurately predict the observables and posterior local geometry at a significantly lower training cost than conventional methods. Cost and error analysis for reduced basis DINO surrogates are provided. Numerical studies demonstrate that DINO-driven MCMC generates effective posterior samples 3--9 times faster than geometric MCMC and 60--97 times faster than prior geometry-based MCMC. Furthermore, the training cost of DINO surrogates breaks even compared to geometric MCMC after just 10--25 effective posterior samples.

\end{abstract}
\begin{keywords}
  Inverse problem, scientific machine learning, uncertainty quantification, MCMC, neural operator
\end{keywords}
\section{Introduction}

\subsection{Quality--cost trade-off in MCMC for Bayesian inverse problems}
Continuum models of physical systems arising in scientific and engineering problems, such as those governed by partial differential equations (PDEs), often contain unspecified or uncertain parameters in the form of spatially and temporally varying functions. Given sparse and noisy observations of the system, infinite-dimensional inverse problems aim to infer the parameter function at which model predictions of observations (i.e., the \textit{observables}) best explain the observed data. A Bayesian formulation is often adopted to rigorously account for various uncertainties in inverse problems, whose solutions are represented as probability distributions of parameters (i.e., the \textit{posterior}). Bayesian inverse problems (BIPs) are of great practical importance due to their ability to enhance the predictability and reliability of computational models for better design, control, and more general decision-making \citep{Biegler2010, Oden2017, kouri2018optimization, Ghattas2021, alexanderian2021inverse}.

Markov chain Monte Carlo (MCMC) methods based on the Metropolis--Hastings (MH) algorithm \citep{metropolis1953equation, hastings1970monte} construct Markov chains whose stationary distributions are the Bayesian posterior \citep{robert2004monte, roberts2004general}. These methods are regarded as the gold standard for rigorous solutions of nonlinear BIPs due to their algorithmic simplicity and asymptotic posterior sampling consistency. A fundamental challenge in designing an efficient MCMC method for BIPs is to optimize the balance between the quality of generated Markov chains and the associated computational cost. The quality of a Markov chain can be quantified by, e.g., its effective sample size and mixing time. The computational cost consists of proposal sampling and acceptance probability computation at each chain position; the latter often involves evaluating the nonlinear \textit{parameter-to-observable (PtO) map} via solving large-scale parametric PDEs.

To generate high-quality Markov chains for posterior sampling, an MCMC method must be capable of agile exploration of the posterior landscape. This often requires intelligent MCMC proposal designs using either (i) the local curvature of the posterior landscape or (ii) a surrogate PtO map. These two approaches represent two major but mostly distinct developments of MCMC methods for infinite-dimensional BIPs: (i) \textit{geometric MCMC} methods with proposals that adapt to posterior local geometry \citep{girolami2011riemann, martin2012stochastic, law2014proposals, bui2014solving, lan2016emulation, beskos2017geometric, lan2019adaptive} and (ii) \textit{delayed-acceptance (DA) MCMC} methods with proposals informed by the surrogate posterior \citep{christen2005markov, efendiev2006preconditioning, lykkegaard2023multilevel}. However, high quality comes at a high cost. For geometric MCMC, exploiting posterior local geometry requires computing the \textit{Jacobian} of the PtO map through solutions of sensitivity or adjoint problems of the PDEs at each chain position. In practice, the accumulated cost of these linear solves often overwhelms their benefit. Moreover, for DA MCMC, constructing a data-driven surrogate of the nonlinear PtO map with infinite-dimensional parameter space may necessitate a large number of offline (prior to MCMC) model solutions to achieve a substantial online (during MCMC) acceleration.

\begin{remark}
    Here, we use Jacobian as a generic term that refers to the derivative of the observable vector with respect to the parameter function. However, there are multiple definitions of differentiability for nonlinear mappings in function spaces. We provide precise names and definitions for the derivatives and Jacobians of the PtO map in \cref{subsec:stochastic_derivative,subsec:hs_matrix}.
\end{remark}

\subsection{Geometric MCMC driven by derivative-informed neural operator}

We consider the use of neural operator surrogates of the PtO map to design MCMC methods based on the following observation: an ideal quality--cost trade-off in an MCMC method can be achieved by a surrogate PtO map that is fast and accurate in predicting both the observables (for DA MCMC) and the posterior local geometry (for geometric MCMC). Neural operators \citep{kovachki2023nueral, kovachki2024operator}, i.e., nonlinear mappings between function spaces constructed using neural networks, have the potential to provide a good approximation of the PtO map in infinite-dimensional BIPs. Notably, the Jacobian of a neural operator can be extracted through automatic differentiation at a low cost.

On the other hand, conventional operator learning methods using input--output samples of the target mapping (i.e., supervised learning) do not enforce direct control of the Jacobian approximation error; thus, the training cost for achieving a small Jacobian approximation error can be prohibitively high for large-scale PDE-constrained target mappings. As a result, neural operator surrogates often struggle to accelerate gradient-based optimization in high or infinite dimensions, where the surrogate-predicted gradient of the optimization objective function substitutes the model-predicted gradient; see, e.g., \citet[Section 4.2.2]{luo2023efficient}. Similarly, we expect the neural operator surrogate constructed via conventional operator learning to struggle in accelerating geometric MCMC.

In this work, we propose an efficient geometric MCMC method leveraging derivative-informed neural operator (DINO, \citealt{oleary2024derivative}). Compared to conventional operator learning with input--output error control, derivative-informed operator learning additionally enforces Jacobian error control. In the setting of BIPs, we generate samples of the PtO map and its Jacobian to train a DINO PtO map surrogate during the offline phase. This surrogate can achieve significantly higher accuracy in predicting both the observables and the Jacobian at a similar training cost as conventional operator learning methods. During the online phase, we deploy the trained DINO surrogate in a delayed-acceptance geometric MCMC method, where both the DINO prediction of posterior local geometry and the DINO PtO map contribute to generating high-quality Markov chains for posterior sampling at a considerably lower cost than conventional geometric MCMC methods.

We provide rigorous comparisons of our method with various baseline MCMC methods for solving challenging BIPs, such as coefficient inversion for a nonlinear diffusion--reaction PDE and inference of a heterogeneous hyperelastic material property. Our numerical results show that
\begin{enumerate}[leftmargin=*]
\item DINO-driven MCMC generates effective posterior samples 60--97 times faster than MCMC based on prior geometry (precondition Crank--Nicolson, \citealt{cotter2013mcmc}), 3--9 times faster than geometric MCMC, and 18--40 times faster than using a conventionally trained neural operator surrogate. When accounting for training sample generation cost, the training cost of DINO surrogates breaks even after collecting just 10--25 effective posterior samples compared to geometric MCMC.

\item Derivative-informed operator learning achieves a surrogate accuracy of the the PtO map Jacobian similar to the surrogate accuracy of the PtO map itself. This is achieved at 16--25 times lower cost in training sample generation than the conventional operator learning method. In our nonlinear diffusion--reaction numerical examples, we observe an estimated 166 times difference in training sample generation cost between the two operator learning methods to achieve an acceleration of geometric MCMC measured by the speed of effective posterior sample generation.

\end{enumerate}

\subsection{Literature review}\label{subsec:literature_review}
In this subsection, we cover literature relevant to our work on neural operator surrogates and MCMC for BIPs.

\subsubsection{Neural operator surrogates}

Constructing neural operator surrogates involves (i) using neural networks to design an architecture that maps between function spaces and allows for sufficient expressivity and (ii) approximating a target mapping by training the neural networks via supervised or semi-supervised learning. The key feature of a neural operator is that its architecture and learning scheme are independent of any particular discretization of the input and output space. We use the term neural operator when input or output belongs to a function space, as it necessitates a neural operator architecture and learning scheme. In this subsection, we briefly overview concepts related to neural operator surrogates that are relevant to our work.

\begin{itemize}[leftmargin=0pt, label={}]
    \item \textit{Neural operator architectures.} The architecture most relevant to this work is reduced basis neural operators that use neural networks to learn the finite-dimensional nonlinear mapping between coefficients of input and output reduced bases. Choices of reduced bases include but are not limited to (i) proper orthogonal decomposition (POD-NN, \citealt{Hesthaven2018}) or principal component analysis (PCA-Net, \citealt{bhattacharya2021model}), (ii) active subspace or derivative-informed subspace (DIP-Net, \citealt{oleary2022learning, oleary2022derivative}), (iii) learned neural network representation of output reduced bases (DeepONet, \citealt{LuJinKarniadakis2019}), and (iv) variational autoencoders (VANO, \citealt{seidman2023variational}). Other architectures include the Fourier neural operator (FNO, \citealt{li2021fourier}) and its variants \citep{cao2023lno, lanthaler2023nonlocal, li2020neural, li2020multipole}. See empirical comparisons of neural operator architectures for learning solution operators of PDEs by \citet{deHoop2022cost, lu2022comprehensive}.

    \item \textit{Operator learning objective.} The operator learning objective function is typically designed to control approximate error in the Bochner norm of nonlinear mappings between function spaces \citep{kovachki2023nueral}. When the objective function is approximated using samples, it leads to a loss function for empirical risk minimization. There are efforts to enhance the loss function using spatial evaluations of strong-form PDE residual, notably for FNO (PINO, \citealt{li2024physics}) and DeepONet (PI-DeepONet, \citealt{wang2021learning}) architectures. The focus of this work is DINO \citep{oleary2024derivative}. Its operator learning objective function is designed to control approximation error in high-dimensional Sobolev spaces of nonlinear mappings. This objective leads to neural operator surrogates that are accurate in both input--output and Jacobian evaluations. DINOs have been successfully deployed in optimization under uncertainty \citep{luo2023efficient} and optimal experimental design \citep{go2023accelerating}, demonstrating notable improvements compared to conventional operator learning methods.
    
    \item \textit{Jacobian vs.\ spatial derivative.} We emphasize the distinction between the Jacobian of the neural operator and the spatial derivative of neural operator output. Evaluating the Jacobian requires differentiating the nonlinear mappings from a parameter to the PDE solution. It often requires repeatedly solving direct or adjoint sensitivity problems with different right-hand side vectors \citep{Ghattas2021}. Controlling approximation error in the Jacobian is challenging in the operator learning setting. On the other hand, controlling approximation error in the spatial derivatives of spatially varying output functions can be implemented straightforwardly using a Sobolev norm over the spatial domain (e.g., $H^1(\Omega)$ norm where $\Omega$ is a spatial domain) for output error measure.
\end{itemize}

\subsubsection{MCMC for Bayesian inverse problems}
We briefly overview three aspects of MCMC methods for infinite-dimensional BIPs relevant to this work. Technical descriptions of some of these methods can be found in \cref{sec:preliminary}.
\begin{itemize}[leftmargin=0pt, label={}]
    \item \textit{Scalability.} During computation, the posterior is approximated on a discretized finite-dimensional subspace of the parameter function space via, e.g., the Galerkin method. Many popular MCMC methods that target the discretized posterior, such as random walk Metropolis, suffer from a deterioration in sampling performance as the discretization dimension increases. A class of \textit{dimension-independent MCMC methods} has emerged \citep{cotter2013mcmc,hairer2014spectral, law2014proposals, cui2016dimension, bui2016fem, rudolf2018generalized} that seek first to design MCMC methods that are well-posed on function spaces and then discretize for computation.

    \item \textit{Exploiting posterior geometry.} A class of \textit{geometric MCMC} methods gained attention in the last decade due to their information geometric approaches to proposal design. Relevant developments in this area include but are not limited to (i) proposals employing fixed or averaged posterior geometry, such as likelihood-informed subspace \citep{cui2016dimension}, active subspace \citep{constantine2016accelerating}, variational and Laplace approximation \citep{pinski2015algroithms, rudolf2018generalized, Petra2014, kim2023hippylibmuq}, and adaptive dimension reduction \citep{lan2019adaptive}; (ii) proposals that adapt to posterior local geometry, such as Riemannian manifold MCMC using the Jacobian and Hessian of the PtO map \citep{girolami2011riemann, bui2014solving}, stochastic Newton MCMC using the Jacobian and the low-rank approximation of the Hessian \citep{martin2012stochastic}, Gaussian process emulation of geometric quantities \citep{lan2016emulation}, and dimension-independent geometric MCMC using the Jacobian, i.e., using a simplified manifold \citep{law2014proposals, beskos2017geometric, lan2019adaptive}. Other notable developments include proposal designs using local likelihood approximations \citep{patrick2016accelerating} and transport maps constructed by low-fidelity model solutions \citep{peherstorfer2019transport}.
    \item \textit{Multifidelity acceleration.} 
        Cheap-to-evaluate low-fidelity models can help alleviate the cost of MCMC due to high-fidelity model solutions via \textit{delayed acceptance (DA) MCMC} \citep{christen2005markov, efendiev2006preconditioning, lykkegaard2023multilevel}. It uses a proposal given by the Markov chain transition rule of an MCMC method targeting the surrogate posterior, leading to a two-stage \citep{christen2005markov} or multi-stage \citep{lykkegaard2023multilevel} procedure for single or multiple surrogate models. \textit{Multilevel MCMC} methods \citep{hoang2013complexity, latz2018multilevel, dodwell2019multilevel, cui2024multilevel} employ a hierarchy of discretizations of a PDE model to reduce the overall computational cost of MCMC.
\end{itemize}

\subsection{Contributions}\label{subsec:contribution}
The main contributions of this work are summarized as follows.

\begin{itemize}[leftmargin=0pt, label={}]
    \item \textit{Formulation and analysis of DINO.} We present an extended formulation of the derivative-informed operator learning proposed by \citet{oleary2024derivative}. In particular, we establish suitable function space settings, i.e., the $H^1_{\mu}$ Sobolev space with Gaussian measure (\cref{subsec:h1_definition}), for derivative-informed operator learning that can be extended beyond the confines of inverse problems and particular choices of neural operator architecture. We also provide (i) a cost analysis for training data generation based on PDE models (\cref{subsec:data_generation}) and (ii) theoretical results on the neural operator approximation error of reduced basis DINO surrogate based on a Poincar\'e inequality for nonlinear mappings between function spaces (\cref{subsec:error}). 
    \item \textit{Efficient DINO-driven geometric MCMC.} We propose an efficient MCMC method (\cref{alg:mcmc}) for infinite-dimensional BIPs via a synthesis of ideas from reduced basis DINO surrogates, DA MCMC, and dimension-independent geometric MCMC. The method employs a proposal that adapts to DINO-predicted posterior local geometry within a delayed acceptance procedure. Compared to conventional geometric MCMC, our method leads to significant cost reduction due to (i) no online forward or adjoint sensitivity solves, (ii) fewer online PDE solves necessary for posterior consistency, and (iii) reduced need for prior sampling. At the same time, our numerical examples show that the method produces high-quality Markov chains typical of a geometric MCMC method, leading to substantial speedups in posterior sampling.
    \item \textit{Detailed numerical studies\footnote{Public release of code and data is contingent on acceptance for journal publication.}.} We provide detailed numerical studies of our methods and other baseline MCMC methods using two infinite-dimensional BIPs: coefficient inversion for a nonlinear diffusion-reaction PDE and inference of a heterogeneous hyperelastic material property.
\end{itemize}

\subsection{Layout of the paper}\label{subsec:layout}

The layout of the paper is as follows. In \cref{sec:preliminary}, we introduce concepts in BIPs and MCMC, including precision definitions of differentiability, posterior local approximation, dimension-independent geometric MCMC, and delayed acceptance MCMC. In \cref{sec:operator_learning}, we present a derivative-informed operator learning method with error control in the $H^1_{\mu}$ Sobolev space with Gaussian measure. In \cref{sec:dino}, we formulate derivative-informed training of reduced basis DINO, discuss its computational cost, and provide error analysis for different choices of reduced bases. In \cref{sec:surrogate_mcmc}, we detail the process of generating proposals that adapt to posterior local geometry with a trained neural operator surrogate and the resulting MCMC acceptance probability computation. In \cref{sec:results_set_up}, we explain the setup for our numerical examples, including an extensive list of baseline and reference MCMC methods, Markov chain diagnostics, efficiency metrics, and software.  In \cref{sec:ndr,sec:hyperelastic}, we showcase and analyze numerical results.
\section{Preliminaries: Bayesian inverse problems and MCMC}\label{sec:preliminary}

\subsection{Notations}
\begin{itemize}[leftmargin=*]
    \item $\langle\cdot,\cdot\rangle_{\mathscr{X}}$ denotes the inner-product on a Hilbert space $\mathscr{X}$ and $\norm{\cdot}_{\mathscr{X}}$ denotes the inner-product induced norm. The subscript is omitted when $\mathscr{X}$ is an Euclidean space.
    \item $\langle x_1, x_2\rangle_{\mathcal{T}}\coloneqq \langle \mathcal{T}^{1/2}x_1, \mathcal{T}^{1/2}x_2\rangle_{\mathscr{X}}$ and $\norm{x}_{\mathcal{T}}\coloneqq\sqrt{\langle\mathcal{T}^{1/2}x, \mathcal{T}^{1/2}x\rangle_{\mathscr{X}}}$ denote the inner-product and norm weighted by a positive and self-adjoint operator $\mathcal{T}:\scrX\to\scrX$ and $\mathcal{T}^{1/2}$ denotes the square root of $\mathcal{T}$. The square root is not explicitly required during computation.
    \item $B(\mathscr{X}_1, \mathscr{X}_2)$ denotes the Banach space of bounded and linear operators between two Hilbert spaces $\mathscr{X}_1$ and $\mathscr{X}_2$ equipped with the operator norm. We use $B(\mathscr{X})$ for $B(\mathscr{X}, \mathscr{X})$.
    \item $\HS(\scrX_1,\scrX_2)\subseteq B(\mathscr{X}_1, \mathscr{X}_2)$ denotes the Banach space of Hilbert--Schmidt operators. $\norm{\cdot}_{\HS}$ denotes the Hilbert--Schmidt norm. We use $\HS(\scrX)$ for $\HS(\scrX,\scrX)$.
    \item $B^{+}_1(\mathscr{X})\subseteq \HS (\mathscr{X})$ denotes the set of positive, self-adjoint, and trace class operators on a Hilbert space $\mathscr{X}$.
    \item $(\mathscr{X},\mathcal{B}(\mathscr{X}))$ denotes a measurable space with $\mathcal{B}(\cdot)$ being the Borel $\sigma$-algebra generated by open sets. $\scrP(\scrX)$ denotes the set of probability measures defined on $(\mathscr{X},\mathcal{B}(\mathscr{X}))$.
    \item $\nu(\dd x)$ denotes a measure on $(\mathscr{X},\mathcal{B}(\mathscr{X}))$ in the sense that $\nu(\scrA) = \int_{\scrA} \nu(\dd x)$, where $\scrA\in\mathcal{B}(\mathscr{X})$ and $x$ is a dummy variable for integration. The expression $\nu_1(\meas x)/\nu_2(\meas x)$ denotes the Radon--Nikodym derivative between two measures $\nu_1,\nu_2\in\scrP(\scrX)$ at $x\in\scrX$.
    \item We use capital letters to denote random variables, i.e., $X\sim \nu\in\scrP(\scrX)$. Both matrices and random vectors are denoted using bold and capitalized letters; they can be distinguished based on the context.
\end{itemize}

\subsection{Nonlinear Bayesian inverse problem}

Let $\mathscr{M}$ be a separable Hilbert space. We refer to $\mathscr{M}$ as the \textit{parameter space}. Let $\bdmc{G}: \scrM\to \R^{d_y}$ be a nonlinear \textit{parameter-to-observable (PtO) map} that represents model predictions of the \textit{observable vector}. We refer to the $\R^{d_y}$, $d_y\in\N$, as the \textit{observable space}. Let $\by\in\R^{d_y}$ denote a set of observed data. We assume that $\by$ is given by a model-predicted observable vector at unknown parameter $m\in\scrM$ corrupted by unknown additive noise $\bn\in\R^{d_y}$:
\begin{align}\label{eq:data_model}
    \by = \bdmc{G}(m) + \bn\,,\quad \bn\iid \pi_{n}\,, && (\text{Data model})
\end{align}
where $\pi_{n}\in\scrP(\R^{d_y})$ is the noise probability density\footnote{For finite-dimensional distributions, we assume their probability densities exist and do not distinguish between densities and measures.}. The inverse problem is to recover $m$ given data $\by$. 

Under the Bayesian approach to inverse problems, we assume prior knowledge of the parameter represented by a \textit{prior distribution} $\mu\in\scrP(\scrM)$. We are interested in characterizing the \textit{posterior distributions} $\mu^{\by}\in\scrP(\scrM)$ representing our updated knowledge of the parameter after acquiring data $\by$. The posterior is defined by Bayes' rule using a Radon--Nikodym (RN) derivative:
\begin{align}\label{eq:bayes_rule}
    \frac{\mu^{\by}(\meas m)}{\mu(\meas m)} &\propto \pi_n(\by-\bdmc{G}(m))\,. && (\text{Bayes' rule})
\end{align}
We adopt the following assumptions, often employed for infinite-dimensional BIPs, e.g., when $\mathscr{M}$ consists of spatially or temporally varying functions.
\begin{assumption}[Gaussian noise]\label{ass:gauss_noise}
The noise distribution is given by $\pi_n\coloneqq\calN(\boldsymbol{0}, \bC_n)$ with covariance matrix $\bC_n\in B_1^{+}(\R^{d_y})$. This leads to the following form of the Bayes' rule:
    \begin{equation}\label{eq:bayes_rule_gauss}
        \frac{\mu^{\by}(\meas m)}{\mu(\meas m)} \coloneqq \frac{1}{z(\by)}\exp\left(-\Phi^{\by}(m)\right)\,,\quad \Phi^{\by}(m) \coloneqq \frac{1}{2}\norm{\by-\bdmc{G}(m)}^2_{\bC_n^{-1}}\quad \mu\text{-a.e.}\,, 
    \end{equation}
    where $\Phi^{\by}:\scrM\to\R$ is the data misfit and $z(\by)\coloneqq\mathbb{E}_{M\sim\mu}\left[\Phi^{\by}(M)\right]$ is the normalization constant.
\end{assumption}
\begin{assumption}[Gaussian prior]\label{ass:gauss_prior}
    The prior distribution is given by $\mu\coloneqq\mathcal{N}(0, \cpr)$ with covariance operator $\cpr\in B_1^{+}(\scrM)$.
\end{assumption}

\begin{assumption}[Well-posedness, {\citealt[Corollary 4.4]{Stuart2010}}]\label{ass:well_posedness}
    The PtO map $\bdmc{G}$ is $\mu$-a.e.\ well-defined, sufficiently bounded, and locally Lipshitz continuous, which implies that the Bayesian inversion is well-posed.
\end{assumption}
\subsection{The Cameron--Martin space and differentiability}\label{subsec:stochastic_derivative}
We consider two Hilbert spaces with prior and noise covariance inverse-weighted inner products on the observable and parameter space. These spaces are known as the Cameron--Martin (CM) space of $\mu$ and $\pi_n$.
\begin{align*}
    &\text{Parameter CM space } \scrH_{\mu}: \quad\left(\left\{m\in\scrM\;\big\vert\; \norm{m}_{\cpr^{-1}}<\infty\right\}, \left\langle\cdot,\cdot\right\rangle_{\cpr^{-1}}\right)\,,\\
    &\text{Observable CM space $\scrY$}: \qquad\qquad\left(\R^{d_y}, \left\langle\cdot,\cdot\right\rangle_{\bC_n^{-1}}\right)\,.
\end{align*}
We have the continuous embedding $\scrH_{\mu}\hookrightarrow\scrM$, i.e., there exists $c>0$ such that $\norm{m}_{\scrM}\leq c\norm{m}_{\cpr^{-1}}$ for all $m\in\scrH_{\mu}$. However, the converse is not true when $\scrM$ has infinite dimensions, e.g., $\norm{M}_{\cpr^{-1}}=\infty$ and $\norm{M}_{\scrM}< \infty$ a.s.\ for $M\sim\mu$. The observable CM space $\scrY$ is isomorphic to $\R^{d_y}$ under the identity map, yet the weighted inner product of $\scrY$ is often preferred in the context of inverse problems as it reflects our confidence in the observed data.

The CM space plays a major role in understanding the equivalence of measures due to the linear transformations of infinite-dimensional Gaussian random functions, and it will be applied extensively in this work. We refer to \citet[Section 2.7]{sullivan2015intro}, and \citet{bogachev1998gaussian}, and \citet{Stuart2010} for detailed references on Gaussian measures and the CM space.

In addition to \cref{ass:gauss_prior,ass:gauss_noise,ass:well_posedness}, we assume the directional differentiability of the PtO map along the parameter CM space $\scrH_{\mu}$.
\begin{assumption}[stochastic G\^ateaux differentiability, {{\citealt[5.2.3]{bogachev1998gaussian}}}]\label{ass:stochastic_derivative}
    There exists a mapping $D_{\scrH_{\mu}}\bdmc{G}:\scrM\to \HS(\scrH_{\mu}, \scrY)$ such that for any $\delta m\in \scrH_{\mu}$, 
    \begin{equation*}
        \lim_{t\to 0}\norm{t^{-1}\left(\bdmc{G}(m+t\delta m) - \bdmc{G}(m)\right) - D_{\scrH_{\mu}} \bdmc{G}(m)\delta m}_{\bC_n^{-1}} = 0\quad \mu\text{-a.e.} 
    \end{equation*}
    The mapping $D_{\scrH_{\mu}} \bdmc{G}$ is called the stochastic derivative of $\bdmc{G}$.
\end{assumption}
The stochastic G\^ateaux differentiability in \cref{ass:stochastic_derivative} is weaker than the typical G\^ateaux differentiability assumption that requires directional differentiability along the whole parameter space $\scrM$ given in \cref{def:gateaux_derivative}.
\begin{definition}[$\mu$-a.e.\ G\^ateaux differentiability]\label[definition]{def:gateaux_derivative}
    There exists a mapping $D\bdmc{G}:\scrM\to\HS(\scrM,\R^{d_y})$ such that for any $\delta m\in \scrM$, 
    \begin{equation*}
        \lim_{t\to 0}\norm{t^{-1}\left(\bdmc{G}(m+t\delta m) - \bdmc{G}(m)\right) - D\bdmc{G}(m)\delta m} = 0\quad \mu\text{-a.e.}
    \end{equation*}
    The mapping $D\bdmc{G}$ is called the G\^ateaux derivative of $\bdmc{G}$.
\end{definition}
Suppose $D\bdmc{G}$ exists, then the stochastic derivative exists. We have $D\bdmc{G}(m)|_{\scrH_{\mu}} = \sder\bdmc{G}(m)$ $\mu$-a.e. Additionally, the stochastic derivative carries over the parameter regularity given by the prior distribution $\mu$, making it the more natural derivative definition for BIPs; see \cref{sec:stochastic_derivative}. Stochastic differentiability is used our derivative-informed operator learning formulation; see \cref{subsec:h1_definition}. The stochastic derivative is often called the Malliavin derivative \citep{nualart2018introduction} or the H-derivative \citep{kac2002quantum} in different contexts.

\subsection{Local Gaussian approximation of the posterior}
 
We consider a linear expansion of the nonlinear PtO map $\bdmc{G}$ at a given $m\in\scrM$:
\begin{equation*}
    \bdmc{G}(\cdot)\approx \bdmc{G}(m) + \sder \bdmc{G}(m)(\cdot -m)\,.
\end{equation*}
Replacing the PtO map in \cref{eq:bayes_rule_gauss} using the linear expansion, we obtain a local Gaussian approximation to the posterior in closed form~\cite[Section 6.4]{Stuart2010}. It is a conditional probability distribution $\mathcal{Q}_{\text{local}}:\scrM\times\calB(\scrM) \to[0,1]$ given by:
\begin{align}\label{eq:local_gaussian}
        \mu^{\by}\approx\mathcal{Q}_{\text{local}}(m,\cdot)=\mathcal{N}(-\sder\Phi^{\by}(m), \cpo(m))\,,
\end{align}
where the negative mean $\sder\Phi^{\by}:\scrM\to\scrH_{\mu}$ is the $\scrH_{\mu}$-Riesz representation of the stochastic derivative of the data misfit,
\begin{equation}\label{eq:ppg}
    \sder\Phi^{\by}(m) \coloneqq \sder\bdmc{G}(m)^*\left(\bdmc{G}(m)-\by\right)\,,
\end{equation}
and the covariance $\cpo:\scrM\to B_1^+(\scrM)$ is given by
\begin{align}
    \mathcal{C}_{\text{post}}(m) \coloneqq (\mathcal{I}_{\scrH_{\mu}} + \calH(m))^{-1}\cpr\,,\quad\calH(m)\coloneqq \sder \bdmc{G}(m)^*\sder \bdmc{G}(m)\,,\label{eq:ppgh}
\end{align}
where $\mathcal{I}_{\scrH_{\mu}}$ is the identity map on $\scrH_{\mu}$, and $\calH:\scrM\to B_1^{+}(\scrH_{\mu})$ is a Gauss--Newton approximation to the stochastic Hessian of the data misfit in \cref{eq:stochastic_hessian}. The mappings $\sder\Phi^{\by}$ and $\calH$ are often known as the \textit{prior-preconditioned gradient (ppg)} and the \textit{prior-preconditioned Gauss-Newton Hessian (ppGNH)} in the context of inverse problems. We demonstrate their connections to their conventional definitions in \cref{appendix:gaussnewtonhessian}.

\subsection{The Metropolis--Hastings algorithm}\label{subsec:mh}
We use the \textit{Metropolis--Hastings (MH) algorithm} to sample from the posterior distribution $\mu^{\by}$ defined over infinite-dimensional $\mathscr{M}$. The MH algorithm is a procedure for generating reversible Markov chains $\{m_j\in\scrM\}_{j=1}^{\infty}$ with a stationary distribution of $\mu^{\by}$. The algorithm prescribes a Markov chain transition rule (i.e., the law of $m_j\mapsto m_{j+1}$) based a proposal distribution and an accept-reject move. The proposal, denoted by $\mathcal{Q}:\mathscr{M}\times\mathcal{B}(\mathscr{M})\to[0,1]$, is a conditional probability distribution. The proposal and the posterior jointly define a set of transition rates\footnote{These transition rates should not be confused with the Markov chain transition rule of the MH algorithm.} between two positions in $\scrM$ as measures (unnormalized) on the product space $\scrM\times\scrM$:
\begin{equation*}
    \nu(\dd m, \dd m^{\dagger}) \coloneqq \mathcal{Q}(m, \dd m^{\dagger})\mu^{\by}(\dd m)\,,\quad \nu^T(\dd m, \dd m^{\dagger}) \coloneqq \mathcal{Q}(m^{\dagger}, \dd m)\mu^{\by}(\dd m^{\dagger})\,.
\end{equation*}
The ratio of transition rates, $\rho:\mathscr{M}\times\mathscr{M}\to[0, \infty)$, is given by an RN derivative:
\begin{equation}\label{eq:ratio_of_transition_rate}
    \rho(m, m^{\dagger}) \coloneqq \cfrac{\nu^T(\meas m, \meas m^{\dagger})}{\nu(\meas m, \meas m^{\dagger})}=\cfrac{\mathcal{Q}(m^{\dagger}, \dd m)\exp(-\Phi^{\by}(m^{\dagger}))\mu(\dd m^{\dagger})}{\mathcal{Q}(m, \dd m^{\dagger})\exp(-\Phi^{\by}(m))\mu(\dd m)}\,,
\end{equation}
where we use a change of measure in~\eqref{eq:bayes_rule} to represent $\mu^{\by}$. 

The accept-reject move is executed as follows. At a chain position $m_j\in\scrM$, we sample a proposed move $m^{\dagger}\iid\mathcal{Q}(m_j,\cdot)$. The next position is set to $m_{j+1} = m^{\dagger}$, i.e., acceptance, with probability $\alpha(m_j,m^{\dagger})\in[0,1]$ given by
\begin{equation*}
    \alpha(m_j, m^{\dagger}) \coloneqq \min\{1, \rho(m_j, m^{\dagger})\}\,.
\end{equation*}
Alternatively, we set $m_{j+1} = m_j$, i.e., rejection, with probability $1-\alpha(m_{j}, m^{\dagger})$. See \cref{fig:mh} (\textit{left}) for a schematic of the MH algorithm.
\begin{figure}[!h]
    \centering
    \addtolength{\tabcolsep}{-5pt}
    \begin{tabular}{c c}
       \bf The Metropolis--Hastings algorithm  & \bf \makecell{The Metropolis--Hastings algorithm\\ with delayed acceptance}  \\
       \includegraphics[align=t, width = 0.465\linewidth]{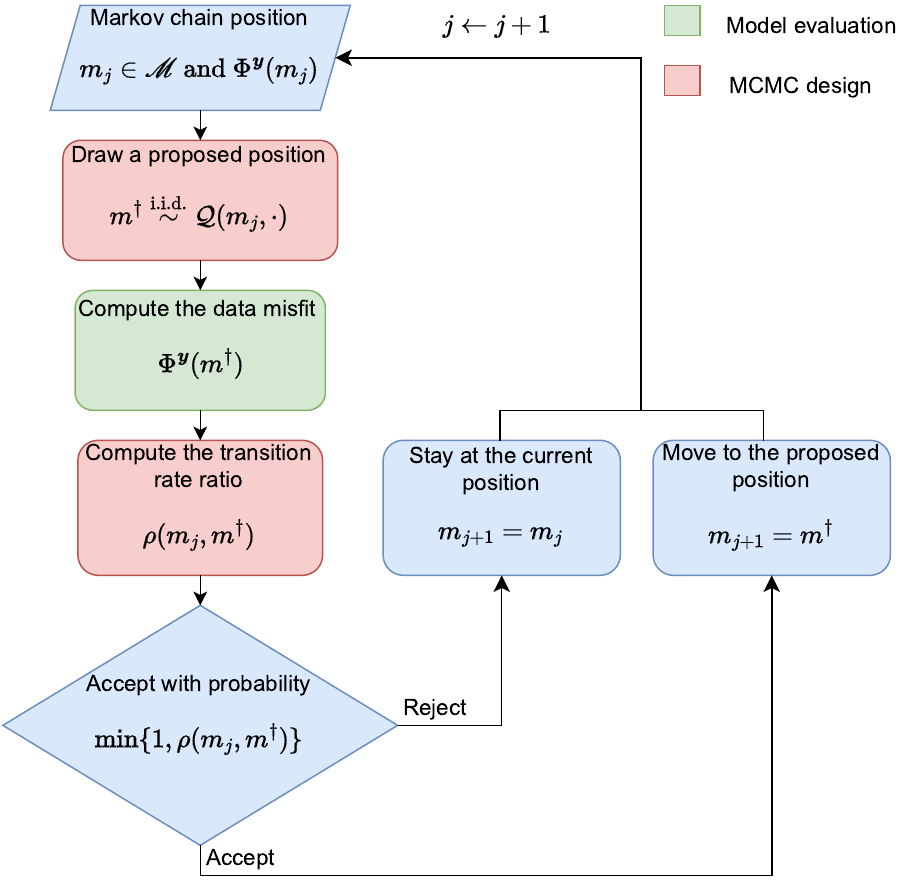}  & \includegraphics[align=t, width=0.525\linewidth]{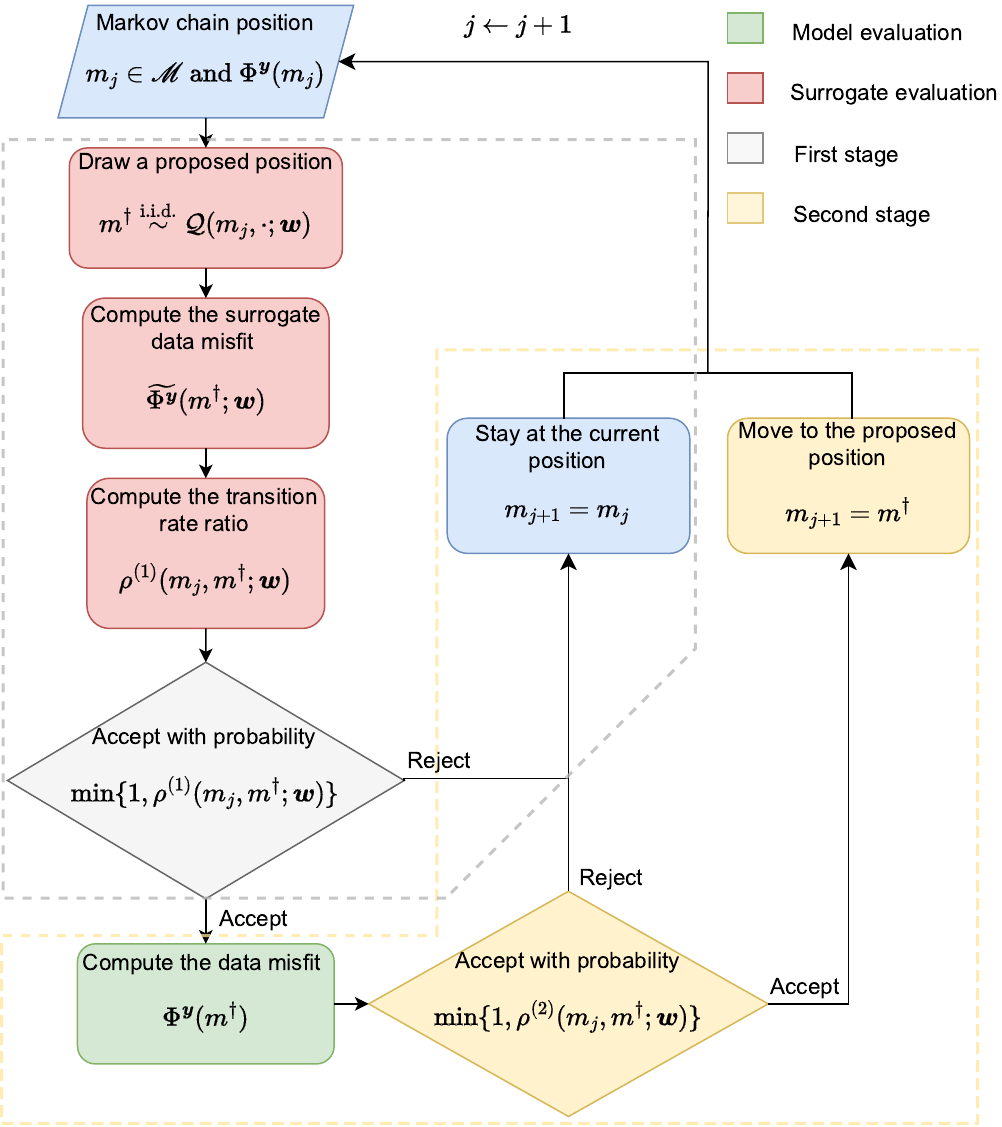}
    \end{tabular}
    \addtolength{\tabcolsep}{5pt}
    \caption{(\textit{left}) A schematic of the MH algorithm for sampling from the posterior distribution $\mu^{\by}$ as described in \cref{subsec:mh}. (\textit{right}) A schematic of the MH algorithm with delayed acceptance enabled by a surrogate PtO map $\widetilde{\bdmc{G}}(\cdot;\bw)$ parameterized by $\bw$. See \cref{subsec:da} for a detailed description of the components of this algorithm.}
    \label{fig:mh}
\end{figure}

\subsection{Dimension-independent MCMC}\label{subsec:pcn}


A dimension-independent MCMC method using the MH algorithm employs a proposal with a well-defined transition rate ratio $\rho$ in~\eqref{eq:ratio_of_transition_rate}. According to the RN theorem \cite[Theorem 2.29]{sullivan2015intro}, $\rho$ is well-defined if $\nu^T\ll\nu$, where $\ll$ denotes the absolute continuity of measures. In particular, \citet{tierney1998note} shows that the MH algorithm rejects all proposed moves if $\nu^T$ and $\nu$ are mutually singular. In finite dimensions, $\nu^T\ll\nu$ holds for most proposal choices (e.g., Gaussian random walk), and one typically expresses $\nu$ and $\nu^T$ as probability densities. However, measures on infinite-dimensional spaces tend to be mutually singular (see, e.g., \citealt[Theorem 2.51]{sullivan2015intro}), and their probability densities do not exist \cite[Theorem 2.38]{sullivan2015intro}. As a result, a finite-dimensional MH algorithm targeting a discretized infinite-dimensional sampling problem often leads to deteriorating sampling performance when the discretization is refined \citep{hairer2014spectral}. For example, the conventional Gaussian random walk proposal given by $\mathcal{Q}_{\text{RW}}(m,\cdot)\coloneqq\mathcal{N}(m, s\cpr)$, $s>0$, fails to be dimension-independent because it leads to an ill-defined $\rho$ in \cref{eq:ratio_of_transition_rate}; see \citealt[Example 5.3]{Stuart2010}, \citealt[Section 2.4]{hairer2014spectral}, and \citealt[Section 3.3]{rudolf2018generalized}.

The building block for a dimension-independent MCMC is the preconditioned Crank--Nicolson (pCN, \citealt{cotter2013mcmc}) proposal, which is reversible with respect to the prior:
\begin{subequations}
\begin{align}
    &\mathcal{Q}_{\text{pCN}}(m,\cdot):=\mathcal{N}(sm, (1-s^2)\calC_{\text{pr}})\,,\quad s\leq1\,, &&& (\text{The pCN proposal})\label{eq:pCN_proposal}\\
    &\mathcal{Q}_{\text{pCN}}(m^{\dagger},\meas m)\mu(\meas m^{\dagger}) = \mathcal{Q}_{\text{pCN}}(m,\meas m^{\dagger})\mu(\meas m)\,. &&& (\text{Prior reversibility})\label{eq:equivalence_pCN}
\end{align}
\end{subequations}
The equivalence of measure in \cref{eq:equivalence_pCN} leads to a well-defined transition rate ratio of the form:
\begin{equation*}
    \rho_{\text{pCN}}(m, m^{\dagger}) = \exp\left(\Phi^{\by}(m)-\Phi^{\by}(m^{\dagger})\right)\,.
\end{equation*}
The pCN proposal is used for deriving the acceptance probability for proposals that include posterior local geometry \citep{beskos2017geometric, rudolf2018generalized, lan2019adaptive}. In particular, these geometry-informed proposals are designed to possess well-defined and close-formed RN derivatives with respect to the pCN proposal. The existence of these RN derivatives leads to dimension-independent geometric MCMC methods. The closed forms of these RN derivatives make evaluations of the acceptance probability straightforward.

\subsection{Geometric MCMC}\label{subsec:mMALA_intro}

We consider the simplified manifold Metropolis-adjusted Langevin algorithm, or mMALA, introduced by \citealt{beskos2017geometric}. It originates from the following Langevin stochastic differential equation (SDE) on $\mathscr{M}$, preconditioned by a position-dependent, positive, and self-adjoint trace class operator $\mathcal{K}:\scrM\to B_1^{+}(\scrM)$:
\begin{equation}\label{eq:langevin}
    \dd M_t = -\frac{1}{2} \mathcal{K}(M_t)\cpr^{-1}\left( M_t + \sder\Phi^{\by}(M_t)\right)\dd t + \mathcal{K}(M_t)^{1/2}\dd W_t\,,\quad t\geq 0\,,
\end{equation}
where $W_t$ is a cylindrical Wiener process on $\mathscr{M}$. This SDE can be derived using a local reference measure of the form:
\begin{equation}\label{eq:local_reference}
    \calQ_{\text{local}}(m,\cdot) = \mathcal{N}(\mathcal{M}(m), \mathcal{K}(m))\,,
\end{equation}
where $\calM:\scrM\to\scrH_{\mu}$ outputs the mean of the reference, which does not appear in the SDE \citep[Section 3.1]{beskos2017geometric}. We assume $\calK(m)$ leads to the equivalence of measure between $\mu$ and $\mathcal{N}(0, \calK(m))$ $\mu$-a.e. Discretizing the above SDE in time using a semi-implicit Euler scheme with a step size $\triangle t\in\R_+$ leads to the following proposal:
\begin{subequations}\label{eq:mmala}
\begin{gather}
    \calQ_{\text{mMALA}}(m,\cdot) \coloneqq \mathcal{N}\left(sm + \left(1-s\right)\mathcal{A}(m), \left(1-s^2\right)\calK(m)\right)\,,\\
    \mathcal{A}(m) \coloneqq m- \calK(m)\cpr^{-1}\left(m + \sder\Phi^{\by}(m)\right)\,,\quad s = \frac{4-\triangle t}{4+\triangle t}\,.
\end{gather}
\end{subequations}
The mMALA and pCN proposals are equivalent in measure $\mu$-a.e.\ \citep[Theorem 3.5]{beskos2017geometric}, and the mMALA proposal has a well-defined transition rate ratio $\rho_{\mmala}:\scrM\times\scrM\to\R_+$ with a closed form given by:
\begin{align*}
    \rho_{\mmala}(m_1, m_2) &\coloneqq \exp\left(\Phi^{\by}(m_1)-\Phi^{\by}(m_2)\right)\frac{\calQ_{\mmala}(m_2,\dd m_1)\mu(\dd m_2)}{\calQ_{\mmala}(m_1,\dd m_2)\mu(\dd m_1)}\\
    & = \exp\left(\Phi^{\by}(m_1)-\Phi^{\by}(m_2)\right)\frac{\rho_0(m_2, m_1)}{\rho_0(m_1, m_2)}\,,
\end{align*}
where $\rho_0:\mathscr{M}\times\mathscr{M}\to\R_+$ is the RN derivative between the mMALA and pCN proposals:
\begin{align}
    \rho_0(m_1, m_2) &\coloneqq \frac{ \calQ_{\mmala}(m_1,\dd m_2)}{\calQ_{\text{pCN}}(m_1,\dd m_2)}\nonumber\\
    & = \exp\left(-\frac{\triangle t}{8}\norm{\mathcal{A}(m_1)}_{\calK(m_1)^{-1}}^2 + \frac{\sqrt{\triangle t}}{2}\left\langle\mathcal{A}(m_1), \widehat{m}\right\rangle_{\calK(m_1)^{-1}}\right)\label{eq:mmala_trr}\\
    &\quad \times {\text{det}_{\scrM}\left(\cpr^{1/2}\calK(m_1)^{-1}\cpr^{1/2}\right)}^{1/2}\exp\left(-\frac{1}{2}\norm{\widehat{m}}^2_{\left(\calK(m_1)^{-1} - \cpr^{-1}\right)}\right)\,,\nonumber
\end{align}
where $\widehat{m} = (m_2 - sm_1)/\sqrt{1-s^2}$ and $\text{det}_\scrM: B_1^{+}(\scrM)\to\R$ is the operator determinant given by eigenvalue product \cite[Theorem 6.1]{gohberg2012traces}.

The mMALA proposal allows flexibility in designing $\calK(m)$ and removing or reducing the ppg $\sder\Phi^{\by}$. For example, the mMALA proposal reduces to the pCN proposal in \eqref{eq:pCN_proposal} when $\calK(m) = \cpr$ and the ppg is removed. We may choose $\mathcal{K}(m)$ such that the local reference measure adapts to the posterior local geometry at each Markov chain position, such as $\mathcal{K}(m) \approx \cpo(m)$ using the local Gaussian approximation in \cref{eq:local_gaussian}. The proposals that incorporate posterior local geometry lead to a class of MCMC methods referred to as \textit{dimension-independent geometric MCMC methods}. 

\subsection{Delayed acceptance MCMC}\label{subsec:da}
The DA MCMC method follows the MH algorithm targeting $\mu^{\by}$ with a special choice of the proposal distribution: the Markov chain transition rule (i.e., the law of $m_j\to m_{j+1}$) of the MH algorithm targeting the surrogate posterior $\widetilde{\mu^{\by}}\in\scrP(\scrM)$ defined by a surrogate PtO map $\widetilde{\bdmc{G}}\approx\bdmc{G}$. The surrogate data misfit and posterior are given by
\begin{subequations}
\begin{align}
    \widetilde{\Phi^{\by}}(m) &\coloneqq \frac{1}{2}\norm{\by-\widetilde{\bdmc{G}}(m)}^2_{\bC_n^{-1}}\,, &&& (\text{Surrogate data misfit}) \label{eq:surrogate_likelihood}\\
    \frac{\widetilde{\mu^{\by}}(\meas m)}{\mu(\meas m)}&\propto \exp(-\widetilde{\Phi^{\by}}(m))\,. &&& (\text{Surrogate posterior}) \label{eq:surrogate_posterior}
\end{align}
\end{subequations}

The DA procedure has two stages at each chain position $m_j$ with a proposed move $m^{\dagger}\iid\mathcal{Q}(m,\cdot)$. 
\begin{enumerate}[leftmargin=*]
    \item A pass-reject move is performed based on the transition rate ratio in \eqref{eq:ratio_of_transition_rate} using surrogate data misfit evaluations:
\begin{equation*}
    \rho^{(1)}(m_j, m^{\dagger}) = \frac{\mathcal{Q}(m^{\dagger},\meas m_j)\exp(-\widetilde{\Phi^{\by}}(m^{\dagger}))\mu(\meas m^{\dagger})}{\mathcal{Q}(m_j,\meas m^{\dagger})\exp(-\widetilde{\Phi^{\by}}(m_j))\mu(\meas m_j)}\,.
\end{equation*}
The proposed move $m^{\dagger}$ is passed to the second stage with probability $\alpha^{(1)}(m_j, m^{\dagger})\coloneqq\min\{1, \rho^{(1)}(m_j, m^{\dagger})\}$. Otherwise, the second stage is skipped and set $m_{j+1} = m_j$ (i.e., rejection) with probability $1-\alpha^{(1)}(m_j, m^{\dagger})$.
    \item An accept-reject move of the MH algorithm is performed based on a proposal $\mathcal{Q}_{\text{DA}}(m, \meas m^{\dagger})$ prescribed by the first stage pass-reject move:
    \begin{equation*}
        \mathcal{Q}_{\text{DA}}(m,\meas m^{\dagger}) = \alpha^{(1)}(m_j, m^{\dagger})\mathcal{Q}(m, \meas m^{\dagger}) + \left(1-\alpha^{(1)}(m_j, m^{\dagger})\right)\delta_{m_j}(\meas m^{\dagger})\,,
    \end{equation*}
    where $\delta_{m_j}$ is the Dirac mass concentrated on $m_j$. Since $\mathcal{Q}_{\text{DA}}$ is reversible with respect to $\widetilde{\mu^{\by}}$, we have $\mathcal{Q}_{\text{DA}}(m,\meas m^{\dagger})\widetilde{\mu^{\by}}(\meas m) = \mathcal{Q}_{\text{DA}}(m^{\dagger},\meas m)\widetilde{\mu^{\by}}(\meas m^{\dagger})$.
    As a result, the transition rate ratio for $\mathcal{Q}_{\text{DA}}$ is given by
    \begin{equation}\label{eq:da_second}
    \rho^{(2)}(m_j, m^{\dagger}) = \frac{\exp(-\widetilde{\Phi^{\by}}(m_j))\exp(-\Phi^{\by}(m^{\dagger}))}{\exp(-\widetilde{\Phi^{\by}}(m^{\dagger}))\exp(-\Phi^{\by}(m_j))}\,.
\end{equation}
\end{enumerate}
See \cref{fig:mh} (\text{right}) for a schematic of DA MCMC.

The DA procedure allows proposed moves to be rejected in the first stage solely based on surrogate evaluations. This feature potentially leads to a significant reduction in the evaluation counts of the true PtO map during posterior sampling. On the other hand, the efficiency of DA MCMC relies heavily on the quality of the surrogate approximation. When the surrogate PtO map is accurate, most rejections occur during the first stage without model solutions, and most proposed moves passed to the second stage are accepted. Higher surrogate approximation error leads to more frequent second-stage rejection, thus increasing the average computational cost per Markov chain sample and deteriorating posterior sampling efficiency. See \cref{app:da_and_error} for additional discussion on surrogate approximation in DA MCMC.

\section{Operator learning in $H^1_\mu$ Sobolev space with Gaussian measure}\label{sec:operator_learning}

We consider an operator learning problem of optimizing the weight $\bw\in\R^{d_w}$ of an operator surrogate $\widetilde{\bdmc{G}}(\cdot;\bw):\mathscr{M}\to\scrY$ so that $\widetilde{\bdmc{G}}$ is close to the PtO map $\bdmc{G}$ measured by certain metric. When the operator surrogate $\widetilde{\bdmc{G}}$ is represented using a neural network, the weight $\bw$ consists of the tunable parameters of neural networks. In \cref{subsec:l2_training}, we introduce the conventional operator learning method. We present our derivative-informed operator learning method in \cref{subsec:h1_definition}. In \cref{subsec:hs_matrix}, we discuss matrix representations of Hilbert--Schmidt operators for efficient operator learning.

\subsection{Operator learning in $L^2_\mu$ Bochner space}\label{subsec:l2_training}

The typical operator learning method approximates the PtO map in the $L^2_{\mu}(\scrM;\scrY)$ Bochner space, or $L^2_{\mu}$ for short. It is defined by:
\begin{align*}
    L^2_{\mu}(\scrM;\mathscr{Y}) &\coloneqq \left\{\bdmc{T}:\scrM\to \mathscr{Y}\;\Big\vert\;\norm{\bdmc{T}}_{L^2_{\mu}}<\infty\right\}\,, &&& (L^2_{\mu}\text{ Definition})\\
    \norm{\bdmc{T}}_{L^2_{\mu}(\scrM;\mathscr{Y})} &\coloneqq \left(\mathbb{E}_{M\sim\mu}\left[\norm{\bdmc{T}(M)}_{\bC_n^{-1}}^2\right]\right)^{1/2}\,. &&& (L^2_{\mu}\text{ norm})
\end{align*}
The operator learning objective function is designed to control the approximation error in $L^2_{\mu}(\scrM;\scrY)$:
\begin{subequations}\label{eq:loss_l2}
\begin{align}
    &\bw^{\ddagger} = \argmin_{\bw\in\R^{d_w}} \calL^{\infty}_{L^2_\mu}(\bw)\,, &&&(\text{Operator learning objective})\\
    &\calL^{\infty}_{L^2_\mu}(\bw)\coloneqq\frac{1}{2}\norm{\bdmc{G}-\widetilde{\bdmc{G}}(\cdot;\bw)}_{L^2_{\mu}(\scrM;\mathscr{Y})}^2\,. &&& (\text{Error control in }L^2_{\mu}(\scrM;\scrY))
\end{align}
\end{subequations}
The objective $\mathcal{L}^{\infty}_{L^2_{\mu}}$ can be estimated via input--output pairs $\{m_j, \bdmc{G}(m_j)\}_{j=1}^{n_t}$ with $m_j\iid\mu$, which leads to a loss function $\calL^{n_t}_{L^2_{\mu}}$ defined as follows:
\begin{align}\label{eq:empirical_loss_l2}
    &\calL^{\infty}_{L^2_{\mu}}(\bw)\approx \calL^{n_t}_{L^2_{\mu}}(\bw;\{m_j\}_{j=1}^{n_t})\coloneqq \frac{1}{2n_t}\sum_{j=1}^{n_t}\norm{\bdmc{G}(m_j)-\widetilde{\bdmc{G}}(m_j;\bw)}^2_{\bC_n^{-1}}\,. &&
\end{align}
The operator surrogate can be constructed via finding $\bw^{\ddagger}$ that minimizes the loss $\calL^{n_t}_{L^2_{\mu}}$. 

\subsection{Operator learning in $H^1_{\mu}$ Sobolev space with Gaussian measure}\label{subsec:h1_definition}

In this work, we are interested in designing MCMC methods using operator surrogate that requires small approximation errors in both operator evaluations (for approximating ppg in \cref{eq:ppg} and efficient DA procedure) and its derivative evaluations (for approximating the ppg and ppGNH in \cref{eq:ppg,eq:ppgh}). Therefore, we consider controlling the operator surrogate error in the \textit{$H^1$ Sobolev space with Gaussian measure}, or $H^1_{\mu}$ for short. It is a Hilbert space of nonlinear mappings with an inner product-induced norm that measures the distance between nonlinear mappings using \textit{the discrepancy in their stochastic derivative evaluations} in addition to the discrepancy in the mappings:
\begin{align*}
    &H^1_{\mu}(\scrM;\mathscr{Y}) \coloneqq \left\{\bdmc{T}\in L^2_{\mu}(\scrM;\mathscr{Y})\;\Big\vert\;\norm{\sder\bdmc{T}(M)}_{L^2_{\mu}\left(\scrM; \HS(\scrH_{\mu}, \mathscr{Y})\right)}<\infty\right\}\,, &&& ( H^1_{\mu}\text{ definition})\\
    &\norm{\bdmc{T}}_{H^1_{\mu}(\scrM;\mathscr{Y})}\coloneqq \left( \norm{\bdmc{T}}_{L^2_{\mu}(\scrM;\mathscr{Y})}^2 + \norm{\sder\bdmc{T}}_{L^2_{\mu}\left(\scrM; \HS(\scrH_{\mu}, \mathscr{Y})\right)}^2 \right)^{1/2}\,,&&& (H^1_{\mu}\text{ norm})\\
    &\norm{\sder\bdmc{T}}_{L^2_{\mu}\left(\scrM; \HS(\scrH_{\mu}, \mathscr{Y})\right)} \coloneqq \left(\mathbb{E}_{M\sim\mu}\left[\norm{\sder\bdmc{T}(M)}^2_{\HS(\scrH_{\mu},\mathscr{Y})}\right]\right)^{1/2}\,. &&& (\text{Semi-norm})
\end{align*}
See \citet{bogachev1998gaussian} and references therein for a detailed discussion on the definition and properties of $H^1_{\mu}(\scrM;\scrY)$. The following logarithmic Sobolev (\Cref{thm:log-sobolev}) and Poincar\'e (\Cref{thm:poincare}) inequalities hold on $H^1_{\mu}(\scrM;\scrY)$, which are essential for establishing approximation error bounds on operator surrogate and Bayesian inversion. In particular, we have a Poincar\'e constant of $1$ on $H^1_{\mu}(\scrM;\scrY)$.
\begin{theorem}[Logarithmic Sobolev inequality, {\citealt[5.5.1]{bogachev1998gaussian}}]\label[theorem]{thm:log-sobolev}
If $\calS \in H^1_{\mu}(\scrM)\coloneqq H^1_{\mu}(\scrM;\R)$, then the following inequality holds
\begin{align*}
    \mathbb{E}_{M\sim\mu}\left[\calS(M)^2\ln\left(|\calS(M)|\right)\right]&\leq \mathbb{E}_{M\sim\mu}\left[\norm{\sder\calS(M)}^2_{\scrH_{\mu}}\right]\\
    &\quad + \frac{1}{2}\mathbb{E}_{M\sim\mu}\left[\calS(M)^2\right]\ln\left(\mathbb{E}_{M\sim\mu}\left[\ln\left(\calS(M)^2\right)\right]\right)\,.
\end{align*}
where $\sder\calS$ is the $\scrH_{\mu}$-Riesz representation of the stochastic derivative of $\calS$.
\end{theorem}
\begin{theorem}[Poincar\'e inequality, {\citealt[5.5.6]{bogachev1998gaussian}}]\label[theorem]{thm:poincare}
    If $\bdmc{T}\in H^1_{\mu}(\scrM;\scrY)$, then
    \begin{equation*}
        \norm{\bdmc{T} - \mathbb{E}_{M\sim\mu}[\bdmc{T}(M)]}^2_{L^2_{\mu}(\scrM;\mathscr{Y})}  \leq \norm{\sder\bdmc{T}}^2_{L^2_{\mu}\left(\scrM; \HS(\scrH_{\mu}, \mathscr{Y})\right)}\,.
    \end{equation*}
\end{theorem}

The operator learning problem with error control in $H^1_{\mu}(\scrM;\scrY)$ is formulated as
\begin{subequations}\label{eq:loss_h1}
\begin{align}
    &\bw^{\ddagger}=\argmin_{\bw\in\R^{d_w}} \calL_{H^1_{\mu}}^{\infty}(\bw)\,, &&& (\text{Operator learning objective})\\
    &\calL_{H^1_\mu}^{\infty}\coloneqq\frac{1}{2}\norm{\bdmc{G}-\widetilde{\bdmc{G}}(\cdot;\bw)}_{H^1_{\mu}(\scrM;\mathscr{Y})}^2\,. &&& (\text{Error control in }H^1_{\mu}(\scrM;\scrY))
\end{align}
\end{subequations}
The operator learning objective $\calL_{H^1_\mu}^{\infty}$ can be estimated via joint samples of the operator evaluations and stochastic derivative evaluations $\{$ $m_j\iid \mu$, $\bdmc{G}(m_j)$, $\sder\bdmc{G}(m_j)\}_{j=1}^{n_t}$, which leads to a loss function $\calL_{H^1_\mu}^{n_t}$ defined as follows:
\begin{equation}\label{eq:empirical_loss_h1}
\begin{aligned}
    \calL_{H^1_\mu}^{\infty}(\bw)\approx \calL_{H^1_\mu}^{n_t}(\bw;\{m_j\}_{j=1}^{n_t}) &\coloneqq \frac{1}{2n_t}\sum_{j=1}^{n_t}\bigg(\norm{\bdmc{G}(m_j) - \widetilde{\bdmc{G}}(m_j;\bw)}_{\bC_n^{-1}}^2\\
    &\quad + \norm{\sder\bdmc{G}(m_j)-\sder\widetilde{\bdmc{G}}(m_j;\bw)}_{\HS(\scrH_{\mu},\scrY)}^2\bigg)\,.
\end{aligned}
\end{equation}
In the context of neural network-based operator learning, we refer to the resulting operator surrogates as \textit{derivative-informed neural operators} (DINOs).

\subsection{Matrix representations of Hilbert--Schmidt operators}\label{subsec:hs_matrix}
We consider a matrix representation of the stochastic derivative for the purpose of generating training samples. For an arbitrary pair of orthonormal basis (ONB) on the parameter and observable CM spaces
\begin{equation*}
        \scrH_{\mu}\text{-ONB }:\quad\{\psi_k\}_{k=1}^{\infty}\,,\quad
        \scrY\text{-ONB }:\quad\{\bv_j\}_{j=1}^{d_y}\,,
\end{equation*}
we define a \textit{Jacobian}, denoted by $\bJ:\scrM\to\in\HS(l^2,\R^{d_y})$ where $l^2$ denotes the Hilbert space of squared-summable sequences, using an isometric isomorphism between $\HS(l^2,\R^{d_y})$ and $\HS(\scrH_{\mu},\scrY)$ defined by the bases:
\begin{subequations}\label{eq:jacobian_def}
\begin{align}
    \left(\bJ(m)\right)_{jk} &\coloneqq \bv_j^T\bC_n^{-1} \sder\bdmc{G}(m)\psi_k\,, &&&(\text{Bijective linear mapping})\label{eq:jacobian_components}\\
    \left(\bJ(m)^T\right)_{jk} &\coloneqq \left\langle\psi_j,\sder\bdmc{G}(m)^*\bv_k\right\rangle_{\cpr^{-1}}\,,&&&(\text{Jacobian matrix transpose})\label{eq:jacobian_transp} \\
    \norm{\bJ(m)}_{F} &=\norm{\sder\bdmc{G}(m)}_{\HS(\scrH_{\mu};\scrY)}\,, &&& (\text{Isometry})\label{eq:isometry}
\end{align}
\end{subequations}
where $\norm{\cdot}_F$ is the Frobenius norm. The Frobenius and HS inner products are the same on $\HS(l^2, \R^{d_y})$, and $\bJ(m)$ can be interpreted as a matrix via its components defined in \cref{eq:jacobian_components,eq:jacobian_transp}. Importantly, the isometry in \cref{eq:isometry} is independent of the choice of basis, while the mapping between $\bJ$ and $\sder\bdmc{G}$ in \cref{eq:jacobian_components} depends on the choice of basis.

With the matrix representation of the stochastic derivative, training samples can be generated as Jacobian matrices (i.e., Jacobian evaluations) at each parameter sample. Furthermore, we can estimate the derivative approximation error at each parameter sample via the Frobenius norm of the error in the Jacobian matrices. However, the size of Jacobian matrices can be problematic in numerical computation. Assume that numerical computation is performed in a discretized parameter space $\scrM^h\subset\scrM$ using the Galerkin method, where $\scrM^h$ is isomorphic to $\R^{d_m}$. Then, the discretized Jacobian $\bJ^h$ outputs $\R^{d_y\times d_m}$ matrices, i.e., $\bJ^h:\scrM^h\to \R^{d_y\times d_m}$; thus, storing and learning the Jacobian matrices generated at a large number of parameter samples can be intractable for large-scale problems. As a result, dimension reduction of the parameter space is essential for derivative-informed $H^1_{\mu}$ operator learning.

\cite{oleary2024derivative} argue that restricting the derivative-informed operator learning using a pre-determined rank-$r$ reduced basis $\{\psi_j\}_{j=1}^{r}$ with $r\ll d_m$ leads to tractable and accurate learning of the derivative for a wide-range of PDE models. We adopt this strategy in this work. The following section describes DINO with error control in $H^1_{\mu}$ that extends reduced basis derivative-informed operator learning to our setting.
\section{Reduced basis derivative-informed neural operator}\label{sec:dino}
Assume we have a set of reduced $\scrH_{\mu}$-ONBs of rank $r$ denoted by $\{\psi_j\}_{j=1}^{r}$. They define a pair of linear encoders $\Psi_r^*\in \HS(\scrH_{\mu},\R^r)$ and decoders $\Psi_r\in \HS(\R^r, \scrH_{\mu})$ on $\scrH_{\mu}$ with $\Psi_r^*\Psi_r = \bI_r\in\R^{r\times r}$, where $\bI_r$ is the identity matrix:
\begin{subequations}\label{eq:projector}
\begin{align}
    \Psi_{r}^*&:\scrH_{\mu}\ni m\mapsto \sum_{j=1}^{r} \left\langle m, \psi_j\right\rangle_{\cpr^{-1}}\be_j\in\R^{r}\,, &&& (\text{Parameter encoder})\\
    \Psi_{r}&:\R^{r}\ni \bmm_r\mapsto\sum_{j=1}^{r}(\bmm_r)_j\psi_j\in\scrH_{\mu}\,, &&& (\text{Parameter decoder})
\end{align}
\end{subequations}
where $\Psi_r^*$ is the adjoint of $\Psi_r$ and $\be_j$ is the unit vector along the $j$th coordinate. Using the matrix representation of the HS operator introduced in \cref{subsec:hs_matrix}, the linear encoder and decoder may be represented as 
\begin{equation*}
    \Psi_r = \begin{bmatrix}
        \vertbar & \vertbar & & \vertbar\\
        \psi_1 & \psi_2 &\cdots & \psi_r\\
        \vertbar & \vertbar & & \vertbar
    \end{bmatrix}\,, \quad \Psi_r^* = \begin{bmatrix}
        
        \horzbar& \left\langle\psi_1,\cdot\right\rangle_{\cpr^{-1}} & \horzbar \\
        \horzbar & \left\langle\psi_2,\cdot\right\rangle_{\cpr^{-1}} & \horzbar \\
        & \vdots &\\
        \horzbar & \left\langle\psi_r,\cdot\right\rangle_{\cpr^{-1}} & \horzbar
    \end{bmatrix}\,.
\end{equation*}
We extend the range and domain of the encoder and decoder from $\scrH_{\mu}$ to $\scrM$ and define a projection $\calP_r$ on $\scrM$ as follows:
\begin{equation}\label{eq:projection}
    \mathcal{P}_r\coloneqq\Psi_r\Psi_r^*:\scrM\to \text{span}(\{\psi_j\}_{j=1}^r)\,.
\end{equation}
We emphasize that the $\scrM$-adjoint of $\Psi_r$ is $\Psi^*_r\cpr$, and the $\scrM$-adjoint of $\Psi_r^*$ is $\cpr^{-1}\Psi_r$. Similarly, let $\bV\in\HS(\R^{d_y},\scrY)$ be a matrix with columns consist of $\scrY$-ONB vectors $\{\bv_j\}_{j=1}^{d_y}$:
\begin{align*}
    \bV = \begin{bmatrix}
        \vertbar & \vertbar & & \vertbar\\
        \bv_1 & \bv_2 &\cdots & \bv_{d_y}\\
        \vertbar & \vertbar & & \vertbar
    \end{bmatrix}\,. &&&(\text{Observable basis})
\end{align*}
We note that $\bV^* = \bV^T\bC_n^{-1}$, where $\bV^T$ is the matrix transpose of $\bV$. 

We parameterize the operator surrogate using a neural network $\nn:\R^r\times\R^{d_w}\to\R^{d_y}$:
\begin{align}\label{eq:rb_dino}
    \widetilde{\bdmc{G}}(m;\bw) &\coloneqq \bV\nn(\Psi_r^* m, \bw)\,. &&& (\text{Reduced basis neural operator})
\end{align}
The neural network represents the nonlinear mapping from the coefficients of reduced $\scrH_{\mu}$-ONBs $\{\psi_j\}_{j=1}^r$ to the coefficients of $\scrY$-ONBs $\{\bv_j\}_{j=1}^{d_y}$.
\begin{remark}
    In this work, we restrict our attention to parameter dimension reduction only. We acknowledge that reducing the observables is essential for many BIPs, such as those with high-resolution image data and time-evolving data. Recent work by \cite{baptista2022gradientbased} studies optimal joint parameter and data dimension reduction in the context of BIPs based on logarithmic Sobolev inequality, which can be readily applied to our setting due to \Cref{thm:log-sobolev}. However, there are many more practical considerations when jointly reducing the input and output dimensions for efficient DINO training. For this reason, we reserve dimension reduction of the observables for future work.
\end{remark}

The stochastic derivative of the operator surrogate and its adjoint can be expressed using the surrogate reduced Jacobian $\widetilde{\bJ_r}(\cdot;\bw):\scrM\to\R^{d_y\times r}$ through the neural network Jacobian $\partial_{\bmm_r} \bff_{\text{NN}}(\cdot, \bw):\R^{r}\to\R^{d_y\times r}$:
\begin{align*}
    \widetilde{\bJ_r}(m;\bw) &\coloneqq \partial_{\bmm_r} \bff_{\text{NN}}(\Psi_r^*m, \bw)\,,\quad \begin{cases}
    \sder\widetilde{\bdmc{G}}(m;\bw) = \bV\widetilde{\bJ_r}(m;\bw)\Psi_r^*\,,\\
    \sder\widetilde{\bdmc{G}}(m;\bw)^* = \Psi_r\widetilde{\bJ_r}(m;\bw)^T\bV^*\,.      
    \end{cases}
\end{align*}
Using the reduced basis architecture in \cref{eq:rb_dino} with $\scrH_{\mu}$ and $\scrY$-ONBs and the isometric isomorphism in \cref{eq:jacobian_components}, the derivative-informed $H^1_{\mu}$ operator learning objective in \cref{eq:loss_h1} can be reduced as follows:
\begin{align*}
\begin{split}
    \calL_{H^1_\mu}^{\infty}(\bw)
    &\propto\frac{1}{2}\mathbb{E}_{M\sim\mu}\bigg[\Big\lVert\underbrace{\bV^*\bdmc{G}(M) - \nn(\Psi_r^* M, \bw)}_{\displaystyle\in\R^{d_y}}\Big\rVert^2+\\
    &\quad\Big\lVert\underbrace{\bV^*\sder \bdmc{G}(M)\Psi_r - \partial_{\bmm_r}\nn(\Psi_r^*M, \bw)}_{\displaystyle\bJ_r(M)-\widetilde{\bJ_r}(M;\bw)\in\R^{d_y\times r}}\Big\rVert^2_{F}\bigg]\,,
\end{split}
\end{align*}
where constant terms independent of $\bw$ are eliminated and the reduced Jacobian $\bJ_r:\scrM\to\R^{d_y\times r}$ of the PtO map is given by 
\begin{align*}
    \bJ_r(m) &\coloneqq \bV^*\sder\bdmc{G}(m)\Psi_r\,. &&& (\text{Reduced Jacobian matrix})
\end{align*}
The reduced loss function can now be estimated via joints samples $\{m_j\iid\mu$, $\bV^*\bdmc{G}(m_j)$, $\bJ_r(m_j)\}_{j=1}^{n_t}$.

\subsection{A brief summary}

\begin{figure}[!h]
\centering
\begin{tikzpicture}[scale = 0.8, transform shape]

\node[bag] (Point) at (4.5,3.2) [minimum width=1cm,minimum height=2cm] {\textbf{Reduced basis DINO surrogate}\\ \\$\displaystyle\bdmc{G}(m)\approx\widetilde{\bdmc{G}}(m;\bw) \coloneqq \bV\nn(\Psi_r^*m, \bw)$};

\node[draw, trapezium, thick, rotate=-90, trapezium stretches body, 
  text width=2cm, align=center] (Input) at (0,0) {\rotatebox{90}{%
  \parbox{2.8cm}{\centering Input encoder\\ \hspace{0.1 in}\\$\displaystyle\Psi_r^*\in\HS(\scrM, \R^r)$}}};

\node[bag,draw,outer sep=0pt,thick] (NN) at (4.5,0) [minimum width=3cm,minimum height=4cm] {Neural network\\ with weights $\bw$ \\ \\
                                                                    $\nn(\cdot, \bw):\mathbb{R}^r\to\R^{d_y}$};
\draw (Input.north west) -- (NN.north west) (Input.north east) -- (NN.south west);

\node[bag,draw,outer sep=0pt,thick] (Output) at (9,0)[minimum width=2cm,minimum height=2cm] {Output basis\\ \\
                                                                    $\bV\in\HS(\R^{d_y}, \scrY)$};

\draw (NN.north east) -- (Output.north west) (NN.south east) -- (Output.south west);

\node[bag,draw,outer sep=0pt,thick] (NN1) at (4.5,-4.6) [minimum width=7cm,minimum height=3cm] {Neural network training using joint samples of input--output--reduced Jacobian\\ 
\\
                                                                    $\displaystyle\min_{\bw\in\R^{d_w}} \frac{1}{2n_t}\sum_{j=1}^{n_t}\left(\Big\lVert\bV^*\bdmc{G}(m_j) - \nn(\Psi_r^* m_j, \bw)\Big\rVert^2+ \Big\lVert\bJ_r(m_j) - \partial_{\bmm_r}\nn(\Psi_r^*m_j, \bw)\Big\rVert^2_{F}\right)$ };

\draw[dotted] (NN.south west) -- (NN1.north west);
\draw[dotted] (NN.south east) -- (NN1.north east);

\end{tikzpicture}
\caption{A schematic of reduced basis DINO architecture and learning for surrogate approximation $\widetilde{\bdmc{G}}\approx \bdmc{G}$ in $H^1_{\mu}(\scrM;\scrY)$. }
\end{figure}
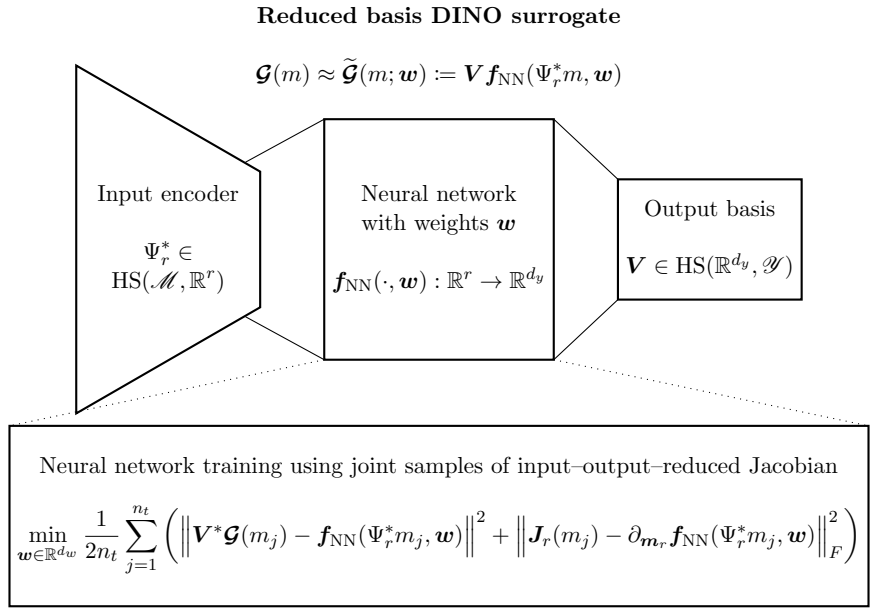

We emphasize the following important points about our operator learning formulation:
\begin{enumerate}[leftmargin=*]
    \item The derivative-informed operator learning $\bdmc{G}\approx\widetilde{\bdmc{G}}(\cdot;\bw)$ is formulated as an approximation problem in $H^1_{\mu}(\scrM;\scrY)$, a Sobolev space of nonlinear mappings between two separable Hilbert spaces $\scrM$ and $\scrY$. While the parameter space $\scrM$ has infinite dimensions and the observable CM space $\scrY$ has finite dimensions in our setting, the learning formulation is general and can be applied to infinite-dimensional output spaces. However, the stochastic derivative must remain an HS operator when the output space becomes a separable Hilbert space.
    \item The neural network in the reduced basis DINO learns the mapping from the reduced coefficient vector of the parameter $\Psi_r^*m\in\R^r$ to the model-predicted coefficient vector of the observables $\bV^*\bdmc{G}(m)\in\R^{d_y}$.
    \begin{align*}
        &\nn(\cdot, \bw)\approx \Psi_r^*m\mapsto \bV^*\bdmc{G}(m)\,. &&& (\text{Neural network approximation})
    \end{align*}
    \item The neural network Jacobian of the reduced basis DINO learns the nonlinear mapping from the reduced coefficient vector of the parameter $\Psi_r^*m\in\R^{r}$ to the model-predicted reduced Jacobian matrix $\bJ_r(m)\in\R^{d_y\times r}$:
    \begin{align*}
        &\partial_{\bmm_r}\nn(\cdot, \bw)\approx \Psi_r^*m\mapsto \bJ_r(m)\,. &&&(\text{Neural network Jacobian approximation})
    \end{align*}
    \item As a result of the reduced basis architecture, the training sample storage and training cost are independent of the discretization dimension of the parameter space.
\end{enumerate}

\begin{remark}
In the following presentation, we often omit the notation for dependency on the neural network parameter $\bw$ and use the tilde symbol $\widetilde{\cdot}$ when referring to the quantities computed via surrogate evaluations.
\end{remark}

\subsection{Derivative and prior-based reduced bases}
This subsection describes two types of reduced bases that can be used to construct DINO. The first type is based on the derivative-informed subspace (DIS, \citealt{constantine2014active, zahm2020gradient, Cui2021datafree, oleary2022derivative}. The reduced bases for the derivative-informed subspace can be found by the following eigenvalue problem in $\scrH_{\mu}$ for the ppGNH \eqref{eq:ppgh}:
\begin{align}\label{eq:jtj_eigenvalue}
    \text{Find } \big\{(\disev{j}, \disbasis{j})&\in\R_+\times\scrH_{\mu}\big\}_{j=1}^{\infty}\text{ with decreasing } \disev{j} \text{ such that }\hfill\nonumber\\
    &\begin{cases}
    \left(\mathbb{E}_{M\sim\mu}\left[\calH(M)\right] - \disev{j}\calI_{\scrH_{\mu}}\right)\disbasis{j} = 0\,,& j\in\N\,;\\
    \left\langle\disbasis{j},\disbasis{k}\right\rangle_{\cpr^{-1}} = \delta_{jk} & j,k\in\N\,.
    \end{cases}
\end{align}
We select the first $r$ bases that correspond to the $r$ largest eigenvalues to form encoders and decoders. During numerical computation, a Monte Carlo estimate of the expected ppGNH $\mathbb{E}_{M\sim\mu}[\calH(M)]$ is computed at a set of prior samples $m_j\iid\mu$, $j=1,\dots, n_{\text{DIS}}$:
\begin{equation}\label{eq:mc_ppgnh}
    \widehat{\cal{H}}(\{m_j\}_{j=1}^{n_{\text{DIS}}}) \coloneqq \frac{1}{n_{\text{DIS}}}\sum_{j=1}^{{n_{\text{DIS}}}}\calH(m_j)\approx\mathbb{E}_{M\sim\mu}[\calH(M)]\,.
\end{equation}
The eigenvalue problem in \cref{eq:jtj_eigenvalue} is solved to obtain the eigenpairs $\big\{(\widehat{\disev{j}}, \widehat{\disbasis{j}})\big\}_{j=1}^{r}$ of $\widehat{\calH}$, which gives the following DIS approximation of the expected ppGNH:
\begin{equation}\label{eq:dis_gnh}
    \mathbb{E}_{M\sim\mu}[\calH(M)]\approx \widehat{\Psi_r^{\text{DIS}}}\widehat{\boldsymbol{\Lambda}_r^{\text{DIS}}}\widehat{\Psi_r^{\text{DIS}}}^*,
\end{equation}
where the linear encoder and decoder are defined as in \cref{eq:projector} and $\widehat{\boldsymbol{\Lambda}_r^{\text{DIS}}}\in\R^{r\times r}$ is a diagonal matrix consists of the eigenvalues.

The second type of reduced bases is based on the Karhunen--Lo\`eve expansion (KLE) of the prior distribution:
\begin{equation*}
    M = \sum_{j=1}^{\infty} \sqrt{\kleev{j}} \Xi_j\eta_j \sim \mu\,,\quad \Xi_j\iid \mathcal{N}(0, 1)\,,\quad \left\langle\eta_j, \eta_k\right\rangle_{\scrM} = \delta_{jk}\,,
\end{equation*}
where $\{(\kleev{j},\eta_j)\in\R_+\times\scrM\}_{j=1}^{\infty}$ are eigenpairs of the prior covariance $\cpr$ with $\scrM$-orthonormal eigenbases and decreasing eigenvalues. We refer to the $r$-dimensional subspace spanned by $\{\eta_j\}_{j=1}^r$ as the rank-$r$ KLE subspace or simply the KLE subspace. A set of reduced $\scrH_{\mu}$-ONBs of the KLE subspace $\{\klebasis{j}\}_{j=1}^{r}$ can be found by
\begin{equation*}
    \klebasis{j} = \sqrt{\kleev{j}}\eta_j\,, \quad 0\leq j\leq r\,.
\end{equation*}
The KLE reduced bases $\{\klebasis{j}\}_{j=1}^{r}$ can be computed with high precision for some representations of the $\cpr$, notably Laplacian inverse or bi-Laplacian inverse Mat\'ern covariances for Gaussian random functions \citep{Bui2013, VillaPetraGhattas21}. More generally, the KLE reduced bases can be approximated from samples.

\subsection{Training sample generation and cost analysis for PDE models}\label{subsec:data_generation}

We describe a training sample generation procedure and its cost analysis when the PtO map $\bdmc{G}$ is defined through a PDE. In particular, we consider an abstract variational residual form of the PDE as follows:
\begin{align}\label{eq:pde_residual_form}
    \text{Given } m\in\scrM \text{ find } u\in\scrU\text{ such that } \mathcal{R}(u, m) = 0 \in\scrV\,, && (\text{PDE model})
\end{align}
where $\scrU$ and $\scrV$ are Hilbert spaces corresponding to the spaces of PDE state and residual, and $\calR:\scrU\times\scrM\to\scrV$ is the PDE residual operator. The residual space $\scrV$ is the dual space of the \textit{space of adjoint variable} in the context of PDE-constrained optimization \citep{Antil2018, manzoni2021optimal} and the two spaces are identical when $\scrV$ is a Hilbert space. While the numerical examples in this work focus on steady-state problems where $\scrU$ is a Sobolev space defined over a spatial domain, e.g., $H^1(\Omega)$ where $\Omega$ is a spatial domain, our methodology is general. It can be applied to, e.g., time-evolving problems where $\scrU$ is a time-evolving space, e.g., $L^2([0, T];H^1(\Omega))$ where $T$ is a terminal time. 

Assume the PDE solution operator $\mathcal{F}:\scrM\to\scrU$ is composed with a linear observation operator $\bdmc{O}\in \HS(\scrU,\scrY)$ to define the PtO map $\bdmc{G}$:
\begin{align*}
    \bdmc{G}\coloneqq \bdmc{O}\circ\mathcal{F}\,\,,\quad \mathcal{R}(\mathcal{F}(m),m) = 0\quad\mu\text{-a.e.} &&& (\text{PDE-constrained PtO map})
\end{align*}
Each evaluation of the PtO map requires solving a PDE in \cref{eq:pde_residual_form}.
The action and the adjoint action of the stochastic derivative of the PtO map is given by the stochastic derivative of the PDE solution operator $D_{\scrH_{\mu}}\calF(m)\in B(\scrH_{\mu}, \scrU)$:
\begin{equation*}
    \sder\bdmc{G}(m)\delta m = \left(\bdmc{O}\circ\sder\mathcal{F}(m)\right)\delta m\,,\quad \sder\bdmc{G}(m)^*\delta \by = \left(\sder\mathcal{F}(m)^*\circ\bdmc{O}^*\right)\delta\by\,,
\end{equation*}
where the action of the derivative is given by the partial G\^ateaux derivatives of the residual with respect to the PDE state and the parameter (in the direction of $\scrH_{\mu}$), denoted by $\partial_\scrU\mathcal{R}(\mathcal{F}(m),m)\in B(\scrU,\scrV)$ and $\partial_{\scrH_{\mu}}\mathcal{R}(\mathcal{F}(m),m)\in B(\scrH_{\mu},\scrV)$ respectively. In particular, the implicit function theorem \citep{Ciarlet2013} implies the following relations:
\begin{align*}
    \sder\calF(m)\delta m &= -\underbrace{\left(\partial_{\scrU} \mathcal{R}(\mathcal{F}(m), m)\right)^{-1}}_{\displaystyle\mathclap{\delta v\mapsto\delta u}}\underbrace{\partial_{\scrH_{\mu}}\mathcal{R}(\mathcal{F}(m),m)}_{\displaystyle\mathclap{\delta m\mapsto \delta v}}\delta m\,, &&& (\text{Direct sensitivity})\\
    \sder\calF(m)^*\delta u &= -\underbrace{\partial_{\scrH_{\mu}}\mathcal{R}(\mathcal{F}(m),m)^*}_{\displaystyle\mathclap{\delta v\mapsto\delta m}}\underbrace{\left(\partial_{\scrU} \mathcal{R}(\mathcal{F}(m), m)^*\right)^{-1}}_{\displaystyle\mathclap{\delta u\mapsto\delta v}}\delta u\,, &&& (\text{Adjoint sensitivity})
\end{align*}
where $\delta v\in\scrV$ indicates a variation in the PDE residual or, equivalently, an adjoint variable. Evaluating the action of $\sder\calF(m)$ requires solving the linearized PDE problem for $\delta v\mapsto\delta u$, and evaluating its adjoint action $\sder\calF(m)^*$ requires solving the linear adjoint problem for $\delta u\mapsto\delta v$; see, e.g., \citealt[Section 5]{Ghattas2021}.

The associated computational cost for generating $n_t$ training samples at parameter samples $m_j\iid \mu$, $j=1\dots n_t$, can be decomposed as follows:
\begin{equation*}
\setlength\extrarowheight{5pt}
\begin{tabular}{|c r l|}\hline
        & $1\times$ & \textbf{Cost of reduced bases estimation}\\
     $+$ & $n_t\times$ & \textbf{Cost of a PDE solve}\\
     $+$ & $n_t\times$ & \textbf{Cost of evaluating the reduced Jacobian }$\bJ_r$\\\hline
     $=$ & & \textbf{Cost of sample generation for DINO training} \\\hline
\end{tabular}
\setlength\extrarowheight{-5pt}
\end{equation*}
When compared to $L_\mu^2$ training of an operator surrogate, DINO training requires additionally forming reduced Jacobian matrices $\bJ_r(m_j)\in\R^{d_y\times r}$ at each parameter sample $m_j$ via rows or columns. In \Cref{tab:cost_jacobian} and the following paragraph, we provide a simple cost analysis for this task when the parameter and the state space are discretized. 

For time-evolving problems, we assume the state space $\scrU$ is discretized such that $\scrU^h$ is isomorphic to $\R^{d_t\times d_u}$, where $d_t$ is the dimension of the temporal discretization and $d_u$ is the dimension of the spatial discretization. We assume such discretization leads to $d_t$ systems of equations (linear PDE) or $d_t$ iterative systems of equations (nonlinear PDE) of size $d_u\times d_u$. For steady-state problems, we take $d_t = 1$. When a direct solver is used, the cost of factorizing systems of equations for a typical PDE problem is $O(d_td_u^{3/2})$ and $O(d_td_u^{2})$ for 2D and 3D spatial domains, while back-substitution has a cost of $O(d_td_u \ln d_u)$ \citep{davis2016survey}. The factorization used for solving a linear PDE can be reused to form $\bJ_r(m_j)$ via back-substitution, making the additional cost of $H^1_{\mu}$ training sample generation scale much slower with $d_u$ compared to the cost of $L^2_{\mu}$ training. When an iterative solver is used, one can reuse preconditioners for a linear PDE to form $\bJ_r(m_j)$, but their cost analysis should be performed on a case-by-case basis. For nonlinear PDEs, one needs to solve one linear system of equations with $\min\{r, d_y\}$ different right-hand side vectors to form $\bJ_r(m_j)$, which is potentially much cheaper than solving a highly nonlinear PDE problem via iterative methods such as the Newton--Rapshon method.

\begin{table}[!h]
\centering
\setlength\extrarowheight{10pt}
\begin{tabular}{|rc|}\hline
       \multicolumn{2}{|c|}{\textbf{Forming a reduced Jacobian matrix $\bJ_r(m_j)\in\R^{d_y\times r}$}}\\\hline
       Linearized forward sensitivity: & $\quad \displaystyle\text{column}_l(\bJ_r(m_j)) = \bV^*\sder\bdmc{G}(m_j)\psi_l \,,\quad l=1\dots r\,.$\\
         Adjoint sensitivity: &
       $\quad\displaystyle\text{row}_k(\bJ_r(m_j)) = \Psi_r^*\sder\bdmc{G}(m_j)^*\bv_k \,,\quad k=1\dots d_y\,.$\\
       \hline
\end{tabular}
    \begin{tabular}{|c|c|c|}\hline
        \textbf{Linear} & \textbf{Solver} & \makecell{\textbf{Operation}\\ ($d_t=1$ for steady-state problems)}\\\hline
        \multirow{2}{*}{\cmark} & Direct & \makecell[l]{$d_t\times\min\{d_y, r\}\times$ Back-substitution \\(\textbf{note: significant cost saving from reusing factorization})}  \\\hhline{~--}
         & Iterative & \makecell[l]{$d_t\times\min\{d_y, r\}\times$ Iterative solve \\(\textbf{note: significant cost saving from reusing preconditioner})}  \\\hline 
         \multirow{2}{*}{\xmark} & Direct & \makecell[l]{$d_t \times $ Factorization $+$ $d_t\times\min\{d_y, r\}\times$ Back-substitution}\\\hhline{~--} 
         & Iterative & \makecell[l]{$d_t \times $ Preconditioner build $+$  $d_t\times\min\{d_y, r\}\times$ Iterative solve}\\\hline
    \end{tabular}
    \caption{The cost analysis of forming reduced Jacobian matrix $\bJ_r(m)$ at a parameter sample $m_j$ given parameter reduced bases $\{\psi_j\}_{j=1}^{\infty}$ and observable basis $\{\bv\}_{j=1}^{d_y}$.\label{tab:cost_jacobian}}
\setlength\extrarowheight{-10pt}
\end{table}

\subsection{Neural operator approximation error}\label{subsec:error}
This subsection briefly discusses approximation error for reduced basis DINO surrogates. To the best of our knowledge, there are no existing theoretical studies on $H^1_{\mu}$ approximation error with input dimension reduction. We focus on theoretical results that isolate various sources of $L^2_{\mu}$ approximation error under the assumption that the true mapping lives in $H^1_{\mu}$ and comment on the relation between neural network size and the $L^2_{\mu}$ approximation error. 

Understanding the $L^2_{\mu}$ approximation error of the operator surrogate is important as it is closely linked to the efficiency of the DA procedure. The connection between the second stage acceptance probability and the $L^2_{\mu}$ approximation error is discussed in \cref{app:da_and_error}.

Here, we provide results on the $L^2_{\mu}(\scrM;\scrY)$ approximation error of the DIS and KLE reduced basis neural operators. Our results show that a reduced basis architecture leads to approximation error contributions due to truncation and neural network approximation of the optimal reduced mapping. This mapping, denoted by $\bdmc{G}_r$, can be defined explicitly \cite[Proposition 2.3]{zahm2020gradient} for a given pair of linear encoder $\Psi_r^*$ and decoder $\Psi_r$ constructed as in \cref{eq:projector}. Let $\calP_r \coloneqq \Psi_r\Psi_r^*$ be a projection on $\scrM$ as in \cref{eq:projection}. We have
\begin{subequations}\label{eq:optimal_reduced_mapping}
\begin{align}
        &\bdmc{G}_r(m) \coloneqq \mathbb{E}_{M\sim\mu}\left[\bdmc{G}\left(\calP_rm + (\calI_{\scrM}-\calP_r)M\right)\right]\,, && (\text{Subspace $L_\mu^2$ projection})\\
        &\norm{\bdmc{G}-\bdmc{G}_r}_{L^2_{\mu}(\scrM;\scrY)} = \inf_{\stackrel{\bdmc{T}:\scrM\to\scrY}{\text{Borel func.}}}\norm{\bdmc{G}-\bdmc{T}\circ\calP_r}_{L^2_{\mu}(\scrM;\scrY)}\,. && (\text{Optimal reduced mapping})
\end{align}
\end{subequations}
The following propositions extend the results on DIS and KLE subspace by \citet[Proposition 2.6 and 3.1]{zahm2020gradient} to a function space setting using the Poincar\'e inequality in \Cref{thm:poincare}. The proofs are provided in \cref{app:proof_error}. 
\begin{proposition}[$L^2_\mu$ approximation error, DIS]\label[proposition]{prop:dis_error}
    Assume $\bdmc{G}\in H^1_{\mu}(\scrM;\scrY)$. Let $\{(\disev{j}$, $\disbasis{j})\}_{j=1}^{\infty}$ and $\{(\widehat{\disev{j}}, \widehat{\disbasis{j}})\}_{j=1}^{\infty}$ be the $\scrH_{\mu}$-orthonormal eigenpairs of the expected ppGNH $\mathbb{E}_{M\sim\mu}\left[\calH(M)\right]$ in \cref{eq:jtj_eigenvalue} and its estimator $\widehat{\calH}$ in \cref{eq:mc_ppgnh} with decreasing eigenvalues. Consider a reduced basis neural operator constructed using a linear encoder $\widehat{\Psi_r^{\normalfont\ignorespaces \text{DIS}}}^*\in \HS(\scrM,\R^{d_y})$ based on $\{\widehat{\disbasis{j}}\}_{j=1}^{r}$ as in \cref{eq:projector} and any $\scrY$-orthonormal basis $\bV\in\HS(\R^{d_y},\scrY)$:
    \begin{equation*}
        \widetilde{\bdmc{G}}(\cdot;\bw) \coloneqq \bV\circ\nn(\cdot,\bw)\circ\widehat{\Psi_r^{\normalfont\ignorespaces \text{DIS}}}^*\,.
    \end{equation*}
    The following upper bound holds for the $L^2_{\mu}(\scrM;\scrY)$ approximation error of $\widetilde{\bdmc{G}}$ to $\mathcal{\bdmc{G}}$:
    \begin{align*}
        \norm{\bdmc{G}-\widetilde{\bdmc{G}}(\cdot;\bw)}_{L^2_{\mu}(\scrM;\scrY)}&\leq \overbrace{\norm{\nn(\cdot,\bw) - \bV^*\circ\bdmc{G}_r\circ\widehat{\Psi_r^{\normalfont\ignorespaces \text{DIS}}}}_{L^2_{\mathcal{N}(\bzero, \bI_r)}(\R^r;\R^{d_y})}}^{\displaystyle\mathclap{\text{Neural network error}}}\\
        &\qquad +  \Big(\underbrace{\sum_{j=r+1}^{\infty}\disev{j}}_{\displaystyle\mathclap{\text{Basis truncation error}}} + \underbrace{2r\norm{\mathbb{E}_{M\sim\mu}\left[\calH(M)\right]-\widehat{\calH}}_{B(\scrH_{\mu})}}_{\displaystyle\mathclap{\text{Sampling error}}}\Big)^{1/2}\,,
    \end{align*}
    where $\widehat{\Psi_r^{\normalfont\ignorespaces \text{DIS}}}\in \HS(\R^{r},\scrM)$ is the linear decoder based on $\{\widehat{\disbasis{j}}\}_{j=1}^{r}$ as in \cref{eq:projector}, and $\bdmc{G}_r$ is the optimal reduced mapping of $\bdmc{G}$ in \cref{eq:optimal_reduced_mapping}.
\end{proposition}

\begin{proposition}[$L^2_\mu$ approximation error, KLE]\label[proposition]{prop:kle_error}
    Assume that $\bdmc{G}\in H^1_{\mu}(\scrM;\scrY)$ is Lipshitz continuous with a Lipschitz constant $c_{\bdmc{G}}\geq 0$, i.e., 
    \begin{equation*}
        \norm{\bdmc{G}(m_1)-\bdmc{G}(m_2)}_{\bC_n^{-1}} \leq c_{\bdmc{G}}\norm{m_1-m_2}_{\scrM}\quad \forall m_1, m_2\in\scrM\,.
    \end{equation*}
    Let $\{(\kleev{j}, \eta_j)\}_{j=1}^{\infty}$ be the $\scrM$-orthonormal eigenpairs of $\cpr$ with decreasing eigenvalues. Consider a reduced basis neural operator constructed as \cref{eq:rb_dino} using a linear encoder ${\Psi_r^{\normalfont\ignorespaces \text{KLE}}}^*\in \HS(\scrM,\R^{d_y})$ based on $\left\{\klebasis{j}\coloneqq\sqrt{\kleev{j}}\eta_j\right\}_{j=1}^{r}$ as in \cref{eq:projector} and any $\scrY$-orthonormal basis $\bV\in\HS(\R^{d_y},\scrY)$:
        \begin{equation*}
        \widetilde{\bdmc{G}}(\cdot;\bw) \coloneqq \bV\circ\nn(\cdot,\bw)\circ{\Psi_r^{\normalfont\ignorespaces \text{KLE}}}^*\,.
    \end{equation*}
    The following upper bound holds for $L^2_{\mu}(\scrM;\scrY)$ approximation error of $\widetilde{\bdmc{G}}$ to $\bdmc{G}$:
    \begin{align*}
        \norm{\bdmc{G}-\widetilde{\bdmc{G}}(\cdot;\bw)}_{L^2_{\mu}(\scrM;\scrY)}&\leq \underbrace{\norm{\nn - \bV^*\circ\bdmc{G}_r\circ\Psi_r^{\normalfont\ignorespaces \text{KLE}}}_{L^2_{\mathcal{N}(\bzero, \bI_r)}(\R^r;\R^{d_y})}}_{\displaystyle\mathclap{\text{Neural network error}}} + \underbrace{c_{\bdmc{G}}\Big(\sum_{j=r+1}^{\infty}\left(\kleev{j}\right)^2\Big)^{1/2}}_{\displaystyle\mathclap{\text{Basis truncation error}}},
    \end{align*}
    where $\Psi_r^{\normalfont\ignorespaces \text{KLE}}\in \HS(\R^{r},\scrM)$ is the linear decoder based on $\{\klebasis{j}\}_{j=1}^{r}$ as in \cref{eq:projector} and $\bdmc{G}_r$ is the optimal reduced mapping of $\bdmc{G}$ in \cref{eq:optimal_reduced_mapping}. Additionally, we have
    \begin{equation*}
        \sum_{j=r+1}^{\infty}\disev{j}\leq c_{\bdmc{G}}^2\sum_{j=r+1}^{\infty}\left(\kleev{j}\right)^2\,.
    \end{equation*}
    where $\{\disev{j}\}_{j=1}^{\infty}$ consists of decreasing eigenvalues of the expected ppGNH in \cref{eq:jtj_eigenvalue}.
\end{proposition}

Furthermore, universal approximation theories of neural networks can help us understand the expressiveness of the neural network architecture (e.g., width, breadth, and activation functions) used in reduced basis neural operator surrogates. An important question is the neural network size, measured by the size of the weight $d_{w}$, needed to achieve a given neural network error tolerance. An exponential convergence in $L^2_{\mathcal{N}(\bzero,\bI_r)}(\R^r;\R)$ for approximating certain analytic functions by deep neural networks with the ReLU activation function is established by \citet[Theorem 4.7]{schwab2023deep}. Their theoretical results can be directly applied to the neural network error in our setting by stacking $\R^{d_y}$ of these deep neural networks to form an output space of $\R^{d_y}$, given that the optimal reduced mapping $\bdmc{G}_r$ is sufficiently regular. Using this construction, the convergence rate derived by Schwab and Zech is scaled linearly by $d_y$.
\section{Geometric MCMC via reduced basis neural operator}\label{sec:surrogate_mcmc}
This section derives dimension-independent geometric MCMC methods with proposals entirely generated by a trained reduced basis neural operator. This work focuses on the mMALA method introduced in \cref{subsec:mMALA_intro} and approximates all components in the mMALA proposal using the surrogate. We note that the derivation in this section is similar to the DR-$\infty$-mMALA method by \cite{lan2019adaptive}, except that our derivation (i) does not involve prior covariance factorization\footnote{We acknowledge that \cite{lan2019adaptive} utilizes $\scrH_{\mu}$-orthonormal reduced bases, which means that the DR-$\infty$-mMALA algorithm can be implemented without using prior factorization.}, (ii) does not distinguish between KLE and DIS reduced bases, and (iii) involves the reduced basis neural operator surrogate introduced in \cref{eq:rb_dino}.

\subsection{Surrogate approximation}
Given a trained neural network $\nn(\cdot;\bw^{\ddagger})$ as in \cref{eq:rb_dino}, we can approximate the data misfit in \cref{eq:bayes_rule_gauss} with $\widetilde{\Phi^{\by}}(\cdot;\bw^{\ddagger})\approx\Phi^{\by}$, 
\begin{subequations}
\begin{align}
    \widetilde{\Phi^{\by}}(m) &\equiv \widetilde{\Phi^{\by}_r}(\Psi_r^*m)\,, &&& (\text{Data misfit})\\
    \widetilde{\Phi^{\by}_r}(\bmm_r) &\coloneqq \frac{1}{2}\norm{\bV^*\by - \nn(\bmm_r)}^2\,, &&& (\text{Reduced data misfit})
\end{align}
\end{subequations}
the ppg in \cref{eq:ppg} with $\sder\widetilde{\Phi^{\by}}(\cdot;\bw^{\ddagger})\approx \sder\Phi^{\by}$,
\begin{subequations}
\begin{align}
    \sder\widetilde{\Phi^{\by}}(m) &\equiv\Psi_r\widetilde{\bg_r}(\Psi_r^*m)\,. &&& (\text{ppg})\\
    \R^r\ni\widetilde{\bg_r}(\bmm_r) &\coloneqq \partial_{\bmm_r}\nn(\bmm_r)^T\left(\bV^*\by - \nn(\bmm_r)\right)\,,&&&(\text{Reduced ppg})\label{eq:surrogate_rppg}
\end{align}
\end{subequations}
and the ppGNH in \cref{eq:ppgh} with $\widetilde{\calH}(\cdot;\bw^{\ddagger})\approx \calH$,
\begin{subequations}
    \begin{align}
            \widetilde{\calH}(m) &\equiv \Psi_r\widetilde{\bH_r}(\Psi_r^*m)\Psi_r^* &&&(\text{ppGNH})\\
    \R^{r\times r}\ni\widetilde{\bH_r}(\bmm_r) &\coloneqq (\partial_{\bmm_r}\nn(\bmm_r))^T\partial_{\bmm_r}\nn(\bmm_r)\,. &&& (\text{Reduced ppGNH})\label{eq:surrogate_rppgnh}
    \end{align}
\end{subequations}

\subsection{Surrogate prediction of posterior local geometry}\label{subsec:surrogate_geometry}

An eigendecomposition of the surrogate reduced ppGNH $\widetilde{\bH_r}(\bmm_r)$ with $\bmm_r = \Psi_r^*m$ is computed at the beginning of each step in the MH algorithm. Such an eigendecomposition is necessary for fast evaluations of multiple terms in the transition rate ratio \cref{eq:mmala_trr}. We denote the eigendecomposition of the surrogate reduced ppGNH as
\begin{equation}\label{eq:surrogate_eigendecomposition}
    \widetilde{\bH_r}(\bmm_r) = \widetilde{\bP_r}(\bmm_r)\widetilde{\bD_r}(\bmm_r)\widetilde{\bP_r}(\bmm_r)^T\,,\quad \begin{cases}
        \widetilde{\bP_r}(\bmm_r)^T\widetilde{\bP_r}(\bmm_r) = \bI_r\,;\\
        \left(\widetilde{\bD_r}(\bmm_r)\right)_{jk} = \widetilde{d}_j(\bmm_r)\delta_{jk}\,.
    \end{cases}
\end{equation}
where $\widetilde{\bP_r}(\bmm_r)\in\R^{r\times r}$ is a rotation matrix in $\R^r$ with columns consist of eigenvectors and $\widetilde{\bD_r}(\bmm_r)\in\R^{r\times r}$ is a diagonal matrix consists of eigenvalues $\{\widetilde{d}_j(\bmm_r)\}_{j=1}^{r}$. The rotation matrix nonlinearly depends on the parameter $m$ through $\bmm_r$, thus leading to a pair of position-dependent linear decoder and encoder as follows:
\begin{subequations}
\begin{align}
    \widetilde{\Psi_r}(\bmm_r) &\coloneqq \Psi_r\widetilde{\bP_r}(\bmm_r) &&&(\text{Position-dependent linear decoder})\label{eq:position_dependent_encoder}\\
    \widetilde{\Psi_r}(\bmm_r)^* &\coloneqq \widetilde{\bP_r}(\bmm_r)^T\Psi_r^* &&&(\text{Position-dependent linear encoder})\label{eq:position_dependent_decoder}
\end{align}
\end{subequations}
where the adjoint is taken in $\scrH_{\mu}$ similar to \cref{eq:projector}. The basis functions extracted from the position-dependent linear encoder and decoder represent the dominant directions of the surrogate posterior local geometry.

\begin{remark}
    In the following presentation, we often omit the notation of position dependency for the encoder, decoder, rotation matrix, and eigenvalues of the ppGNH when there is no ambiguity. Moreover, we adopt the index notation of diagonal matrices as in \cref{eq:surrogate_eigendecomposition} to explicitly reveal its structure.
\end{remark}
The covariance of the local Gaussian approximation of the posterior in \cref{eq:local_gaussian} can be approximated as follows:
\begin{align}\label{eq:surrogate_cpost}
    \widetilde{\cpo}(m)=\cpr - \widetilde{\Psi_r}\left(\frac{\widetilde{d}_j}{\widetilde{d_j}+1}\delta_{jk}\right)\widetilde{\Psi_r}^*\cpr\,,\quad \widetilde{\cpo}(m)^{-1} = \cpr^{-1} + \cpr^{-1}\widetilde{\Psi_r}(\widetilde{d}_j\delta_{jk})\widetilde{\Psi_r}^*\,.
\end{align}

\subsection{Sampling from the surrogate mMALA proposal}
We consider an approximation to the mMALA proposal using the reduced basis neural operator with $\mathcal{K}(m)=\widetilde{\cpo}(m;\bw^{\ddagger})$. By \cref{eq:surrogate_cpost} and \cref{eq:mmala}, we arrive at the following surrogate mMALA proposal:
\begin{align*}
    \widetilde{\calQ_{\mmala}}(m,\cdot) &= \mathcal{N}\left(sm + \left(1-s\right)\widetilde{\mathcal{A}}(m), \left(1-s^2\right)\widetilde{\cpo}(m)\right)\,,\quad\,s = \frac{4-\triangle t}{4+\triangle t}\,,\\
    \widetilde{\mathcal{A}}(m) &= \widetilde{\Psi_r}\left(\frac{\widetilde{d}_j}{\widetilde{d}_j+1}\delta_{jk}\right)\widetilde{\Psi_r}^*m - \widetilde{\Psi_r}\left(\frac{1}{\widetilde{d}_j+1}\delta_{jk}\right)\widetilde{\bP_r}^T\widetilde{\bg_r}\,.
\end{align*}
To sample from the surrogate mMALA proposal, we consider the following lemma:
\begin{lemma}\label[lemma]{lemm:gauss_rv}
    Let $M\sim \mathcal{N}(0, \cpr)$, $m\in\scrM$ and $\mathcal{T}\in B(\scrM)$. We have $m + \mathcal{T}M\sim\mathcal{N}(m, \mathcal{T}\cpr\mathcal{T}^*)$. Moreover, if $\Psi_r$ and $\Psi_r^*$ are a set of linear encoder and decoder defined using reduced $\scrH_{\mu}$-ONBs of rank $r$ as in \cref{eq:projector}, then
    \begin{equation*}
        \Psi_r^*M\sim\mathcal{N}(\bzero, \bI_r) \text{ and } (\calI_{\scrM}-\Psi_r\Psi_r^*)M \perp \Psi_r\Psi_r^*M\,,
    \end{equation*}
    where $\perp$ denotes pairwise independency of random elements.
\end{lemma}
Based on \cref{lemm:gauss_rv}, we derive the following proposition for sampling the surrogate mMALA proposal by splitting the proposal into two parts: a position-dependent one for the $r$-dimensional coefficients in the reduced bases and the pCN proposal in the complementary subspace of $\scrM$, Range($\calI_{\scrM}-\Psi_r\Psi_r^*$).
\begin{proposition}\label[proposition]{prop:splitting}
    Given $m\in\scrM$ and $\triangle t>0$, define two conditional distributions with $s \coloneqq (4-\triangle t)/(4+\triangle t)$:
    \begin{enumerate}[leftmargin=*]
        \item $M^{\dagger}_{\perp}\sim\mathcal{Q}_{\text{\normalfont\ignorespaces pCN}}(m,\cdot)$ following the pCN proposal distribution in \eqref{eq:pCN_proposal} given by
        \begin{equation*}
            M_{\perp}^{\dagger} \coloneqq sm + \sqrt{1-s^2}M\,,\quad M\sim\mu\,.
        \end{equation*}
        \item $\bM_r^{\dagger}\sim\pi_{r}(\cdot|\bmm_r=\Psi_r^*m)$, a $r$-dimensional conditional random vector given by
    \begin{equation}\label{eq:reduced_space_sampling}
    \begin{aligned}
        \bM_r^{\dagger} & \coloneqq \widetilde{\bP_r}\left(\frac{\widetilde{d}_j+s}{\widetilde{d}_j+1}\delta_{jk}\right)\widetilde{\bP_r}^T\bmm_r - \widetilde{\bP_r}\left(\frac{1-s}{\widetilde{d}_j+1}\delta_{jk}\right)\widetilde{\bP_r}^T\widetilde{\bg_r}\\
        &\quad+ \widetilde{\bP_r}\left(\left(\frac{1-s^2}{\widetilde{d}_j+1}\right)^{1/2}\delta_{jk}\right)\boldsymbol{\Xi}\,.
    \end{aligned}
    \end{equation}
    \end{enumerate}
    where $\boldsymbol{\Xi}\sim\mathcal{N}(\bzero, \bI_r)$ is independent of $M$. We have 
    \begin{equation*}
        (\calI_{\scrM}-\Psi_r\Psi_r^*)M^{\dagger}_{\perp} + \Psi_r\bM_r^{\dagger}\sim \widetilde{\calQ_{\mmala}}(m,\cdot).
    \end{equation*}
\end{proposition}
See proofs of \cref{lemm:gauss_rv,prop:splitting} in \cref{app:proof_splitting}. While using an operator surrogate for the position-dependent proposal sampling is novel, the idea of proposal splitting is common in dimension-independent MCMC methods; see, e.g., \citet{cui2015data, cui2016dimension, beskos2017geometric, lan2019adaptive}.

\subsection{Evaluating acceptance probabilities}
The RN derivative $\widetilde{\rho_0}(m_1, m_2;\bw^{\ddagger})$ between the surrogate mMALA proposal $\widetilde{\mathcal{Q}_{\mmala}}(m_1,\meas m_2)$ and the pCN proposal $\mathcal{Q}_{\text{pCN}}(m_1,\meas m_2)$ can be efficiently evaluated using the trained neural network. Due to \cref{prop:splitting},  $\widetilde{\rho_0}$ is constant in Range($\calI_\scrM-\Psi_r\Psi_r^*$) and can be reduced to a function in $\R^r$ denoted as $\widetilde{\rho_{0,r}}(\cdot,\cdot;\bw^{\ddagger}):\R^r\times\R^r\to\R_+$:
\begin{align*}
    \widetilde{\rho_0}(m_1, m_2;\bw^{\ddagger}) &\equiv \widetilde{\rho_{0,r}}(\Psi_r^*m_1, \Psi_r^*m_2;\bw^{\ddagger})\,. && (\text{Reduced density w.r.t.\ pCN})
\end{align*}
The form of $\widetilde{\rho_{0,r}}$ is given by
\begin{align*}
    \widetilde{\rho_{0,r}}(\bmm_1, \bmm_2) &\coloneqq \exp\Big(-\frac{\triangle t}{8}\norm{(\widetilde{d}_j\delta_{jk})\widetilde{\bP_r}^T\bmm_1 - \widetilde{\bP_r}^T\widetilde{\bg_r}}^2_{\left(\widetilde{d}_j+1)^{-1}\delta_{jk}\right)}\\ 
    &\quad+ \frac{\sqrt{\triangle t}}{2}\widehat{\bmm}^T\left(\widetilde{\bH_r}\bmm_1 -\widetilde{\bg_r}\right) -\frac{1}{2} \norm{\widehat{\bmm}}^2_{\widetilde{\bH_r}}\Big) + \prod_{j=1}^r\left(\widetilde{d}_j+1\right)^{1/2}\,,
\end{align*}
where $\widehat{\bmm} \coloneqq (\bmm_2 - s\bmm_1)/\sqrt{1-s^2}$. Here, the reduced ppg $\widetilde{\bg_r}$, the reduced ppGNH $\widetilde{\bH_r}$, and the ppGNH eigendemposition $(\widetilde{d}_j\delta_{jk}, \widetilde{\bP_r})$ are defined in \cref{eq:surrogate_rppg,eq:surrogate_rppgnh,eq:surrogate_eigendecomposition} and evaluated at $\bmm_1 = \Psi_r^*m_1$ through the trained neural network.

For the DA MCMC introduced in \cref{subsec:da}, only the surrogate data misfit enters the first stage transition rate ratio, and, thus, it can also be reduced to $\R^r$:
\begin{subequations}
\begin{align}
\rho^{(1)}(m_1, m_2;\bw^{\ddagger}) &\equiv \rho_{r}^{(1)}\left(\Psi_r^*m_1, \Psi_r^*m_2;\bw^{\ddagger}\right)\,,\\
    \rho_{r}^{(1)}(\bmm_1, \bmm_2) &\coloneqq \exp\left(\widetilde{\Phi_r^{\by}}(\bmm_1)-\widetilde{\Phi_r^{\by}}(\bmm_2)\right) \frac{\widetilde{\rho_{0,r}}(\bmm_2, \bmm_1)}{\widetilde{\rho_{0,r}}(\bmm_1, \bmm_2)}\,.\label{eq:da_dino_trr_1}
\end{align}
\end{subequations}
When combined with the proposal splitting in \cref{prop:splitting}, the first stage in the DA procedure can be performed entirely in reduced coefficient space $\R^r$, and prior sampling can be avoided until entering the second stage due to \cref{lemm:gauss_rv}. In the second stage, the true data misfit evaluated at the full proposal is required to maintain the posterior sampling consistency of MCMC:
\begin{equation}\label{eq:da_dino_trr_2}
    \rho^{(2)}(m_1, m_2;\bw^{\ddagger}) = \frac{\exp(-\widetilde{\Phi_r^{\by}}(\Psi_r^*m_j))\exp(-\Phi^{\by}(m^{\dagger}))}{\exp(-\widetilde{\Phi_r^{\by}}(\Psi_r^*m^{\dagger}))\exp(-\Phi^{\by}(m_j))}\,.
\end{equation}
We note that \cref{eq:da_dino_trr_1} allows for proposal rejection without true data misfit evaluation nor prior sampling. The additional cost reduction in prior sampling due to our choice of reduced basis architecture can be important, for example, when the prior is defined through Whittle--Mat\'ern Gaussian random fields \citep{whittle1954on} that require solving fractional stochastic PDEs to sample. In \cref{alg:mcmc}, we summarize the procedure for our DA geometric MCMC method via a reduced basis neural operator surrogate.

\SetKwInOut{Global}{Provided}
\SetKwInOut{Known}{Known at $m_j$}
\SetKwComment{Comment}{$\triangleright$\ }{}
\begin{algorithm}[!ht]
    \SetAlgoLined
    \KwIn{(i) a trained neural network $\nn(\cdot, \bw^{\ddagger})$, (ii) a pair of parameter encoder $\Psi_r^*$ and decoder $\Psi_r$, (iii) an observable basis $\bV$, and (iv) a step size $\triangle t$.}
    
    \Known{(i) the data misfit value $\Phi^{\by}(m_j)$, (ii) the surrogate reduced data misfit value $\widetilde{\Phi_r^{\by}}(\bmm_j)$, where $\bmm_{r,j}\coloneqq\Psi_r^*m_j$ (iii) the surrogate reduced ppg $\widetilde{\bg_r}(\bmm_{r,j})$, (iv) the surrogate reduced ppGNH eigendecomposition $\left(\widetilde{\bD_r}(\bmm_{r,j}), \widetilde{\bP_r}(\bmm_{r,j})\right)$.}
    
    \KwOut{The next position $m_{j+1}\in\scrM$.}
    Sample $\bxi_r\iid \mathcal{N}(\bzero, \bI_r)$\;
    Compute $\bmm_r^{\dagger}$ using $\bxi_r$ via \cref{eq:reduced_space_sampling}\Comment*[r]{Reduced proposal sampling}
    Evaluate $\nn$ and $\partial_{\bmm_r}\nn$ at $\bmm_r^{\dagger}$\;
    Evaluate $\widetilde{\Phi_r^{\by}}$, $\widetilde{\bg_r}$, $\widetilde{\bD_r}$, and $\widetilde{\bP_r}$ at $\bmm_r^{\dagger}$ \Comment*[r]{$\R^{r\times r}$ Hermitian eigenvalue problem}
    Compute $\alpha^{(1)} = \min\{1, \rho_{r}^{(1)}(\bmm_{r,j}, \bmm_r^{\dagger})\}$ via \cref{eq:da_dino_trr_1}\;
    \eIf{$\alpha^{(1)}<\xi_1$ where $\xi_1\iid\text{\normalfont\ignorespaces Uniform}([0,1])$}{
            \Return $m_j$\Comment*[r]{First stage rejection}
        }{ 
        Sample prior $m_{\perp}^{\dagger}\iid \mu$\Comment*[r]{Prior sampling in second stage}
        Compute $m^{\dagger}=\Psi_r \bmm_r^{\dagger} + m_{\perp}^{\dagger}-\Psi_r\Psi_r^*m_{\perp}^{\dagger}$\Comment*[r]{Assemble the full proposal}
        Evaluate the data misfit $\Phi^{\by}(m^{\dagger})$\Comment*[r]{Model evaluation in second stage}
        Compute $\alpha^{(2)} = \min\{1, \rho^{(2)}(m_j, m^{\dagger})\}$ via \cref{eq:da_dino_trr_2}\\
        \eIf{$\alpha^{(2)}<\xi_2$ where $\xi_2\iid\text{\normalfont\ignorespaces Uniform}([0,1])$}{
            \Return $m_j$\Comment*[r]{Second stage rejection}
        }{
        \Return $m^{\dagger}$\Comment*[r]{Second stage acceptance}
        }
        }
\caption{Markov chain transition of surrogate-driven DA mMALA at the $j$-th position $m_j\in\scrM$}\label[algorithm]{alg:mcmc}
\end{algorithm}

\section{Numerical examples: Baseline, chain diagnostics, efficiency metrics, and software}\label{sec:results_set_up}
 The proposed DINO-driven geometric MCMC method is studied on two PDE-constrained BIPs in \cref{sec:ndr,sec:hyperelastic}. In this section, we briefly introduce baseline MCMC methods to assess the efficiency of our proposed MCMC method. Then, we specify two diagnostics for assessing the quality of Markov chains for posterior sampling. Next, we introduce two metrics that quantify the relative efficiency of two MCMC methods.

\begin{table}[!ht]
    \centering
    {\renewcommand{\arraystretch}{1.5}
    \begin{tabular}{| c |c | c | c | }\hline 
        \makecell{\bf Posterior\\\bf geometry\\ \bf information} & {\bf Name} & \makecell{\bf Gauss--Newton \\\bf Hessian \cref{eq:ppgh} \\\bf (approximation)} & \bf Gradient \cref{eq:ppg} \\\hline
        \multirow{2}{*}{None}&pCN & \xmark & \xmark  \\\cline{2-4}
        & MALA & \xmark & \cmark \\\hline
         \multirow{2}{*}{Fixed}&LA-pCN  & $\calH(m_{\text{MAP}})$ in \cref{eq:MAP,eq:ppgh} & \xmark \\\cline{2-4}
        &DIS-mMALA & $\widehat{\Psi_r^{\text{DIS}}}\widehat{\boldsymbol{\Lambda}_r^{\text{DIS}}}\widehat{\Psi_r^{\text{DIS}}}^*$ in \cref{eq:dis_gnh}& \cmark \\\hline
        \makecell{Position\\-dependent}&mMALA & $\calH$ in \cref{eq:ppgh} & \cmark\\\hline
        
    \end{tabular}
    }
    \caption{A list of baseline dimension-independent MCMC methods used in our numerical examples.}
    \label{tab:mcmc_list}
\end{table}

\begin{table}[!htbp]
    \centering
        {\renewcommand{\arraystretch}{1.5}
    \begin{tabular}{|c |c | c | c | c |}\hline
     \makecell{\bf Operator learning\\
    \bf objective function} &{\bf Name} & \makecell{\bf Reduced \\\bf bases} & \makecell{\bf Delayed\\\bf acceptance} & \makecell{\bf Reference\\\bf method \cref{eq:ridge}} \\\hline
    \multirow{2}{*}{\makecell{ Conventional\\$L^2_{\mu}(\scrM;\scrY)$} } & NO-mMALA &  \multirow{4}{*}{\makecell{DIS\\ $r=200$}} & \xmark & r-mMALA\\\hhline{~-~--}
    & DA-NO-mMALA  & & \cmark & DA-r-mMALA\\\hhline{--~--}
    \multirow{2}{*}{\makecell{ Derivative-informed\\$H^1_{\mu}(\scrM;\scrY)$}} & \makecell{DINO-mMALA} &  &\xmark & r-mMALA \\\hhline{~-~--}
    & \makecell{DA-DINO-mMALA} & &\cmark & DA-r-mMALA\\\hline
    \end{tabular}
    }
        \caption{A list of operator surrogate-based dimension-independent geometric MCMC methods used in our numerical examples. We additionally introduce two reference methods, r-mMALA and DA-r-mMALA, that isolate the effects of reduced basis architecture in the empirical performance of surrogate-driven MCMC.}
    \label{tab:dino_mcmc_list}
\end{table}

\subsection{Baseline and reference MCMC methods}

A table of baseline MCMC methods is listed in \cref{tab:mcmc_list}. The mMALA method can be deduced from the generic mMALA proposal in \cref{eq:mmala} with $\calK(m) = (\calI_{\scrH_{\mu}} + \calH(m))^{-1}\cpr$ as in \cref{eq:ppgh}. This method is the same as $\infty$-mMALA by \cite{beskos2017geometric,lan2019adaptive}. The DIS-mMALA method uses the DIS approximation of $\mathbb{E}_{M\sim\mu}[\calH(m)]$ in \cref{eq:dis_gnh} as a fixed approximation to $\calH$. This method is the same as DR-$\infty$-mMALA by \cite{lan2019adaptive} with a fixed DIS reduced basis, similar to gpCN by \cite{rudolf2018generalized} except that gpCN does not include ppg, similar to LI-Langevin by \cite{cui2016dimension} except that (i) DIS-mMALA do not require prior covariance factorization and (ii) DIS-mMALA has only one step size parameter. The LA-pCN method \citep{pinski2015algroithms, kim2023hippylibmuq} utilizes the Laplace approximation to the posterior, which requires solving the following deterministic inverse problem for the maximum a posteriori probability (MAP) estimate \citep{dashti2013MAP, VillaPetraGhattas21}, denoted by $m_{\text{MAP}}\in\scrH_{\mu}$, to construct a proposal:
\begin{subequations}
\begin{gather}
    m_{\text{MAP}} \coloneqq \argmin_{m\in\scrM} \left(\Phi^{\by}(m) + \frac{1}{2}\norm{m}^2_{\cpr^{-1}}\right)\,,\quad \mu^{\by}\approx \mathcal{N}(m_{\text{MAP}}, \cpo(m_{\text{MAP}}))\,,\label{eq:MAP}\\
    \mathcal{Q}_{\text{LA-pCN}}(m,\cdot) \coloneqq \mathcal{N}\left(m_{\text{MAP}}-s m, (1-s^2)\cpo(m_{\text{MAP}})\right)\,.
\end{gather}
\end{subequations}

In \cref{tab:dino_mcmc_list}, we provide a list of MCMC methods driven by reduced basis neural operators detailed in \cref{sec:surrogate_mcmc}. The mMALA proposal approximated by both the $L^2_{\mu}$-trained NO and $H^1_{\mu}$-trained DINO with and without the DA procedure is studied in our numerical examples, and their efficiency is compared with the baseline methods. 

When computationally feasible, we complement the numerical results of our proposed methods with those of two reference MCMC methods: r-mMALA (reduced mMALA) and DA-r-MALA (reduced mMALA with delayed acceptance). These reference methods are designed to isolate the effects of the reduced basis architecture of DINO on the performance of the proposed MCMC methods. They are defined via the following sample average approximation to the optimal reduced mapping \cref{eq:optimal_reduced_mapping} of the PtO map and its reduced Jacobian:
\begin{subequations}\label{eq:ridge}
\begin{align}
    \bdmc{G}(m)&\approx\sum_{j=1}^{n_{\text{rm}}} \bdmc{G}\left(\Psi_r\Psi_r^*m + (\calI_{\scrM}-\Psi_r\Psi_r^*)m_j)\right)\,,\quad m_j\iid \mu\,,\label{eq:ridge_pto}\\
    \bJ_r(m)&\approx \sum_{j=1}^{n_{\text{rm}}} \bV^*\sder\bdmc{G}\left(\Psi_r\Psi_r^*m + (\calI_{\scrM}-\Psi_r\Psi_r^*)m_j)\right)\Psi_r^*\,,\quad m_j\iid \mu\,.\label{eq:ridge_jac}
\end{align}
\end{subequations}
The r-mMALA reference method replaces the PtO map and its stochastic derivative in the mMALA proposal using \cref{eq:ridge}. The DA-r-mMALA reference method additionally includes the DA procedure using \cref{eq:ridge_pto}. We take $n_{\text{rm}} = 20$ in \cref{sec:ndr}.

\subsection{Markov chain diagnostics}\label{subsec:diagnostic}
We focus on two diagnostics that help us understand the quality of Markov chains generated by MCMC: the multivariate potential scale reduction factor (MPSRF) and the effective sample size percentage (ESS\%). In this subsection, we assume access to $n_{c}$ number of independent Markov chains generated by the same MCMC method targeting the same posterior. Each chain has $n_{s}$ samples (after burn-in), denoted by $\{\{m_{j,k}\}_{j=1}^{n_s}\}_{k=1}^{n_{c}}$.

\subsubsection{Wasserstein multivariate potential scale reduction factor}\label{subsubsec:mpsrf}
    The MPSRF \citep{brooks1998general} is a diagnostic for the convergence of MCMC. It compares the multi-chain mean of the empirical covariance within each chain, denoted as $\widehat{\calW}_s\in B(\scrM)$, and the empirical covariance across all chains, denoted as $\widehat{\calV}_s\in B(\scrM)$. They are given by
    \begin{subequations}\label{eq:covariance_estimator}
    \begin{align}
        \widehat{\calW}_s &\coloneqq \frac{1}{n_{c}(n_s-1)}\sum_{k=1}^{n_{c}}\sum_{j=1}^{n_s}\left\langle m_{j,k}- \overline{m}_k, \cdot\right\rangle_{\scrM}\left(m_{j,k} - \overline{m}_k\right)\,,\\
        \widehat{\calV}_s &\coloneqq \frac{n_{s}-1}{n_{s}}\widehat{\calW}_s + \frac{n_{c}+1}{n_{c}(n_{c}-1)}\sum_{k=1}^{n_{c}}\left\langle \overline{m_k}- \overline{m}, \cdot\right\rangle_{\scrM}\left(\overline{m_k} - \overline{m}\right)\,,
    \end{align}
    \end{subequations}
    where $\overline{m_k}$ is the mean of samples in each chain labeled by $k=1,\dots, n_{\text{chain}}$ and $\overline{m}$ is the mean of samples in all chains. We expect the difference in $\widehat{\calW}_s$ and $\widehat{\calV}_s\in B(\scrM)$ to be small as the chains become longer. 
    
In this work, we propose to use the 2-Wasserstein distance for Gaussian measures \citep{dowson1982frechet} to compare the distance between the two covariance operators. We refer to this diagnostic as Wasserstein MPSRF:
\begin{align*}
\begin{split}
    \widehat{R}_{w} &= \text{Wass}_2\left(\mathcal{N}(0, \widehat{\calW}_s), \mathcal{N}(0, \widehat{\calV}_s)\right)\\
    &= \text{Tr}_{\scrM}\left(\widehat{\calW}_s + \widehat{\calV}_s -2 (\widehat{\calW}_s^{1/2}\widehat{\calV}_s\widehat{\calW}_s^{1/2})^{1/2}\right)\,. 
\end{split}&&& (\text{Wasserstein MPSRF})
\end{align*}
The faster $\widehat{R}_w$ decays as a function of the chain length, the faster the pool of chains converges to their stationary distribution (i.e., faster mixing time). While the Wasserstein MPSRF is uncommon for MCMC convergence diagnostics, we find it useful for comparing the performance of MCMC methods in function spaces; see \cref{app:step_size} for additional discussions on this diagnostic.

\subsubsection{Effective sample size percentage distribution}\label{subsubsec:ess}
The ESS\% \citep{gelman2014bayesian} is the estimated percentage of independent samples from a pool of Markov chains. When each sample resides in 1D, such a metric is estimated using the autocorrelation function, denoted by $\text{AC}(t, k)$, of samples that are $t$ positions apart within the $k$-th Markov chain:
\begin{subequations}\label{eq:ess}
\begin{align}
    \text{ESS}\% &= \frac{1}{1 + 2\sum_{t=1}^{2n'+1}\text{MAC}(t)}\,, &&& (\text{Effective sample size percentage})\\
    \text{MACT}(t) &= 1 - \cfrac{\widehat{w}_s-\frac{1}{n_{c}}\sum_{k=1}^{n_{c}}\text{AC}(t, k)}{\widehat{v}_s}\,, &&&(\text{Multichain autocorrelation time})
\end{align}
\end{subequations}
where MACT is the multi-chain AC estimate, $\widehat{w}_s$ and $\widehat{v}_s$ is the 1D version of \cref{eq:covariance_estimator}. The index $n'\in\Z_+$ is chosen to be the largest integer so that the sum of MACT evaluated at neighboring positions is positive.

The simulated MCMC samples $m^h_{j,k}$ belong to a discretized function space $\scrM^h$ of dimension $d_m$. Following \cite{beskos2017geometric} and \cite{lan2019adaptive}, we estimate this 1D diagnostic metric for each degree of freedom (DoF) in the discretized space. This leads to a distribution of $d_m$ ESS\% estimates, visualized using a violin plot.

\subsection{Comparing efficiency of MCMC methods} In this subsection, we introduce two metrics that measure the relative efficiency of a pair of MCMC methods: effective sampling speedup and total effective sampling speedup.

\subsubsection{Effective sampling speedup}

The effective sampling speed quantifies the efficiency of an MCMC method regarding its speed of effective sample generation:
\begin{equation}\label{eq:effectiv_sample_speed}
    \text{Effective sampling speed} = \cfrac{\text{Median(ESS\%)}}{\text{Cost of $100$ Markov chain samples}}\,.
\end{equation}
Instead of directly computing this quantity, we use it to compare the relative efficiency of posterior sampling for different pairs of MCMC methods. The relative efficiency is measured by \textit{the effective sampling speedup}, or speedup for short. It is given by the ratio between the effective sampling speed of the two methods.

\subsubsection{Total effective sampling speedup}

Training sample generation is an important part of the total computational cost for DINO-accelerated geometric MCMC methods, especially when the PtO map is expensive to evaluate. The cost of training sample generation is a fixed offline cost while the cost of posterior sampling scales with the number of MCMC samples. To incorporate both the offline training cost and the online MCMC cost, we introduce an efficiency metric called \textit{the total effective sampling speed} as a function of the effective sample size $n_{\text{ess}}$ required for an MCMC run:
\begin{gather}\label{eq:total_speedup}
    \text{Total effective sampling speed}(n_{\text{ess}}) = \cfrac{n_{\text{ess}}}{\text{Total cost}(n_{\text{ess}})}\,,\\
    \text{Total cost}(n_{\text{ess}}) = \text{Offline cost } + \text{ Cost of } 100 \text{ Markov chain samples }\times\cfrac{n_{\text{ess}}}{\text{ Median(ESS\%)}}\,.\nonumber
\end{gather}

For DINO-mMALA and DA-DINO-mMALA, the offline cost is the sum of training sample generation and neural network training cost. The total effective sampling speed converges to the effective sampling speed when the offline cost is negligible compared to the cost of the MCMC run, i.e., $n_{\text{ess}}\to\infty$. Instead of directly computing this quantity, we use it to compare the relative efficiency of posterior sampling for different pairs of MCMC methods. The relative efficiency is measured by \textit{the total effective sampling speed}, or total speedup for short. It is given by the ratio between the total effective sampling speed of the two methods at a given $n_{\text{ess}}$.

\begin{remark}
In our numerical examples, we do not include the computational cost of step size tuning and burn-in in the total cost. We found that these costs can be significantly reduced if the Markov chain initial position is sampled from the Laplace approximation of the posterior instead of the prior. See \cref{app:step_size} for the step size tuning and chain initialization procedure.
\end{remark}

\subsection{Software}

Our numerical examples are implemented through (i) \texttt{FEniCS} \citep{alnaes2015fenics, logg2012automated} for finite element discretization, solves, and symbolic differentiation of PDE residual operators, (ii) \texttt{hIPPYlib} \citep{VillaPetraGhattas2018, VillaPetraGhattas21} for all components related to inverse problems, e.g. the prior, adjoint solves, Laplace approximation, and some baseline MCMC methods (LA-pCN, MALA, and pCN), (iii) \texttt{hIPPYflow} \citep{hippyflow} for reduced basis estimation and training sample generation, and (iv) \texttt{dino} \citep{dino} for derivative-informed operator learning.
\section{Numerical example: Coefficient inversion for a nonlinear diffusion--reaction PDE}\label{sec:ndr}
We consider the following steady-state nonlinear diffusion--reaction equation in the unit square:
\begin{subequations}
    \begin{align*}
    -\nabla\cdot \exp(m(\bx)) \nabla u(\bx) + u(\bx)^3 &= 0\,, && \bx\in(0,1)^2\,;\\
    u(\bx) &= 0\,, && \bx\in\Gamma_{\text{bottom}}\,;\\
    u(\bx) & = 1\,, && \bx\in\Gamma_{\text{top}}\,;\\
    \exp(m(\bx))\nabla u(\bx)\cdot \bn &= 0\,, && \bx\in\Gamma_{\text{left}}\cup\Gamma_{\text{right}}\,;
\end{align*}
\end{subequations}
where $\Gamma_{\text{bottom}}$, $\Gamma_{\text{top}}$,  $\Gamma_{\text{left}}$, and $\Gamma_{\text{right}}$ are the boundaries the unit square. The inverse problem is to invert for the log-coefficient field $m$ given noise-corrupted discrete observations of the PDE state variable $u$ at a set of spatial positions. 

This section is organized as follows. We introduce the prior, the PtO map, and the setting for Bayesian inversion in \cref{subsec:ndr_prior,subsec:ndr_pto,subsec:ndr_bip}. Then, we present the specifications and results on training operator surrogates with conventional $L^2_{\mu}$ operator learning and the derivative-informed $H^1_{\mu}$ operator learning in \cref{subsec:ndr_surrogate}. Next, we showcase and analyze MCMC results in \cref{subsec:ndr_results}. In \cref{subsubsec:ndr_baseline}, we discuss results on the baseline methods listed in \cref{tab:mcmc_list}. In \cref{subsubsec:ndr_proposal}, we discuss results on NO-mMALA and DINO-mMALA listed in \cref{tab:dino_mcmc_list} to understand the quality of surrogate mMALA proposals. In \cref{subsubsec:ndr_da}, we discuss results on DA-NO-mMALA and DA-DINO-mMALA listed in \cref{tab:dino_mcmc_list}.
\subsection{The prior distribution}\label{subsec:ndr_prior}
We consider the following parameter space and prior distribution
\begin{align*}
    \scrM &\coloneqq L^2((0,1)^2)\,, && (\text{Parameter space})\\
    \mu&\coloneqq\mathcal{N}(0, (-\gamma\lap + \delta\calI_{\scrM})^{-2})\,, &&(\text{The prior distribution})
\end{align*}
where $-\lap: H^1((0,1)^2)\to H^1((0,1)^2)'$ is the weak Laplace operator with a Robin boundary condition for eliminating boundary effects \cite[Equation 37]{VillaPetraGhattas21}. The constants $\gamma,\delta\in\R_+$ are set to $\gamma = 0.03$ and $\delta = 3.33$, which approximately leads to a pointwise variance of $9$ and spatial correlation length of $0.1$. We approximate $\scrM$ using a finite element space $\scrM^h$ constructed by linear triangular elements with $1681$ DoFs. A visualization of prior samples is provided in \cref{fig:ndr_samples}.

\subsection{The parameter-to-observable map}\label{subsec:ndr_pto}
We consider a symmetric variational formulation of the PDE problem, and define the following Hilbert spaces following the notation in \cref{subsec:data_generation}:
\begin{align*}
    \scrU &\coloneqq \left\{u\in H^1((0,1)^2)\;\Big\vert\; u|_{\Gamma_t} = 0\wedge u|_{\Gamma_b} = 0\right\}\,; && (\text{State space})\\
    \scrV &\coloneqq \scrU'\,, && (\text{Residual space})
\end{align*}
where $\scrU'$ denotes the dual space of $\scrU$. To enforce the inhomogeneous Dirichlet boundary condition, we decompose the PDE solution $u$ into $u=u_0 + \bx^T\be_2$, where $\be_2 = \begin{bmatrix} 0 & 1 \end{bmatrix}^T$ and $u_0\in\scrU$ is the PDE state with the homogenous Dirichlet boundary condition. 

The residual operator for this PDE problem, denoted as $\calR:\scrU\times\scrM\to\scrV$, can be defined by its action on an arbitrary test function $p\in \scrU$
\begin{equation}\label{eq:ndr_residual}
\begin{aligned}
    \left\langle\calR(u_0, m), p\right\rangle_{\scrU'\times\scrU} &\coloneqq \int_{(0,1)^2} \Big(\exp(m(\bx)) \left(\grad u_0(\bx)+ \be_2\right) \cdot \grad p(\bx) \\
    &\quad+ \left(u_0(\bx) + \bx^T\be_2\right)^3p(\bx)\Big)\dd\bx\,.
\end{aligned}
\end{equation}
The effective PDE solution operator is defined as a nonlinear parameter-to-state map $\mathcal{F}:\scrM\ni m\mapsto u_0\in\scrU$ where $\mathcal{R}(u_0,m) = 0$. We approximate $\scrU$ using a finite element space $\scrU^h$ constructed by quadratic triangular elements with $3362$ DoFs. Evaluating the discretized PDE solution operator involves solving the discretized residual norm minimization problem via the Newton--Raphson method in $\scrU^h$.

We define the observation operator $\bdmc{O}$ using $25$ randomly-sampled discrete interior points $\{\bx_{\text{obs}}^{(j)}\}_{j=1}^{d_y}$ with $d_y = 25$:
\begin{align}\label{eq:observation_operator}
    \bdmc{O}(u_0) = \begin{bmatrix}\displaystyle
        \int_{\bB_\epsilon(\bx^{(1)}_{\text{obs}})} u(\bx)\dd\bx & \dots &\displaystyle\int_{\bB_\epsilon(\bx^{(25)}_{\text{obs}})} u(\bx)\dd\bx
    \end{bmatrix}^T\,, &&& (\text{Observation operator})
\end{align}
where $\bB_{\epsilon}(\bx)\subset(0,1)^2$ is a ball around $\bx$ with a small radius $\epsilon>0$. The PtO map is $\bdmc{G}\coloneqq \bdmc{O}\circ\mathcal{F}$. Samples of the PtO map are visualized in \cref{fig:ndr_samples}.

\begin{figure}[tpb]
    \centering
    \renewcommand{\arraystretch}{0.1}
    \begin{tabular}{c c c}
    \bf Prior samples & \bf PDE solutions & \bf Observables\\
        \includegraphics[width=0.25\linewidth]{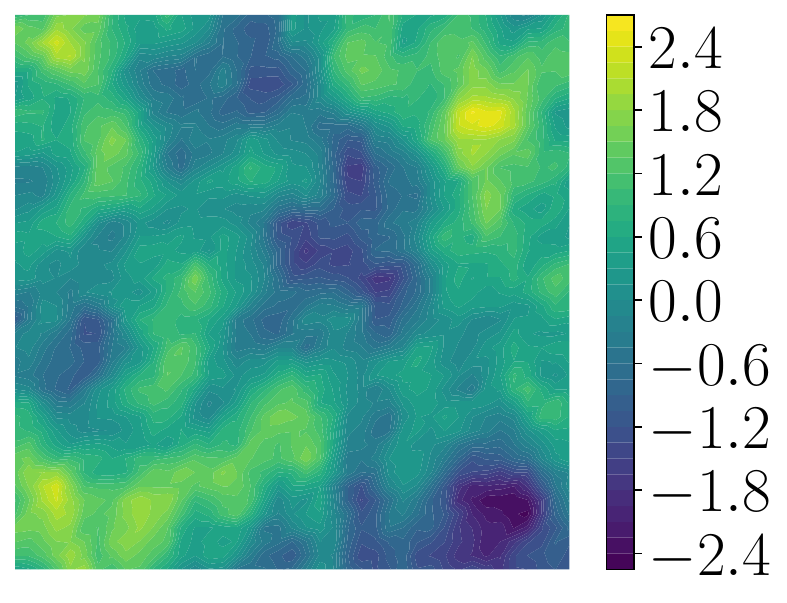}  & \includegraphics[width=0.25\linewidth]{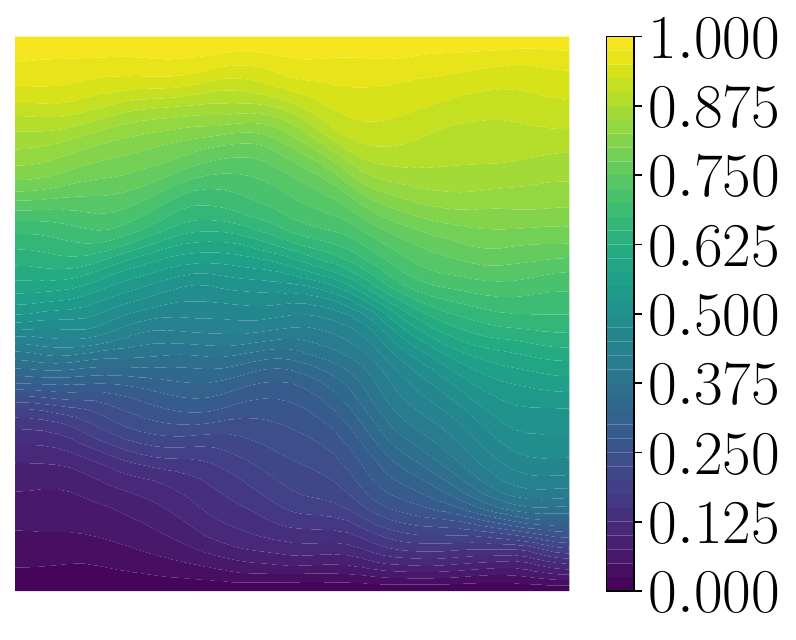} & \includegraphics[width=0.25\linewidth]{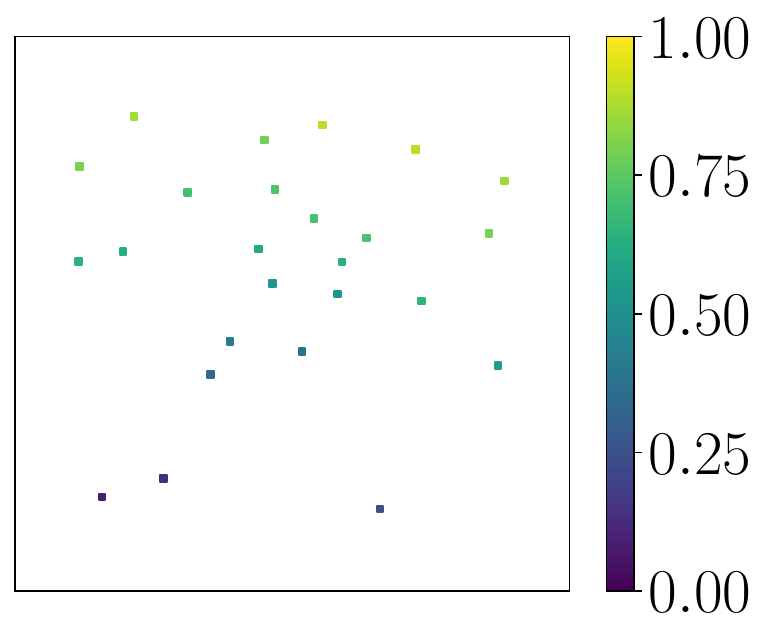} \\
        \includegraphics[width=0.25\linewidth]{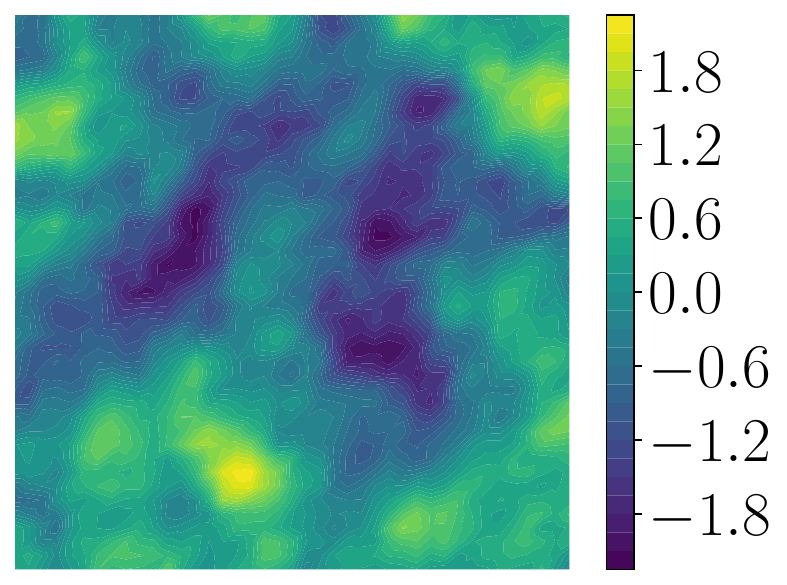} & \includegraphics[width=0.25\linewidth]{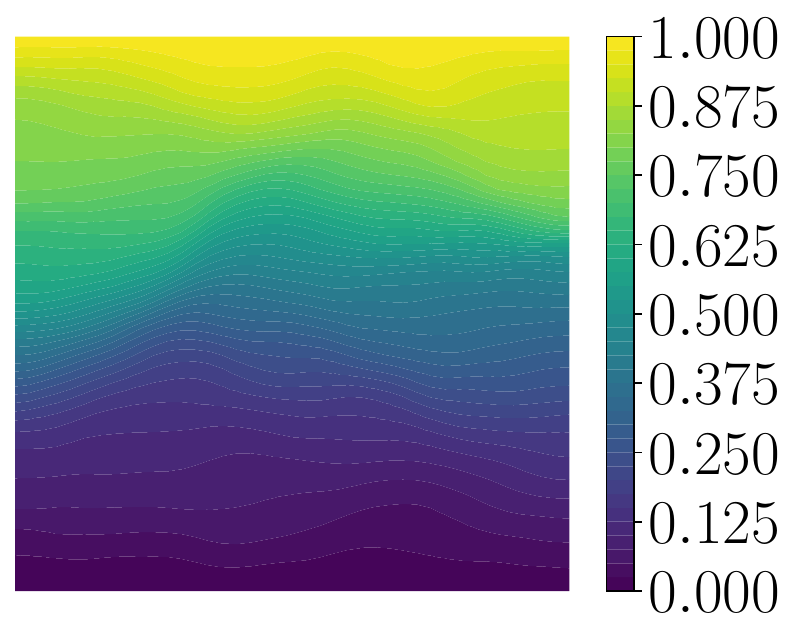} & \includegraphics[width=0.25\linewidth]{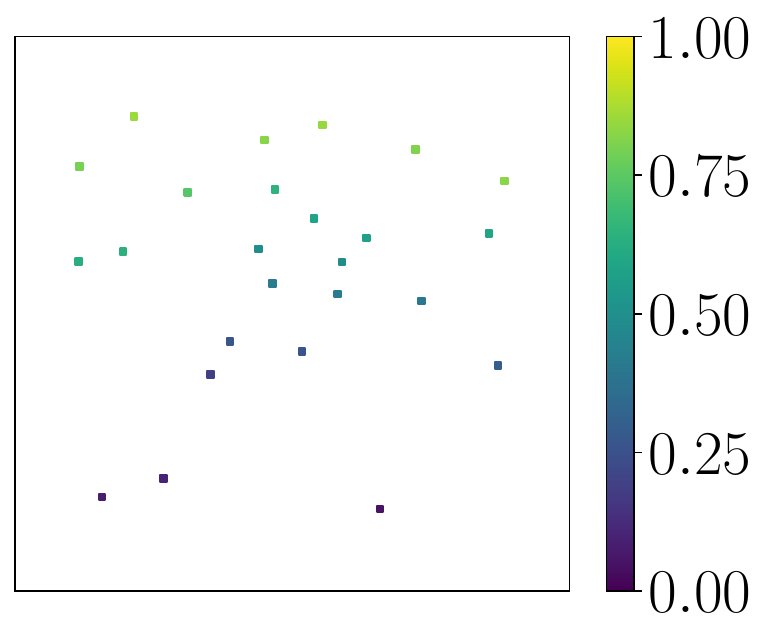} \\
    \end{tabular}
\renewcommand{\arraystretch}{1.0}
    \caption{Visualizations of prior samples ($1681$ DoFs), PDE solutions ($3362$ DoFs), and predicted observables ($\R^{25}$) for coefficient inversion in a nonlinear diffusion--reaction PDE.}
    \label{fig:ndr_samples}
\end{figure}

\subsection{Bayesian inverse problem settings}\label{subsec:ndr_bip}
We generate synthetic data for our BIP using an out-of-distribution piecewise-constant parameter field, following the examples from, e.g., \cite{cui2016dimension,lan2019adaptive}. The model-predicted observable at the synthetic parameter field is then corrupted with $2\%$ additive white noise, which leads to a noise covariance matrix of identity scaled by $v_n=1.7\times 10^{-4}$:
\begin{align*}
    \pi_n=\mathcal{N}(\bzero, v_n\bI)\,. &&(\text{Noise distribution})
\end{align*}
The synthetic data, its generating parameter and PDE solution, and the MAP estimate are visualized in \cref{fig:ndr_settings}.

\begin{figure}[tpb]
    \centering
    \addtolength{\tabcolsep}{-5pt}
    \begin{tabular}{c c c c}
        \makecell{\bf Synthetic Parameter} & \makecell{\bf PDE solution}  & {\bf Observed data} & {\bf MAP estimate} \\
        \includegraphics[width=0.24\linewidth]{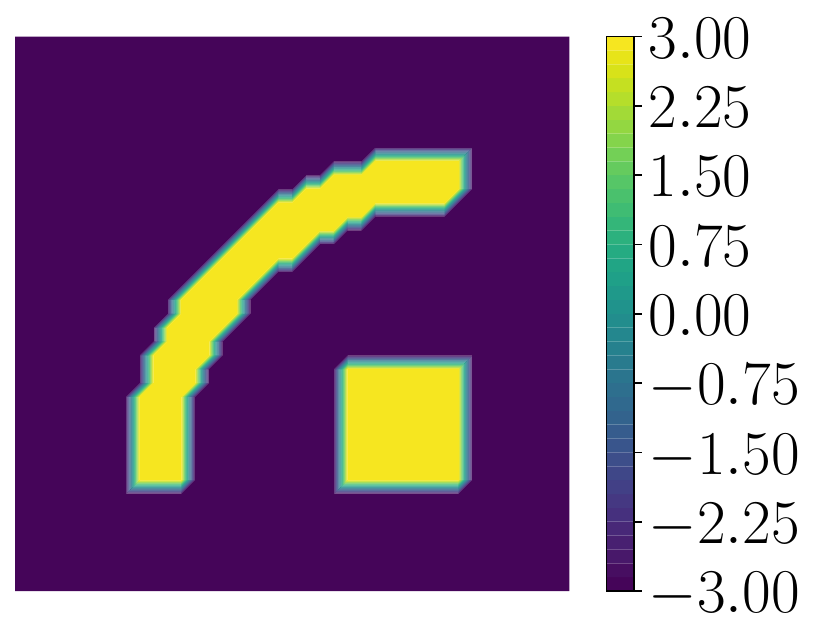} & \includegraphics[width = 0.24\linewidth]{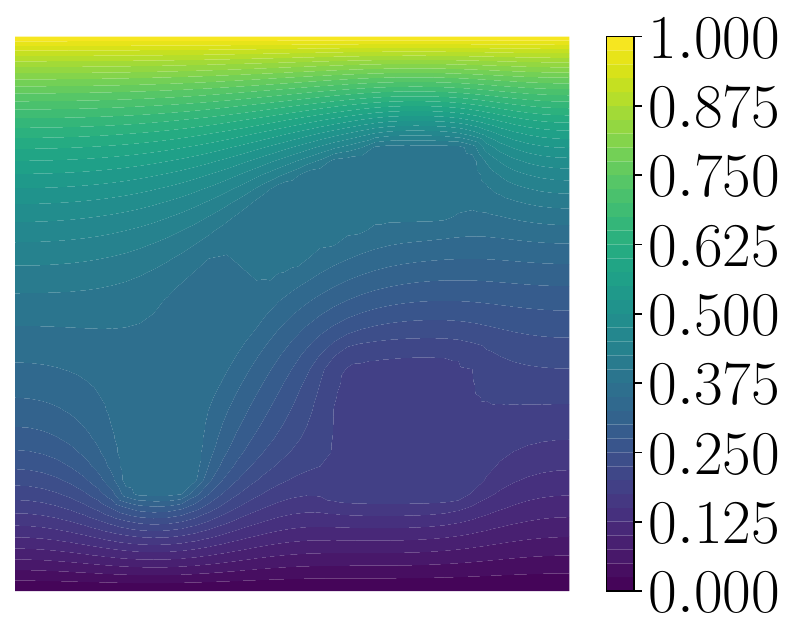} & \includegraphics[width = 0.24\linewidth]{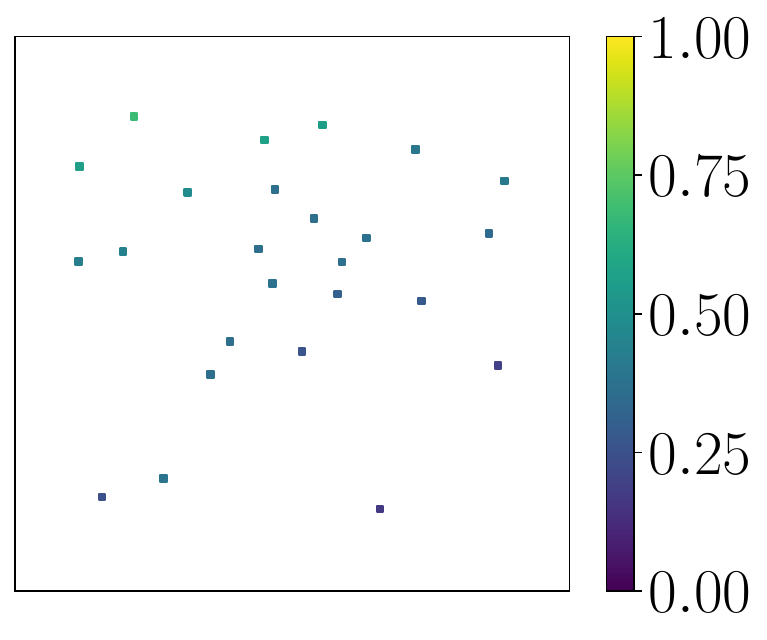} & \includegraphics[width = 0.24\linewidth]{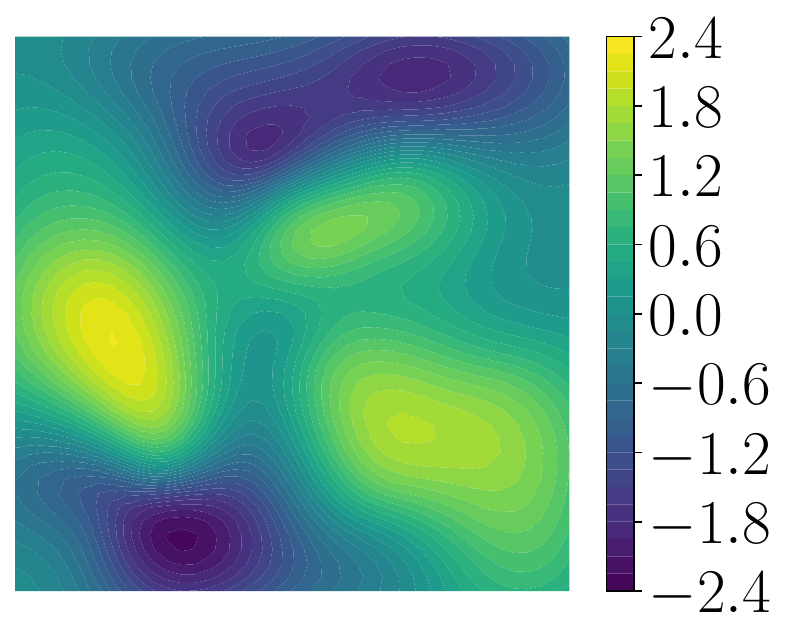} \\
    \end{tabular}
    \addtolength{\tabcolsep}{5pt}
    \caption{Visualization of the BIP setting and the MAP estimate for coefficient inversion in a nonlinear diffusion--reaction PDE.}
    \label{fig:ndr_settings}
\end{figure}

\subsection{Neural operator surrogates}\label{subsec:ndr_surrogate}
We follow the procedure described in \cref{subsec:data_generation} for generating samples for neural network training and testing. Recall that each training sample consists of an i.i.d.\ random parameter field, a model-predicted observable coefficient vector, and a reduced Jacobian matrix. We compute DIS reduced bases of dimension $r=200$ using $n_{\text{DIS}} = 1000$ of the generated samples as specified in \cref{eq:mc_ppgnh}. In particular, the full Jacobian matrix is generated instead of the reduced Jacobian matrix for the first $1000$ samples. Then, a generalized eigenvalue problem is solved to compute the DIS reduced bases \citep{VillaPetraGhattas21, hippyflow}. Selected DIS basis functions are visualized in \cref{fig:ndr_basis}. Forming reduced Jacobian matrices via columns using a direct solver with reused factorization takes $25\%$ of the computing time for solving the nonlinear PDEs using a direct solver, estimated on average over sample generation.

We employ a dense neural network with six hidden layers, each with $400$ hidden neurons and a GELU activation function. The neural network is trained using either the conventional $L^2_{\mu}$ or derivative-informed $H^1_{\mu}$ operator learning objective. For each training method, we use a varying number of training samples $n_t$ for the loss function specified in \cref{eq:empirical_loss_l2,eq:empirical_loss_h1}, with $n_t = 125$, $250$, $\dots$, $16000$. 

The generalization errors of the observable vector prediction and the reduced Jacobian matrix prediction are estimated using $5000$ testing samples. The two types of errors are measured using the relative $L^2_{\mu}$ error defined as follows:
\begin{subequations}\label{eq:generalization_error}
\begin{align}
    \calE_{\text{Obs}}(\bdmc{G}, \widetilde{\bdmc{G}})&\coloneqq \sqrt{\mathbb{E}_{M\sim\mu}\left[\frac{\norm{\bdmc{G}(M)-\widetilde{\bdmc{G}}(M)}^2_{\bC_n^{-1}}}{\norm{\bdmc{G}(M)}^2_{\bC_n^{-1}}}\right]}\,,\\
    \calE_{\text{Jac}}(\bJ_r, \widetilde{\bJ_r}) &\coloneqq \sqrt{\mathbb{E}_{M\sim\mu}\left[\frac{\norm{\bJ_r(M)-\widetilde{\bJ_r}(M)}^2_{F}}{\norm{\bJ_r(M)}^2_{F}}\right]}\,.
\end{align}
\end{subequations}
The generalization accuracy is defined by $(1-\calE)\times 100\%$. In \cref{fig:ndr_accuracy}, we plot the estimated errors as a function of training sample generation cost, measured relative to the averaged cost of one nonlinear PDE solve. For $L^2_{\mu}$-trained neural operators (NOs), we discount the cost of forming reduced Jacobian matrices; thus, its relative cost is the same as the number of training samples. The error plot includes the relative cost of forming reduced Jacobian matrices for derivative-informed $H^1_{\mu}$ operator learning. Additionally, we provide generalization accuracy values on the error plot. 
    \begin{figure}[!h]
        \centering
        \includegraphics[width=0.49\linewidth]{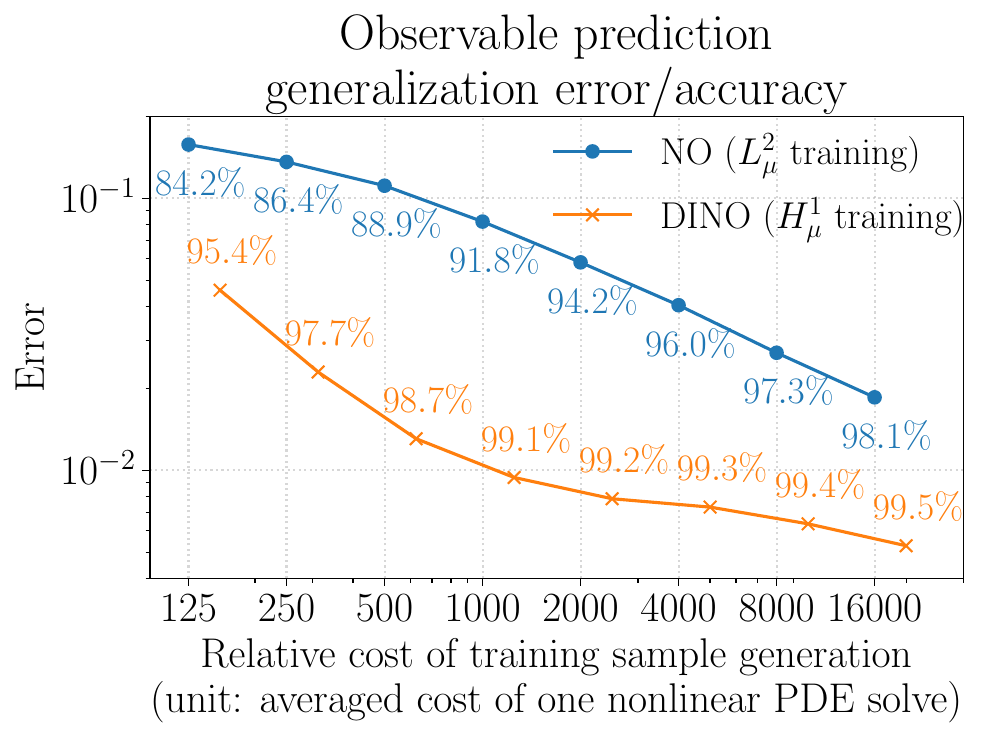}
        \includegraphics[width = 0.49\linewidth]{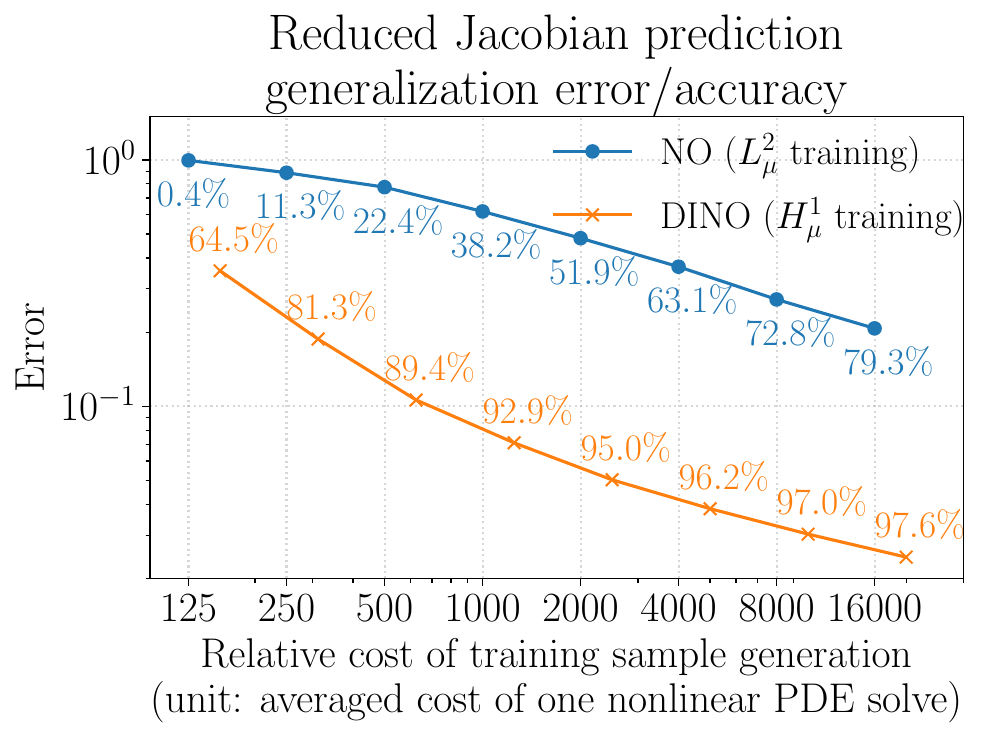}
        \caption{The generalization error and accuracy \cref{eq:generalization_error} for predicting the observable vector and the reduced Jacobian matrix via $L^2_{\mu}$-trained neural operators and $H^1_{\mu}$-trained DINOs for coefficient inversion in a nonlinear diffusion--reaction PDE. The error is plotted as a function of training sample generation cost, measured relative to the averaged cost of one nonlinear PDE solve.}
        \label{fig:ndr_accuracy}
    \end{figure}
    
The plot shows that the derivative-informed $H^1_{\mu}$ operator learning significantly improves the quality of the surrogate at the same training sample generation cost compared to the conventional $L^2_{\mu}$ operator learning. Here is a list of important takeaways from this plot:

\begin{itemize}[leftmargin=*]
    \item To achieve the same generalization accuracy for predicting the observables or the reduced Jacobian matrix, the derivative-informed $H^1_{\mu}$ operator learning is at least 25 times more efficient than the conventional $L^2_{\mu}$ operator learning measured by training sample generation cost.
    \item In the small training sample size regime, e.g., $n_t<1000$, derivative-informed $H^1_{\mu}$ operator learning leads to a much higher generalization error reduction rate for both observable prediction and reduced Jacobian prediction.
    \item To achieve a similar efficiency in posterior sampling compared to mMALA, the operator surrogate needs an estimated $90\%$ generalization accuracy in reduced Jacobian prediction (see \cref{fig:ndr_total_speedup}). The conventional $L^2_{\mu}$ operator learning may at least (estimated via extrapolation) demand the cost of around $116000$ nonlinear PDE solves, while derivative-informed $H^1_{\mu}$ operator learning requires the cost of around $700$ nonlinear PDE solves---two orders of magnitude difference in computational cost. 
\end{itemize} 
These numerical results indicate that the derivative-informed $H^1_{\mu}$ operator learning provides a much superior cost-accuracy trade-off compared to the conventional $L^2_{\mu}$ operator learning. The superiority of derivative-informed learning is more pronounced for large-scale PDE systems since one typically cannot afford to solve these systems at a large number of parameter samples. The superiority of derivative-informed $H^1_{\mu}$ operator learning is decisive when one expects the trained operator to possess an accurate derivative with respect to a high or infinite-dimensional input. 

\subsection{MCMC results}\label{subsec:ndr_results}

We present numerical results on the efficiency of DINO-mMALA and DA-DINO-mMALA methods compared to the baseline and reference MCMC methods. For each method, we collect $n_c = 10$ Markov chains, each with $n_s = 19000$ samples. The step size parameter $\triangle t$ and initialization are chosen by the procedure detailed in \cref{app:step_size}. The statistics of the MCMC runs and posterior visualization are provided in \cref{sec:supplementary}.

\subsubsection{The baseline MCMC methods}\label{subsubsec:ndr_baseline}

The diagnostics for the baseline MCMC methods in \cref{tab:mcmc_list} are visualized in \cref{fig:ndr_baseline}. The diagnostics show that mMALA produces Markov chains with the most effective posterior samples and the fastest mixing speed among the baseline methods. When comparing methods using the same type of posterior geometry information (see \cref{tab:mcmc_list}), MALA is inferior to pCN, and DIS-mMALA is inferior to LA-pCN, even though both MALA and DIS-mMALA include gradient information in their proposal distributions. This result shows the importance of accurate posterior local geometry information (i.e., data misfit Hessians) in MCMC proposal design, as including data misfit gradients alone may compromise the chain quality.

Based on a comparison of the median of ESS\%, mMALA outperformed pCN and LA-pCN, yielding 15.7 and 2.3 times more effective samples, respectively. Yet, each MCMC sample generated by mMALA incurs approximately 2.3 times\footnote{Major contributing factors to the extra cost at each sample $m$ can be roughly decomposed into (1) forming a discretized Jacobian matrix $\bJ^h(m)$ via adjoint solves, (2) solving 
the eigenvalue problem for the operator $\bJ(m)\bJ(m)^T\in\R^{25\times 25}$, and (3) forming the decoder via actions of the prior covariance operator on the encoder.} higher computational costs than pCN and LA-pCN. Consequently, LA-pCN and mMALA achieve equivalent speeds in generating effective posterior samples. The effective sampling speedups of mMALA against other baseline methods are provided in \cref{tab:ndr_cost_per_es}.
\begin{table}[tpb]
    \centering
    \begin{tabular}{|c|c|c|c|c|c|}\hline
        \diagbox[width=10em]{\bf Baseline}{\bf Speedup}    & mMALA & \makecell{DINO-\\mMALA\\$n_t=2000$} & \makecell{DINO-\\mMALA\\$n_t=16000$} & \makecell{DA-DINO-\\mMALA\\$n_t=2000$} & \makecell{DA-DINO-\\mMALA\\$n_t=16000$}    \\\hline
        pCN & 6.8 & 11.9  & 14.3 &26.5 & 59.8 \\\hline
        MALA & 9.4 & 16.5 & 19.6 & 30.6 & 82.7  \\\hline
        LA-pCN & 1 & 1.8 & 2.1 & 3.9 & 8.8 \\\hline
        DIS-mMALA & 5.1 & 8.9  & 10.7 & 19.9 & 44.9 \\\hline
        \makecell{NO-mMALA\\ $n_t=16000$} & 4.8 & 8.4 & 10.1 & 18.7 & 42.2\\\hline
        \makecell{DA-NO-mMALA \\ $n_t=16000$} & 5 & 9 & 10.5 & 19.5 & 44 \\\hline
    \end{tabular}
    \caption{The effective sampling speedup of mMALA, DINO-mMALA, and DA-DINO-mMALA against other baseline MCMC methods for coefficient inversion in a nonlinear diffusion--reaction PDE. The speedup measures the relative speed of generating effective samples for an MCMC method compared against another MCMC method; see \cref{eq:effectiv_sample_speed}.}
    \label{tab:ndr_cost_per_es}
\end{table}

    \begin{figure}[tpb]
        \centering
        \includegraphics[align=c, width=0.4\linewidth]{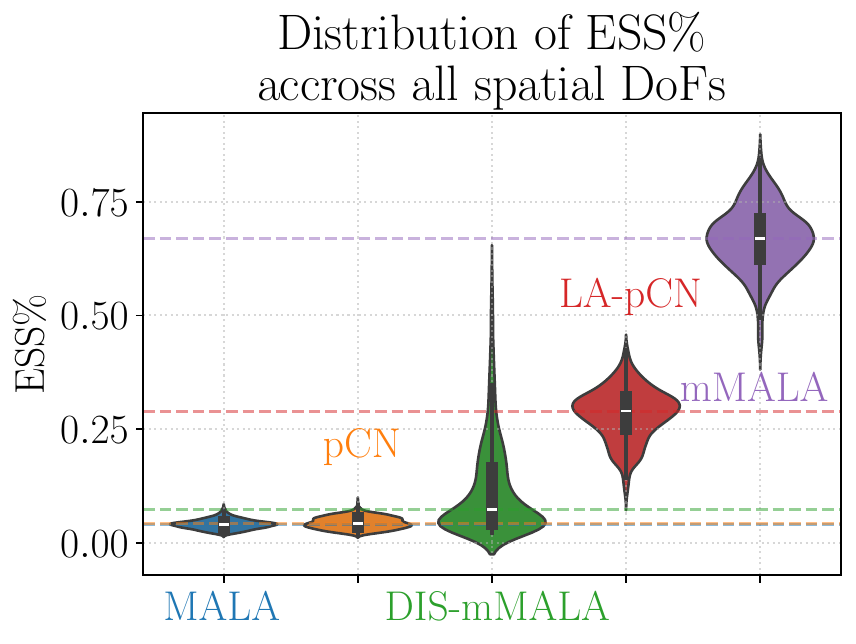}\includegraphics[align=c, width=0.4\linewidth]{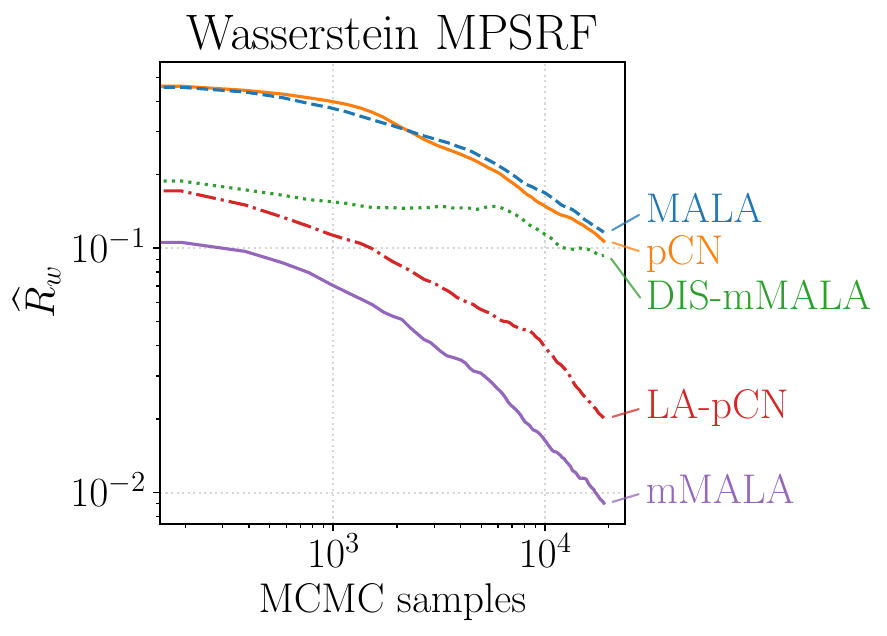}\\
        \caption{Visualization of the diagnostics (see \cref{subsec:diagnostic}) of MCMC chains generated by baseline methods listed in \cref{tab:mcmc_list} for coefficient inversion in a nonlinear diffusion--reaction PDE. (\textit{left}) The violin plot of the ESS\% distributions over $1{,}681$ spatial DoFs for the discretized parameter space. (\textit{left}) The Wasserstein MPSRF as a function of the Markov chain position.}
        \label{fig:ndr_baseline}
    \end{figure}

\subsubsection{Geometric MCMC with surrogate proposals}\label{subsubsec:ndr_proposal}
    For this part of the numerical results, we focus on the quality of the operator surrogate proposal by switching off the DA procedure and using the model-predicted data misfit to compute acceptance probability during MCMC runs. In \cref{fig:dino-mMALA}, we visualize the diagnostics of Markov chains generated by NO-mMALA, DINO-mMALA, r-mMALA, and baseline methods. 
    
    The diagnostics show that DINO-mMALA at $n_t = 2000$ surpasses LA-pCN regarding the effective sample size and mixing speed of MCMC chains. By $n_t=16000$, the ESS\% of DINO-mMALA is close to the baseline mMALA. Measured by the median ESS\%, DINO-mMALA retains $91\%$ of mMALA's effective sample size despite the presence of errors specified in \cref{prop:dis_error}. Estimated by the median ESS\% of r-mMALA, the basis truncation and sampling error accounts for $13\%$ of this reduction in ESS\%. Accounting for the extra computational cost of mMALA at each MCMC sample, DINO-mMALA at $n_t = 16000$ generates effective samples more than twice as fast as mMALA. See the speedups of DINO-mMALA against other methods in \cref{tab:ndr_cost_per_es}. These results suggest that DINO-predicted posterior local geometry is sufficiently accurate for enhancing the trade--off between chain quality and computational cost in geometric MCMC.

    The diagnostics of chains generated by NO-mMALA at $n_t=16000$ implies that MCMC driven by $L^2_{\mu}$-trained NOs lead to low posterior sampling efficiency. While the median ESS\% of NO-mMALA is 1.4 times that of pCN, it is only $11\%$ of the median ESS\% of DINO-mMALA at $n_t=16000$. Recall that $L^2_{\mu}$-trained NO at $n_t=16000$ has an estimated $98.1\%$ and $79.3\%$ generalization accuracy in observable and reduced Jacobian prediction; see \cref{fig:ndr_accuracy}. However, to surpass LA-pCN in terms of median ESS\%, the reduced Jaocbian prediction accuracy needs to be around $90\%$ (i.e., $n_t\approx 1000$ for $H^1_{\mu}$-trained DINO), which would require $L^2_{\mu}$-trained NO at least $n_t = 116000$ to achieve; see comments in \cref{subsec:ndr_surrogate}.
    
        \begin{figure}[tpb]
        \centering
        \includegraphics[align=c, width=0.55\linewidth]{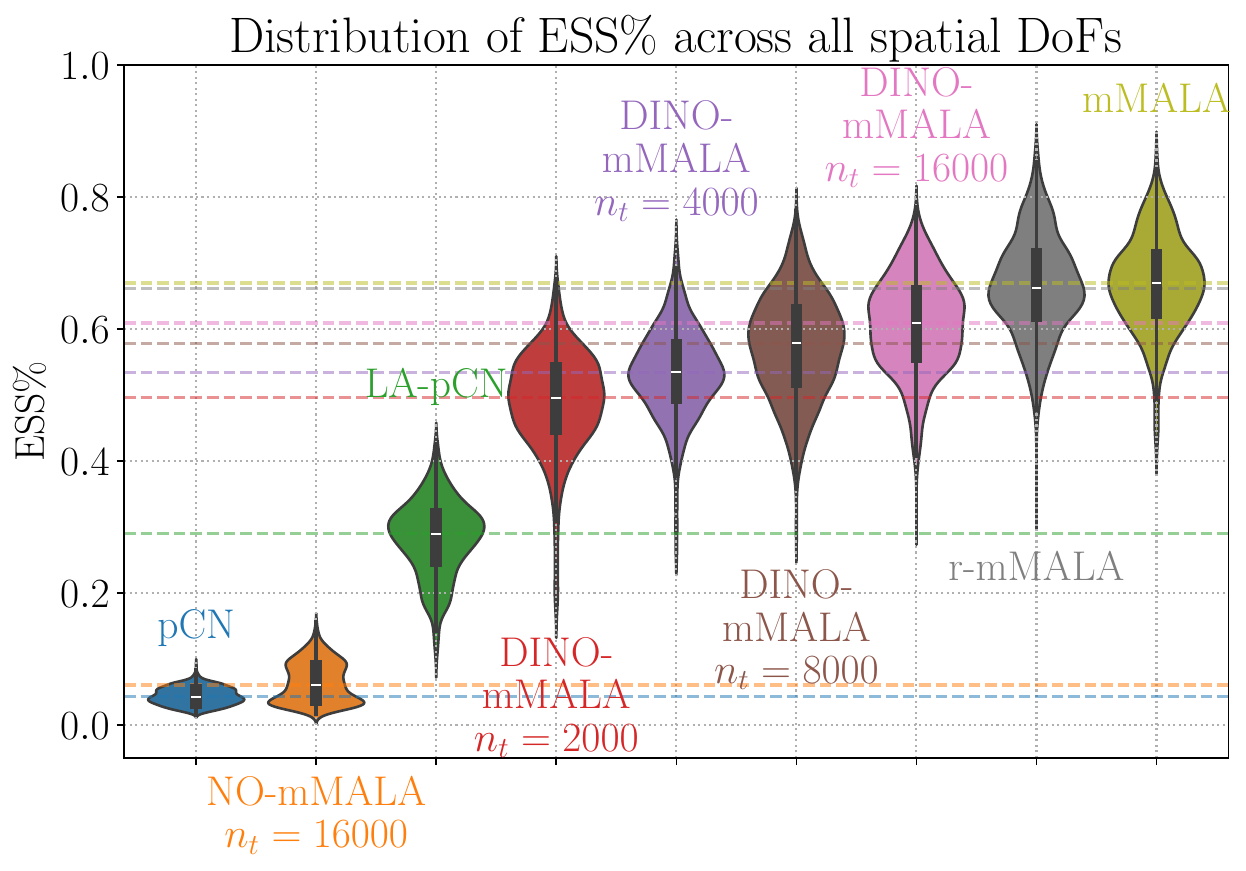}\includegraphics[align=c, width = 0.45\linewidth]{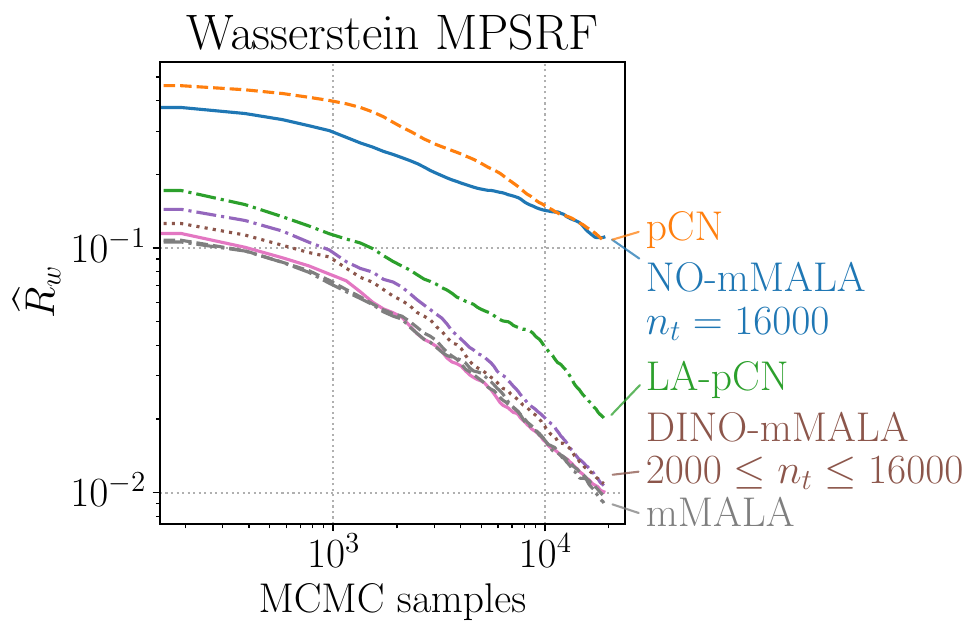}
        \caption{The diagnostics (see \cref{subsec:diagnostic}) of Markov chains generated by DINO-mMALA, NO-mMALA, and other baseline and reference MCMC methods in \cref{tab:mcmc_list,tab:dino_mcmc_list} for coefficient inversion in a nonlinear diffusion--reaction PDE. The symbol $n_t$ denotes the training sample size. (\textit{left}) The ESS\% diagnostic. (\textit{right}) The Wasserstein MPSRF diagnostic.}
        \label{fig:dino-mMALA}
    \end{figure}

\subsubsection{Delayed acceptance geometric MCMC with surrogate proposals}\label{subsubsec:ndr_da}
    In \cref{fig:da-dino-mpsrf} (\textit{left}) and \cref{fig:da-dino-ess}, we visualize the diagnostics of the Markov chains generated by DA geometric MCMC with operator surrogate proposal; see \cref{alg:mcmc}. By $n_t = 2000$, DA-DINO-mMALA outperforms LA-pCN regarding the effective sample size and mixing speed. According to the median ESS\%, DINO-mMALA retains $73\%$ of mMALA's effective sample size. Estimated by the ESS\% of DA-r-mMALA, the basis truncation error accounts for approximately $36\%$ of this reduction in ESS\%. After including the extra computational cost of mMALA at each MCMC sample and the cost reduction of the DA procedure, DA-DINO-mMALA at $n_t = 16000$ generates effective samples around 9 times faster than mMALA.
    
    For MCMC driven by $L^2_{\mu}$-trained NOs, DA-NO-mMALA at $n_t = 16000$ has an effective sample size and a mixing speed similar to pCN. When accounting for the cost reduction of the DA procedure, DA-NO-mMALA at $n_t = 16000$ generates effective samples 1.4 times faster than pCN. However, it is still 5 times slower than LA-pCN and mMALA. Furthermore, it is 19.5 and 44 times slower than DA-DINO-mMALA at $n_t = 2000$ and $16000$. See the speedups of DA-DINO-mMALA and DA-NO-mMALA against other methods in \cref{tab:ndr_cost_per_es}. 

    In \cref{fig:da-dino-mpsrf} (\textit{right}), we visualize the proposal acceptance rate in the first and second stages of the DA procedure for both DA-DINO-mMALA and DA-NO-mMALA. Recall that a low acceptance rate in the first stage leads to a low computational cost, and the acceptance rate in the second stage is correlated to the accuracy of the surrogate approximation. The plot shows that DA-DINO-mMALA has a high second-stage acceptance rate, improving consistently as the training sample size grows. When comparing between $L^2_{\mu}$-trained NO and $H^1_{\mu}$-trained DINO, the second-stage acceptance rate for DA-NO-mMALA is half of the rate for DA-DINO-mMALA, and the first-stage acceptance rate for DA-NO-mMALA is around 3 times the rate for DA-DINO-mMALA.

    In \cref{fig:ndr_total_speedup}, we plot the total effective sampling speedups of DA-DINO-mMALA against DA-NO-mMALA at $n_t=16000$, pCN, LA-pCN and mMALA. Recall from \cref{eq:total_speedup} that the total speedups include the offline cost of surrogate construction, and it is a function of the effective sample size requested for an MCMC run. The total speedups against LA-pCN and mMALA show that if one aims to collect more than just 10 effective posterior samples, it is more cost-efficient to use DA-DINO-mMALA with $n_t=1000$. On the other hand, the asymptotic speedup as $n_{\text{ess}}\to\infty$ (i.e., the effective sampling speedup in \cref{tab:ndr_cost_per_es}) at $n_t=1000$ is relatively small. A better asymptotic speedup requires a better surrogate trained with more samples. At $n_t=16000$, an asymptotic speedup of $8.8$ against LA-pCN and mMALA is achieved, and one needs to collect just 66 effective posterior samples to break even the offline training cost.

    \begin{figure}[tpb]
        \centering
        \includegraphics[align=c,width=0.5\linewidth]{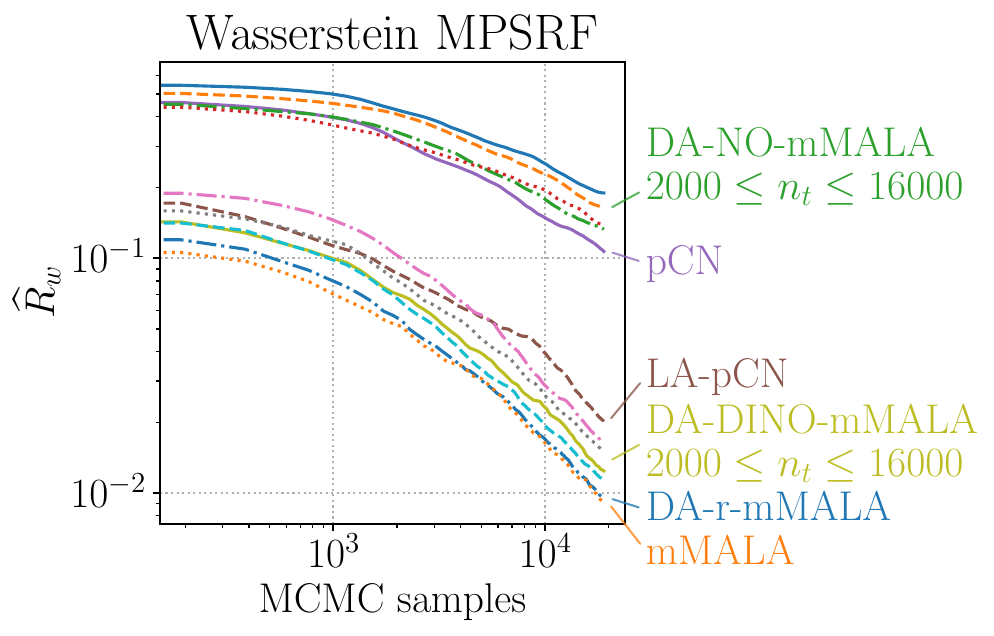}\includegraphics[align=c,width = 0.4\linewidth]{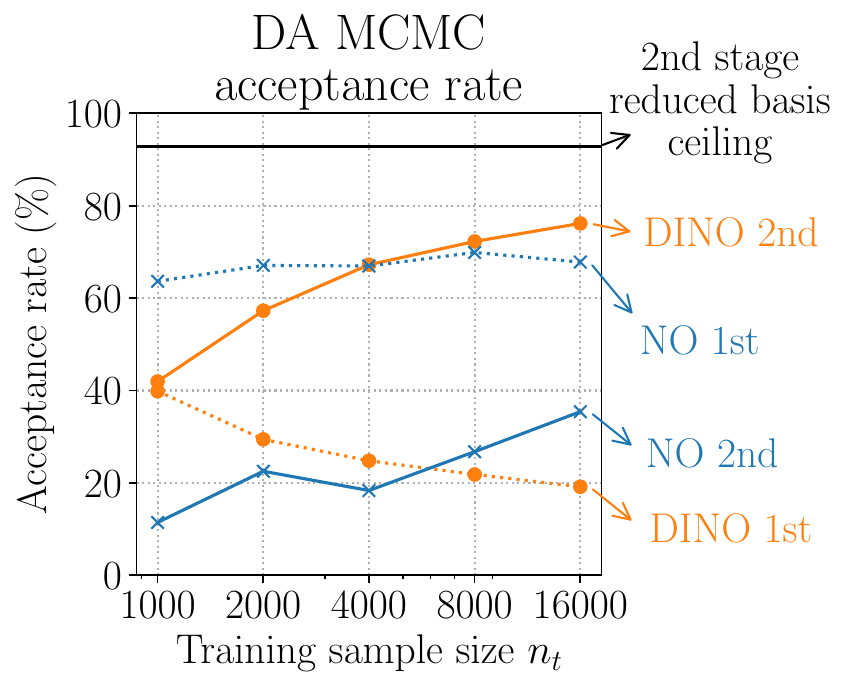}
        \caption{(\textit{left}) The Wasserstein MPSRF diagnostic (see \cref{subsubsec:mpsrf}) of Markov chains generated by DA-DINO-mMALA, DA-NO-mMALA, and other baseline and reference MCMC methods for coefficient inversion in a nonlinear diffusion--reaction PDE. (\textit{right}) The proposal acceptance rate in the first and second stages of the DA procedure as a function of training sample size. The reduced basis ceiling indicates the estimated (via DA-r-mMALA in \cref{tab:dino_mcmc_list}) highest second-stage acceptance rate for the reduced basis architecture with $r=200$ DIS reduced bases.}
        \label{fig:da-dino-mpsrf}
    \end{figure}

    \begin{figure}[tpb]
        \centering
        \includegraphics[width=0.8\linewidth]{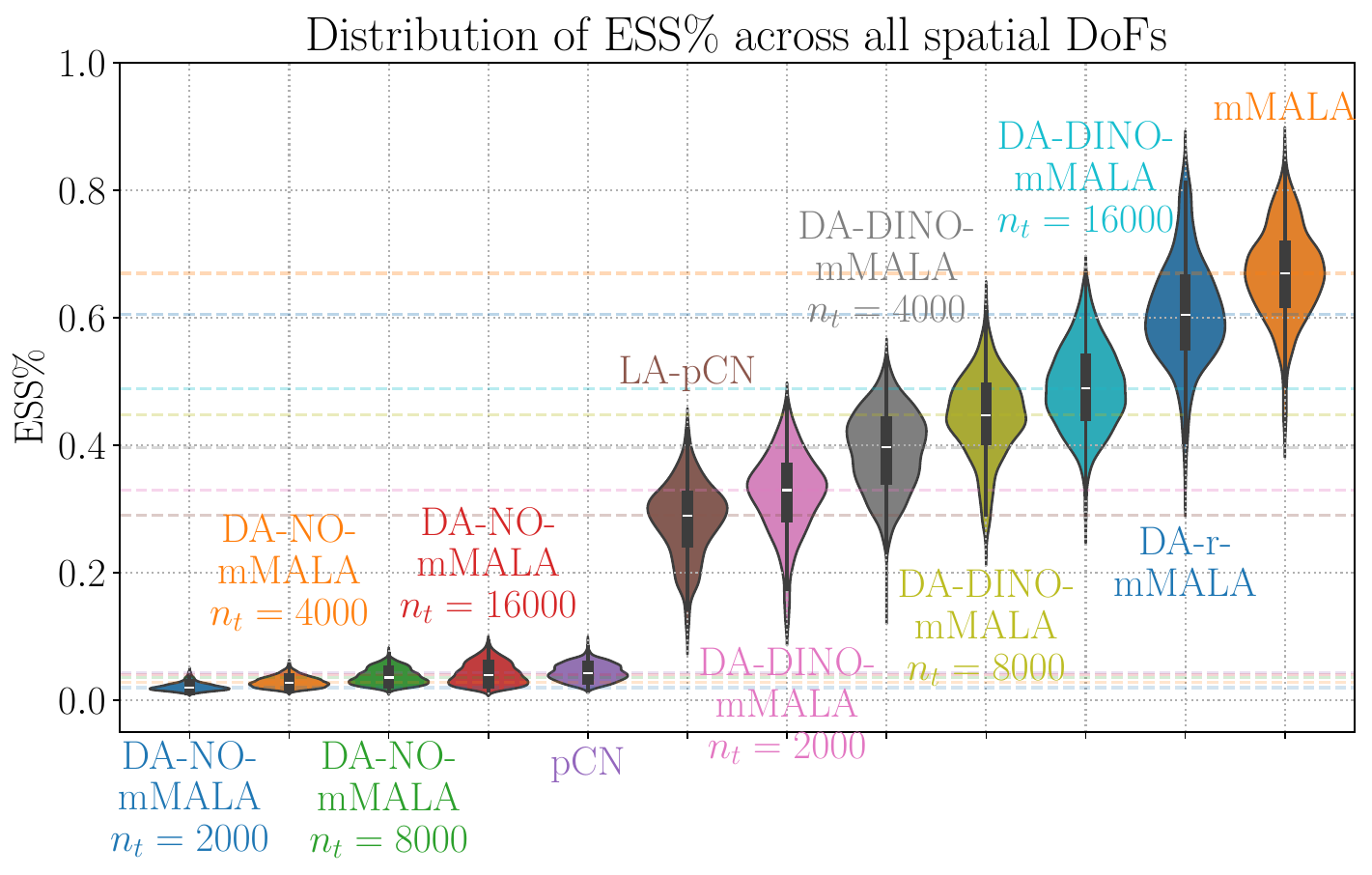}
        \caption{The ESS\% diagnostic (see \cref{subsubsec:ess}) of Markov chains generated by DA-DINO-mMALA, DA-NO-mMALA, and other baseline and reference MCMC methods for coefficient inversion in a nonlinear diffusion--reaction PDE. The symbol $n_t$ denotes the training sample size for operator learning.}
        \label{fig:da-dino-ess}
    \end{figure}

    \begin{figure}[tpb]
        \centering
        \includegraphics[width=0.49\linewidth]{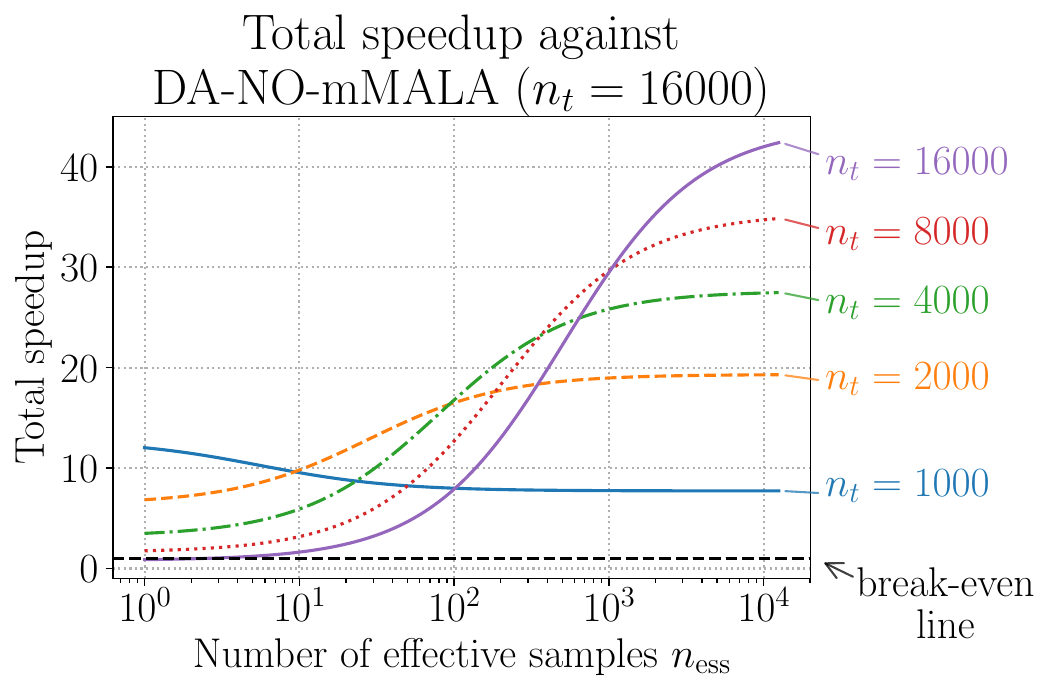}
        \includegraphics[width=0.49\linewidth]{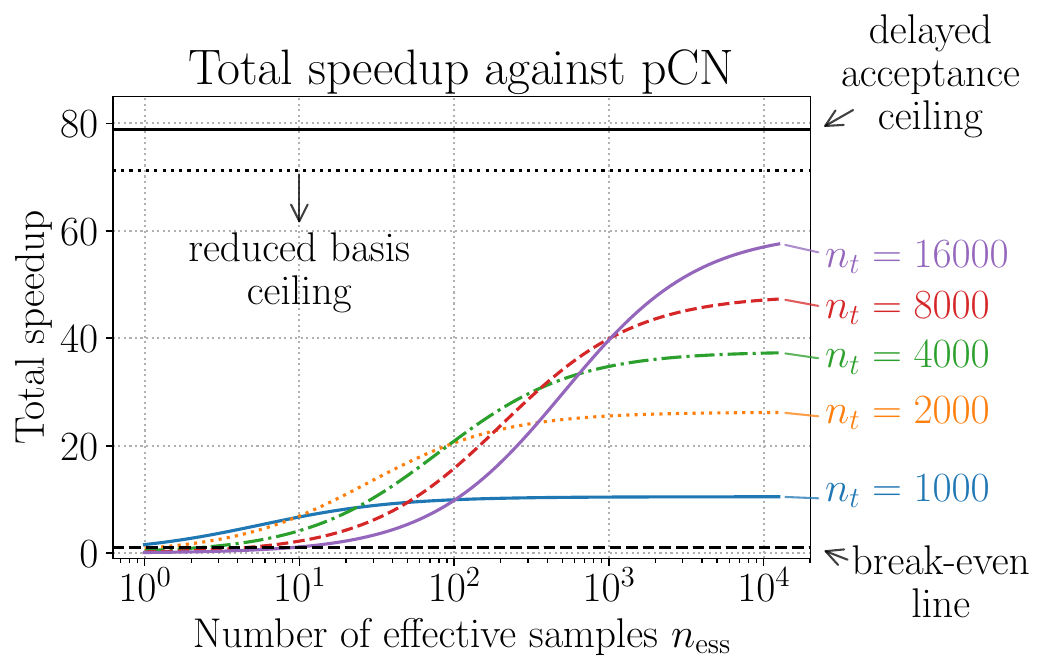}\\
        \includegraphics[width= 0.49\linewidth]{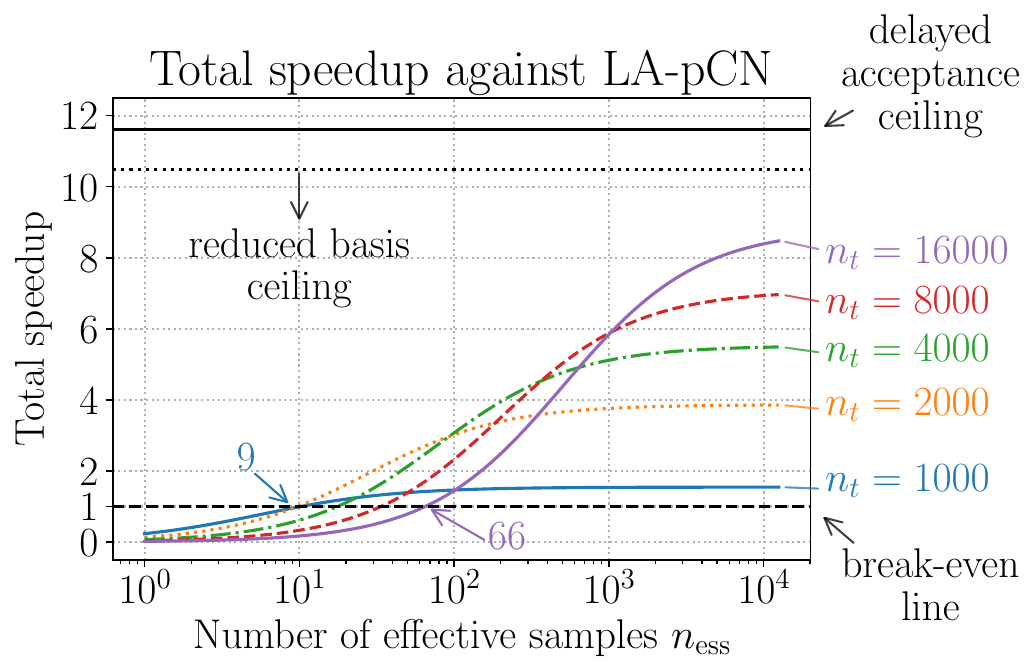}         \includegraphics[width= 0.49\linewidth]{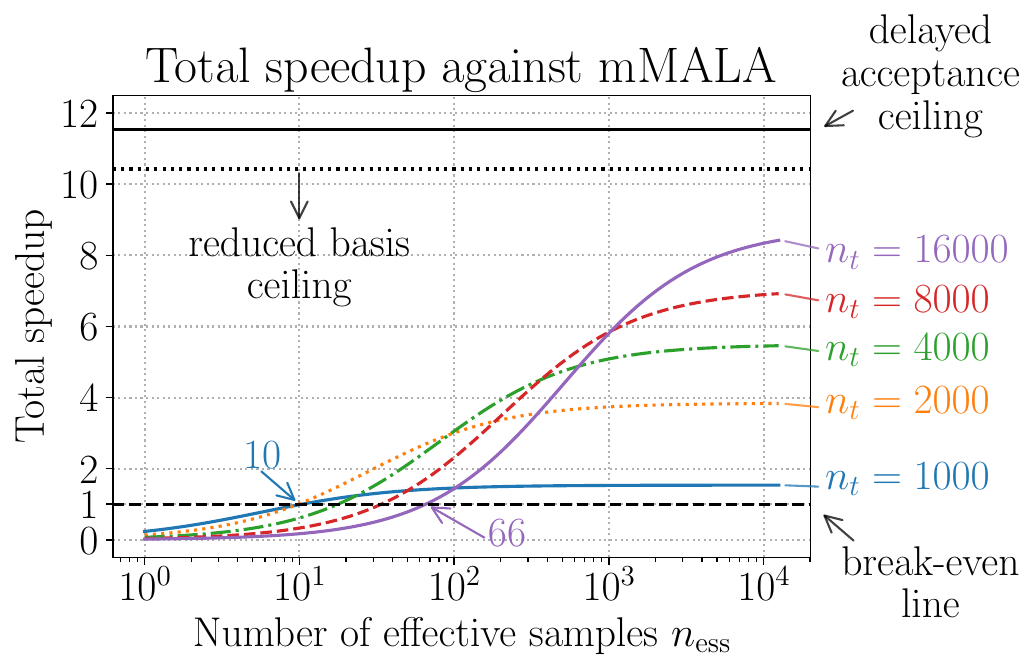}
        \caption{Total effective sampling speedups of DA-DINO-mMALA against DA-NO-mMALA, pCN, LA-pCN, mMALA as a function of effective sample size collected in an MCMC run for coefficient inversion in a nonlinear diffusion--reaction PDE. The total speedup in \cref{eq:total_speedup} compares the relative speed of an MCMC method for generating effective samples when including all computational costs, offline (e.g., training and MAP estimate) and online (MCMC). The break-even line indicates an equal total effective sampling speed of the two MCMC methods. The reduced basis ceiling indicates the estimated (via DA-r-mMALA in \cref{tab:dino_mcmc_list}) optimal speedup for the reduced basis architecture with $r=200$ DIS reduced bases. The delayed acceptance ceiling indicates the estimated asymptotic total speedup when the operator surrogate has no error (i.e., model-evaluated PtO map and reduced Jacobian), which leads to $100\%$ second stage acceptance rate. The symbol $n_t$ denotes the training sample size.}
        \label{fig:ndr_total_speedup}
    \end{figure}

\section{Numerical example: Inference of a heterogeneous hyperelastic material property}\label{sec:hyperelastic}

In this section, we consider an experimental scenario where a rectangular thin film of hyperelastic material is stretched on two opposite edges. The inverse problem aims to recover Young's modulus field characterizing spatially varying material strength from measurements of the material deformation.

This section is organized as follows. We introduce the material deformation model, the prior, the PtO map, and the setting for Bayesian inversion in \cref{subsec:hyperelasticity_model,subsec:hyperelasticity_prior,subsec:hyperelasticity_pto,subsec:hyperelasticity_bip}.  Then, we showcase and analyze MCMC results in \cref{subsec:hyperelasticity_results}. In \cref{subsubsec:hyperelasticity_baseline}, we discuss results on the baseline MCMC methods listed in \cref{tab:mcmc_list}. In \cref{subsubsec:hyperelasticity_da}, we discuss results on DA-NO-mMALA and DA-DINO-mMALA listed in \cref{tab:dino_mcmc_list}.

\subsection{The neo-Hookean model for hyperelastic material deformation}\label{subsec:hyperelasticity_model}

Let $\Omega = (0, 1)\times(0,2)$ be a normalized reference configuration for the hyperelastic material under the thin film approximation. The material coordinates $\bx\in\Omega$ of the reference configuration are mapped to the spatial coordinates $\bx + \bu(\bx)$ of the deformed configuration, where $\bu:\Omega\to\R^2$ is the material displacement. Internal forces are developed as the material deforms. These internal forces depend on the underlying stored internal energy; for a hyperelastic material, the strain energy $\calW_e$ depends on the deformation gradient, i.e.,  $\calW_e=\calW_e(\bF)$ where $\bF = \bI + \nabla \bu$ and $\bI\in\R^{2\times 2}$ is the identity matrix. We consider the neo-Hookean model for the strain energy density \citep{marsden1994mathematical,ogden1997non, gonzalez2008first}:
\begin{equation}
	\calW_e(\bF) = \frac{\mu}{2} (\mathrm{tr}(\bF^T\bF) - 3) + \frac{\lambda}{2} \left(\text{ln\,det}(\bF)\right)^2 - \mu \text{ln\,det}(\bF)\,.
\end{equation}
Here, $\lambda$ and $\mu$ are the Lam\'e parameters which are assumed to be related to Young's modulus of elasticity, $E$, and Poisson ratio, $\nu$, as follows:
\begin{equation}
    \lambda = \frac{E \nu}{(1+\nu)(1-2\nu)}\,, \qquad \mu = \frac{E}{2(1+\nu)}\,.
\end{equation}
We assume prior knowledge of Poisson ratio $\nu = 0.4$, and an epistemically-uncertain and spatially-varying Young's modulus, $E:\Omega\to(E_{\text{min}}, E_{\text{max}})$ with $0<E_{\text{min}}<E_{\text{max}}$. We represent $E$ through a parameter field $m:\Omega\to\R$ as follows
\begin{equation*}
    E(m(\bx)) = \frac{1}{2}\left(E_{\text{max}} - E_{\text{min}}\right)\left(\text{erf}(m(\bx)) + 1\right) + E_{\text{min}}\,,
\end{equation*}
where $\text{erf}:\R\to(-1, 1)$ is the error function.

The first Piola--Kirchhoff stress tensor is given by
\begin{equation}
\bP_e(m, \bF) = 2\frac{\partial\calW_e(m, \bF)}{\partial \bF}\,.
\end{equation}
Assuming a quasi-static model with negligible body forces, the balance of linear momentum leads to the following nonlinear PDE:
\begin{subequations}
\begin{align}\label{eq:linmombal}
	\nabla \cdot \bP_e(m(\bx), \bF(\bx)) &= \bzero\,, && \bx\in \Omega\,;\\
	\bu(\bx) &= \bzero\,, && \bx\in \Gamma_{\text{left}}\,;\\
    \bu(\bx) & = 3/2\,, && \bx\in \Gamma_{\text{right}}\,;\\
	\bP_e(m(\bx), \bF(\bx)) \cdot \bn  &= \bzero\,, && \bx\in \Gamma_{\text{top}}\cup\Gamma_{\text{bottom}};
\end{align}
\end{subequations}
where $\Gamma_t$, $\Gamma_r$, $\Gamma_b$, and $\Gamma_l$ denote the material domain's top, right, bottom, and left boundary. Notice that the stretching is enforced as a Dirichlet boundary condition, and the strain specified on the $\Gamma_{\text{right}}$ is $0.75$.

\subsection{The prior distribution}\label{subsec:hyperelasticity_prior}
The normalized Young's modulus follows a prior distribution defined through a Gaussian random field $M\sim\mu$ with $E_{\text{min}} = 1$ and $E_{\text{max}} = 7$:
\begin{align*}
    \scrM &\coloneqq L^2(\Omega)\,, && (\text{Parameter space})\\
    \mu&\coloneqq\mathcal{N}(0, (-\gamma\grad\cdot\bA\grad + \delta\calI_{\scrM})^{-2})\,, &&(\text{The prior distribution})
\end{align*}
where the differential operator is equipped with a Robin boundary for eliminating boundary effects, and $\bA\in\R^{2\times 2}$ is a symmetric positive definite anisotropic tensor given by
\begin{gather*}
    \bA = \begin{bmatrix}
        \theta_1\sin(\alpha)^2 & (\theta_1-\theta_2)\sin(\alpha)\cos(\alpha)\\
        (\theta_1-\theta_2)\sin(\alpha)\cos(\alpha) & \theta_2\cos(\alpha)^2
    \end{bmatrix}\,,
\end{gather*}
where $\theta_1 = 2$, $\theta_2=1/2$, $\alpha = \arctan(2)$. The constants $\gamma,\delta\in\R_+$ are set to $\gamma = 0.3$ and $\delta = 3.3$, which approximately leads to a pointwise variance of 1 and a spatial correlation of 2 and 1/2 perpendicular and along the left bottom to right top diagonal of the spatial domain. We approximate $\scrM$ using a finite element space $\scrM^h$ constructed by linear triangular finite elements with $2145$ DoFs. Prior samples are visualized in \cref{fig:hyperelasticity_samples}.

\subsection{The parameter-to-observable map}\label{subsec:hyperelasticity_pto}

We consider a symmetric variational formulation for hyperelastic material deformation, and define the following Hilbert spaces following the notation in \cref{subsec:data_generation}:
\begin{align*}
    \scrU &\coloneqq \left\{\bu\in H^1(\Omega;\R^2)\;\Big\vert\; \bu|_{\Gamma_l} = \bzero\wedge \bu|_{\Gamma_r} = \bzero\right\}\,; && (\text{State space})\\
    \scrV &\coloneqq \scrU'\,, && (\text{Residual space})
\end{align*}
where $\scrU'$ denotes the dual space of $\scrU$. To enforce the inhomogeneous Dirichlet boundary condition, we decompose the displacement $\bu$ into $\bu=\bu_0 + \bB\bx$, where $\bB = \begin{bmatrix} 3/4 & 0 \\ 0 & 1 \end{bmatrix}$ and $\bu_0\in\scrU$ is the PDE state with the homogenous Dirichlet boundary condition. 

The residual operator $\calR:\scrU\times\scrM\to\scrV$ is defined by its action on an arbitrary test function $\bp\in \scrU$:
\begin{align}\label{eq:hyper_residual}
    \left\langle\calR(\bu_0, m), p\right\rangle_{\scrU'\times\scrU} \coloneqq \int_{\Omega} \bP_e\left(m(\bx), \bI + \grad \bu_0(\bx) + \bB \right)\grad \bp(\bx)\dd\bx\,.
\end{align}
The effective PDE solution operator $\mathcal{F}:\scrM\ni m\mapsto \bu_0\in\scrU$ satisfies $\mathcal{R}(\mathcal{F}(m),m) = 0$. We approximate $\scrU$ using a finite element space $\scrU^h$ constructed by quadratic triangular elements with $16770$ DoFs. Evaluating the discretized PDE solution operator involves solving the discretized residual norm minimization problem via the Newton--Raphson method in $\scrU^h$.

We define a observation operator $\bdmc{O}\in\scrU\to \R^{64}$ using $32$ equally spaced discrete interior points $\{\bx_{\text{obs}}^{(j)}\}_{j=1}^{32}$ similar to \cref{eq:observation_operator}.
The PtO map is $\bdmc{G}\coloneqq \bdmc{O}\circ\mathcal{F}$. We visualize the output of the PtO map in \cref{fig:hyperelasticity_samples}.

\begin{figure}
    \centering
    \renewcommand{\arraystretch}{0.1}
    \addtolength{\tabcolsep}{-5pt}
    \begin{tabular}{c c c c}
    \bf Prior samples & \bf Young's modulus & \bf Deformed configurations & \makecell{\bf Observables\\ (2-norm)}\\
        \includegraphics[width=0.2\linewidth]{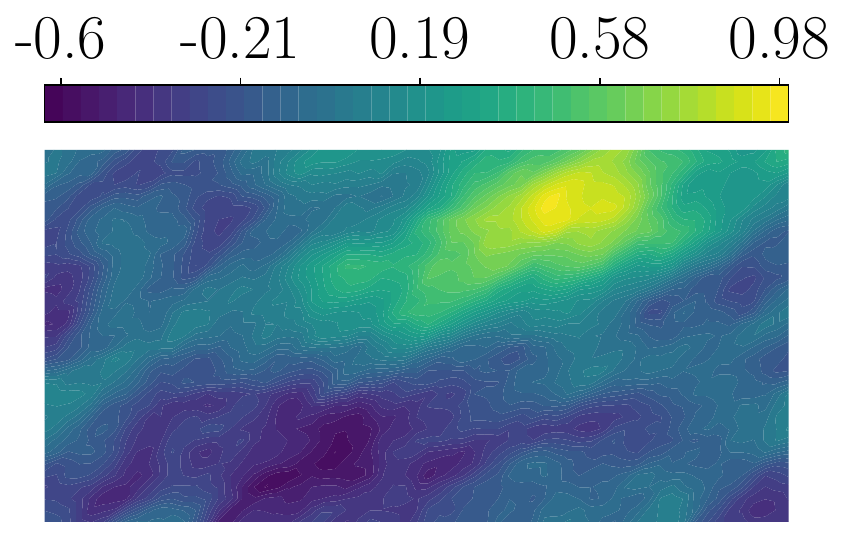}  & 
        \includegraphics[width=0.2\linewidth]{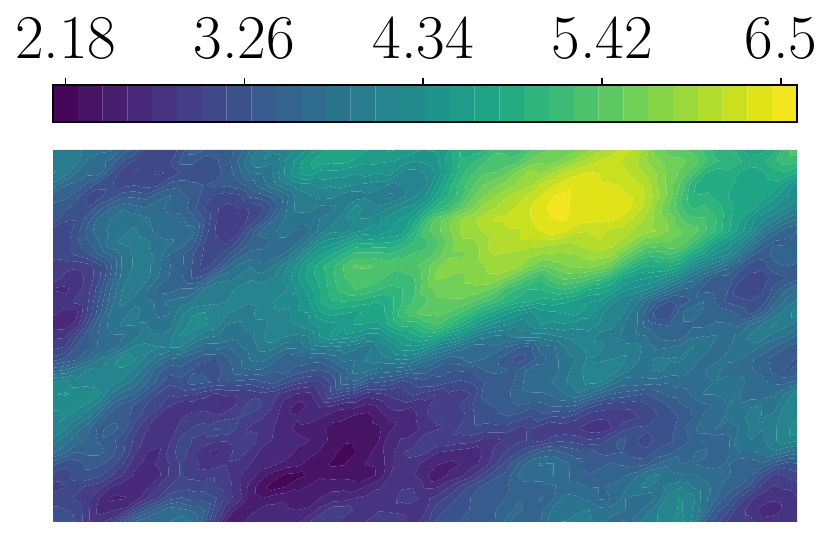} &\includegraphics[width=0.3\linewidth]{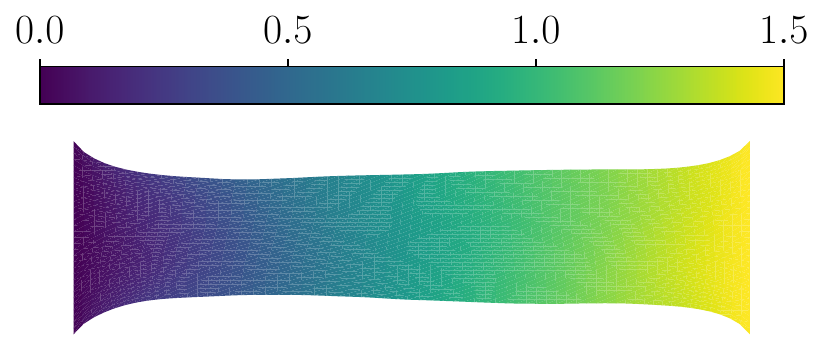} & \includegraphics[width=0.2\linewidth]{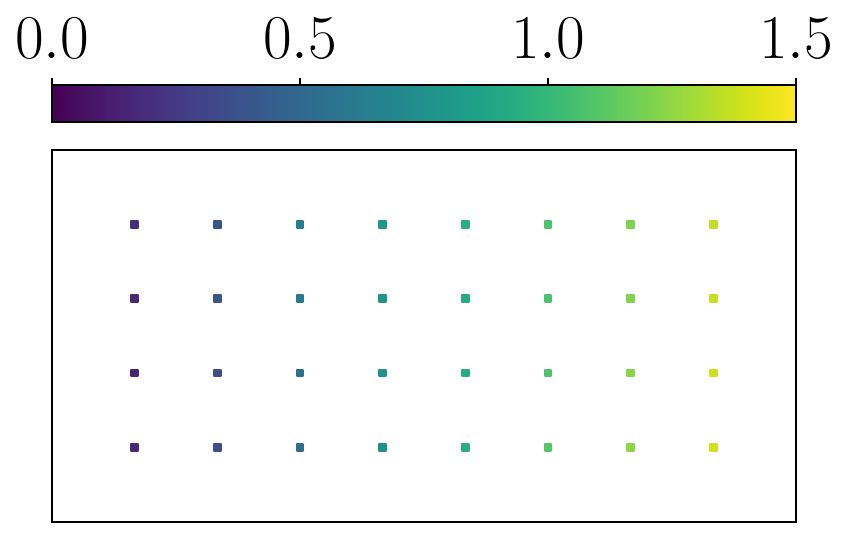} \\
        \includegraphics[width=0.2\linewidth]{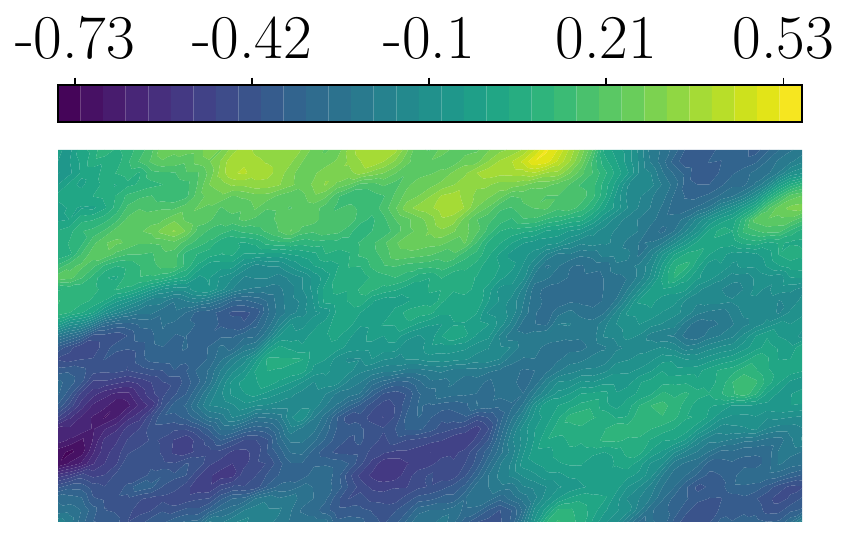} & 
        \includegraphics[width=0.2\linewidth]{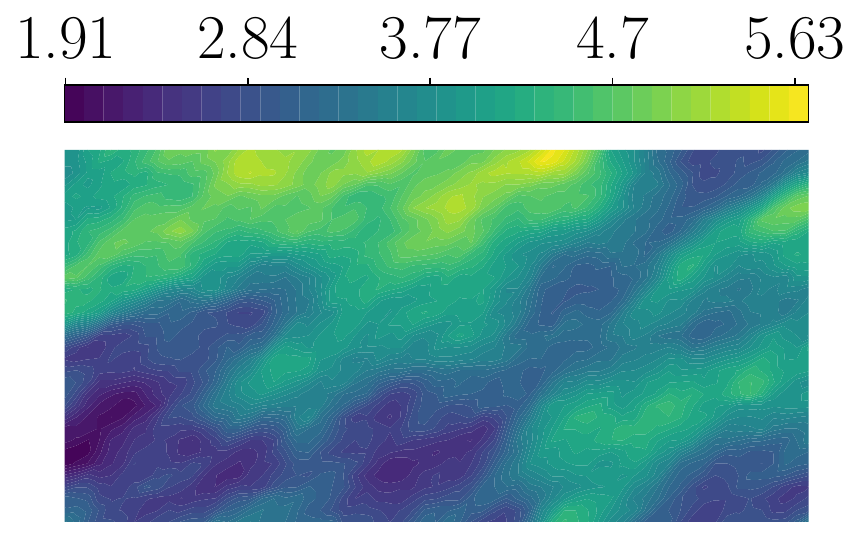} &\includegraphics[width=0.3\linewidth]{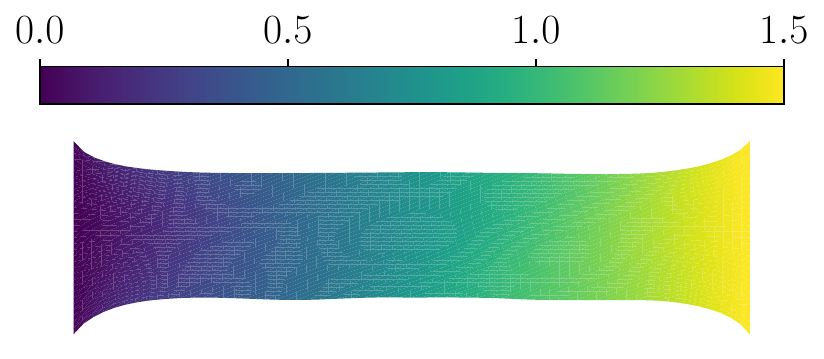} & \includegraphics[width=0.2\linewidth]{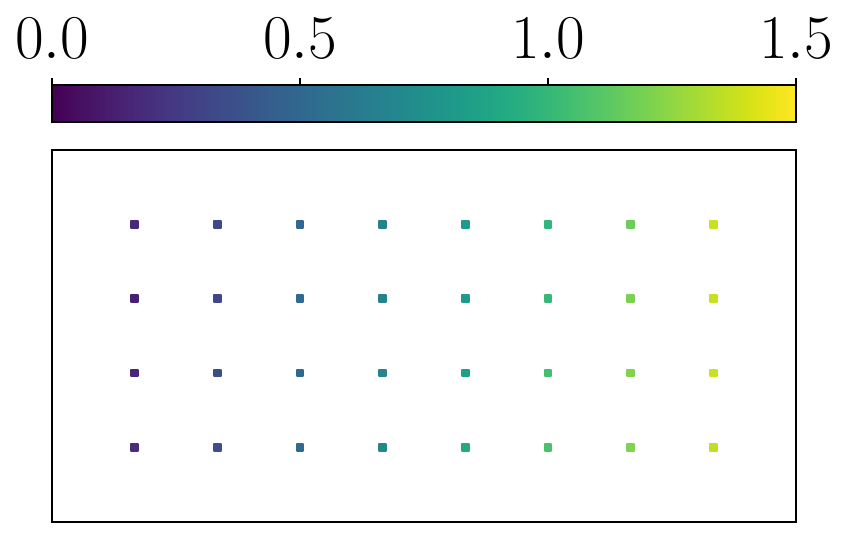} \\
    \end{tabular}
    \addtolength{\tabcolsep}{5pt}
\renewcommand{\arraystretch}{1.0}
    \caption{Visualizations of prior samples ($2145$ DoFs), deformed configuration ($16770$ DoFs), and predicted observables (in $\R^{64}$) for hyperelastic material property discovery.}
    \label{fig:hyperelasticity_samples}
\end{figure}

\subsection{Bayesian inverse problem settings}\label{subsec:hyperelasticity_bip}

We generate synthetic data for our BIP using a prior sample. The model-predicted observables at the synthetic parameter field are corrupted with $1\%$ additive white noise, which has a noise covariance matrix of identity scaled by $v_n=1.8\times 10^{-4}$. The synthetic data, its generating parameter and PDE solution, and the MAP estimate are visualized in \cref{fig:hyperelasticity_settings}.

\begin{figure}[tpb]
    \centering
    \addtolength{\tabcolsep}{-5pt}
    \begin{tabular}{c c c c}
        \makecell{\bf Synthetic Parameter} & \makecell{\bf Deformed configuration}  & \makecell{\bf Observed data\\ (2-norm)}  & {\bf MAP estimate} \\
        \includegraphics[width=0.22\linewidth]{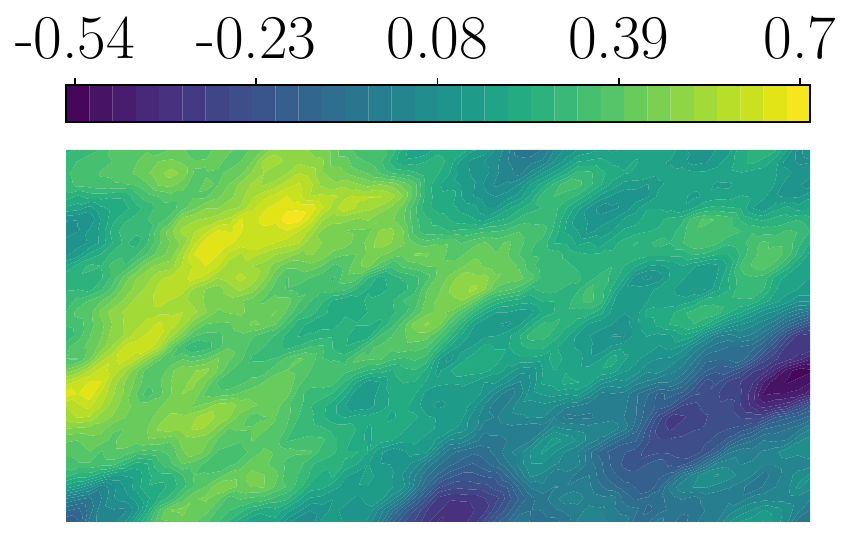} & \includegraphics[width = 0.3\linewidth]{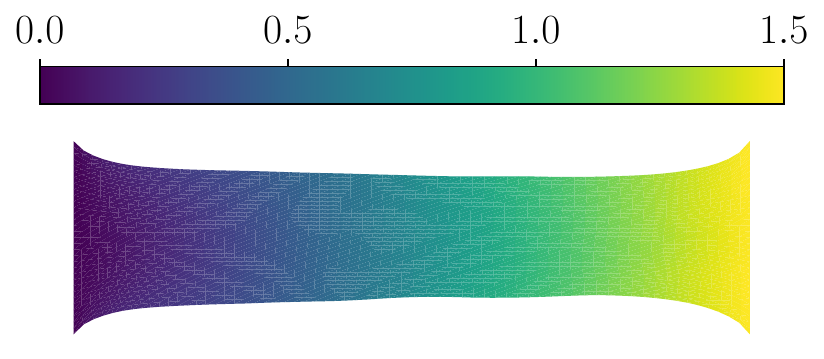} & \includegraphics[width = 0.22\linewidth]{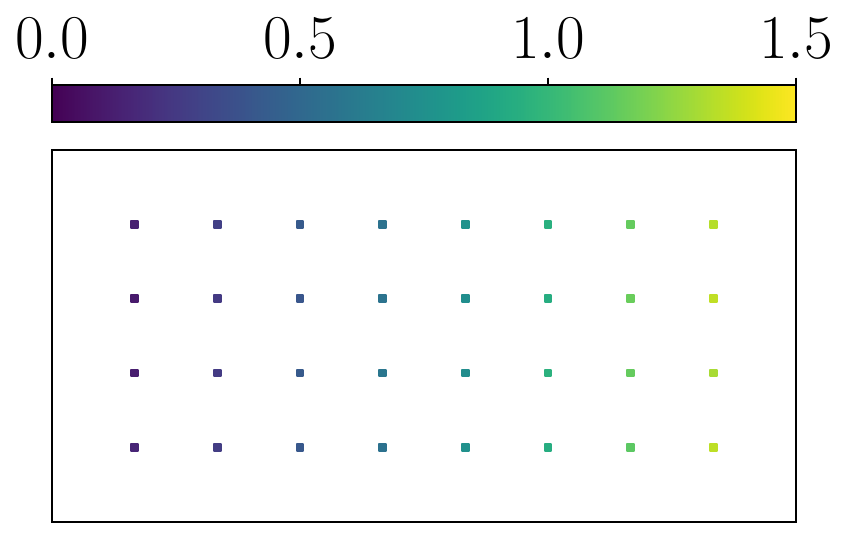} & \includegraphics[width = 0.22\linewidth]{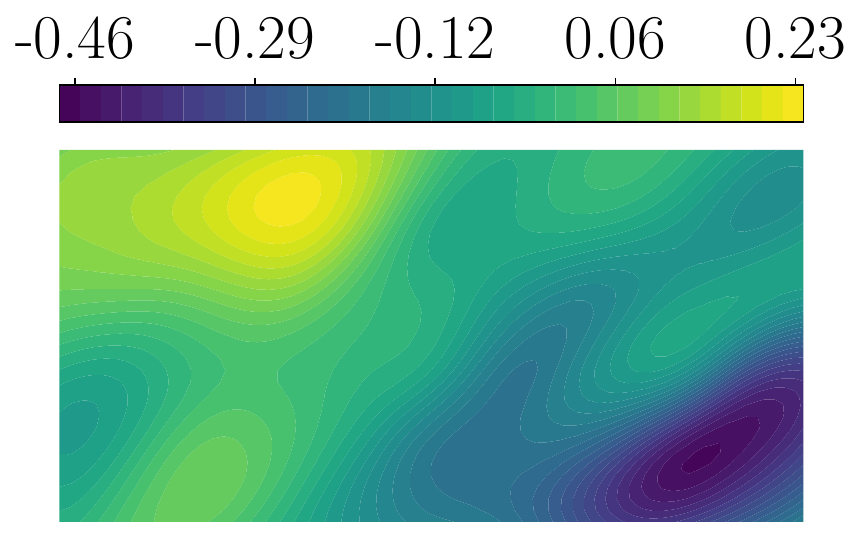} \\
    \end{tabular}
    \addtolength{\tabcolsep}{5pt}
    \caption{Visualization of the BIP setting and the MAP estimate for hyperelastic material property discovery.}
    \label{fig:hyperelasticity_settings}
\end{figure}

\subsection{Neural operator surrogates}\label{subsec:hyperelasticity_surrogate}

We follow the procedure described in \cref{subsec:data_generation} for generating samples for neural network training and testing. We compute DIS reduced bases of dimension $r=200$ using $n_{\text{DIS}} = 500$ of the generated samples as specified in \cref{eq:mc_ppgnh}. Selected DIS basis functions are visualized in \cref{fig:hyper_basis}. Forming reduced Jacobian matrices via columns using a direct solver with reused factorization takes $10\%$ of the computing time for solving the nonlinear PDEs using a direct solver, estimated on average over sample generation. Note that the relative cost of forming reduced Jacobian is low mainly because a large number of Newton--Raphson iterations are needed to solve the PDE.

We use a dense neural network architecture with six hidden layers, each with $400$ hidden neurons and a GELU activation function, trained using $n_t = 125$, $250$, $\dots$, $8000$ samples. We estimate the generalization errors of the trained neural networks using $2500$ testing samples. In \cref{fig:hyperelasticity_accuracy}, we plot the estimated errors as a function of training sample generation cost, measured relative to the averaged cost of one nonlinear PDE solve.

\begin{figure}[tpb]
    \centering
    \includegraphics[width = 0.49\linewidth]{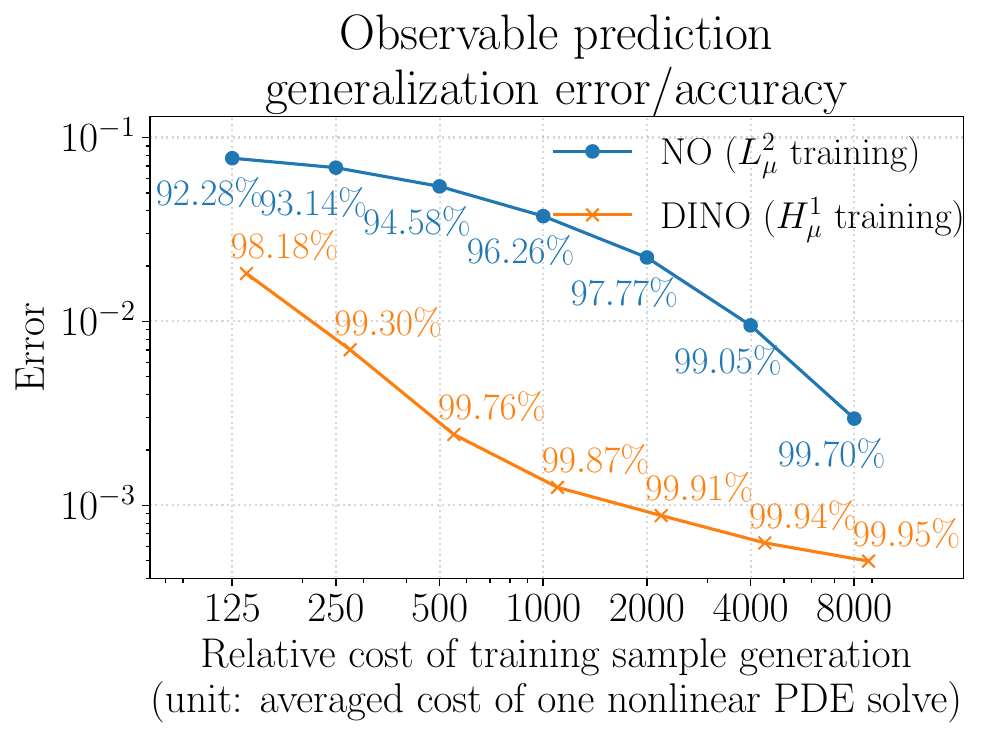}\includegraphics[width=0.49\linewidth]{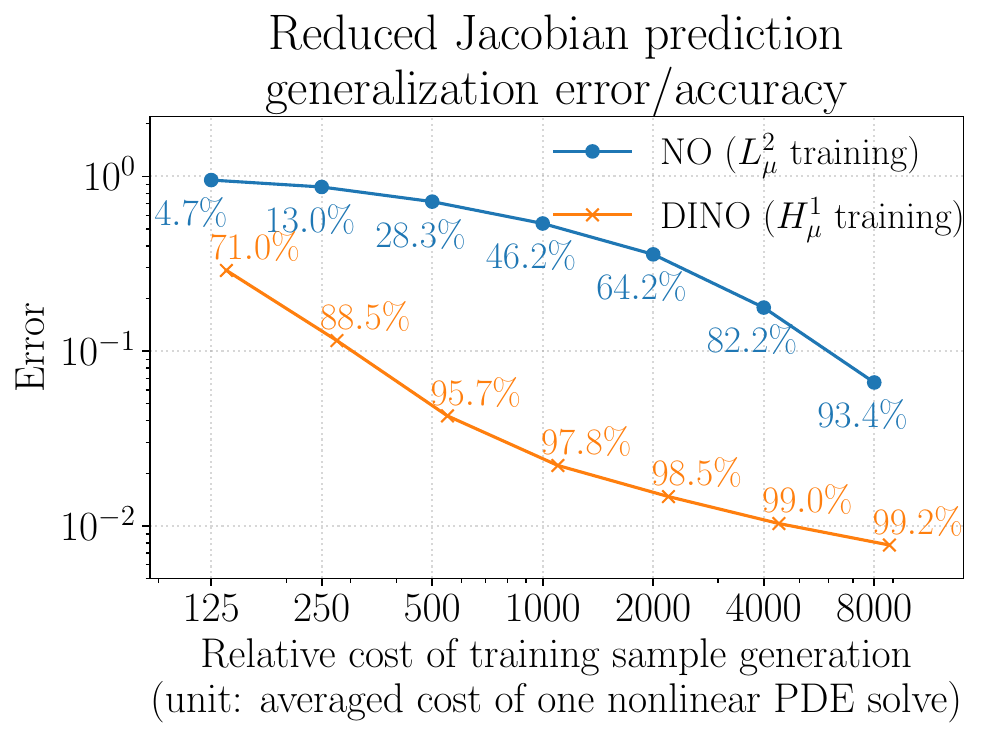}
    \caption{The generalization error and accuracy \cref{eq:generalization_error} of observable prediction and reduced Jacobian prediction made by $L^2_{\mu}$-trained neural operators and $H^1_{\mu}$-trained DINOs for inference of a heterogeneous hyperelastic material property. The error is plotted as a function of training sample generation cost, measured relative to the averaged cost of one nonlinear PDE solve.}
    \label{fig:hyperelasticity_accuracy}
\end{figure}

The plot leads us to similar conclusions outlined in \cref{subsec:ndr_surrogate}. It shows that the derivative-informed $H^1_{\mu}$ operator learning significantly enhances the quality of the neural operator surrogate at the same training sample generation cost compared to the conventional $L^2_{\mu}$ operator learning. Notably, achieving comparable observable and reduced Jacobian prediction generalization accuracy requires approximately 16 times fewer training samples with $H^1_{\mu}$ training.
 
\subsection{MCMC results}\label{subsec:hyperelasticity_results}

We present numerical results on the efficiency of DA-DINO-mMALA compared to the baseline MCMC methods. For each method, we collect $n_c = 10$ Markov chains with different initialization, each with $n_s = 19000$ samples. The step size parameter $\triangle t$ and initialization are chosen carefully according to the procedure detailed in \cref{app:step_size}. The statistics of the MCMC runs and posterior visualization are provided in \cref{sec:supplementary}.

\subsubsection{The baseline MCMC methods}\label{subsubsec:hyperelasticity_baseline}

    \begin{figure}[tpb]
        \centering
        \includegraphics[width=0.4\linewidth]{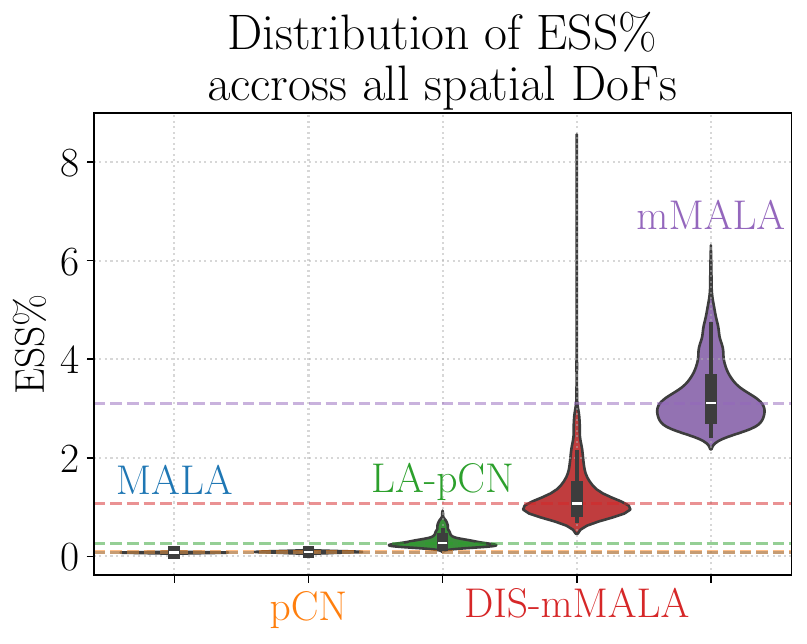}\includegraphics[width=0.43\linewidth]{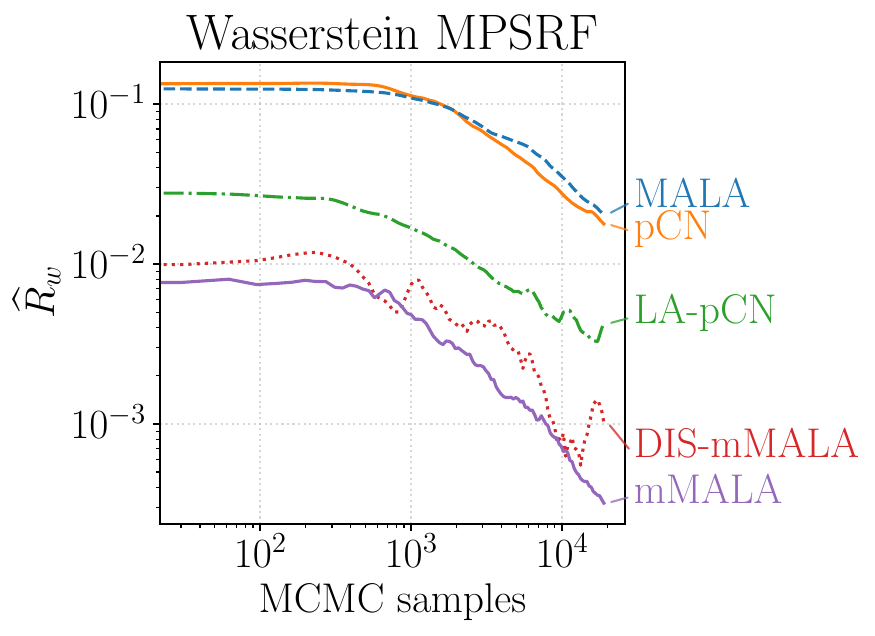}\\
        \caption{Visualization of the diagnostics (see \cref{subsec:diagnostic}) of MCMC chains generated by baseline MCMC methods listed in \cref{tab:mcmc_list} for inference of a heterogeneous hyperelastic material property. (\textit{left}) The violin plot of the ESS\% distributions over $2{,}145$ spatial DoFs for the discretized parameter space. We note that the ESS\% of DIS-mMALA is larger than that of mMALA at only $5$ spatial DoFs. (\textit{left}) The Wasserstein MPSRF as a function of the Markov chain position.}
        \label{fig:hyper_baseline}
    \end{figure}

\begin{table}[tpb]
    \centering
    \begin{tabular}{|c|c|c|c|}\hline
        \diagbox[width=10em]{\bf Baseline}{\bf Speedup}    & mMALA & \makecell{DA-DINO-mMALA\\$n_t=500$} & \makecell{DA-DINO-mMALA\\$n_t=4000$}    \\\hline
        pCN & 29.6 & 72.7  & 96.2 \\\hline
        MALA & 38.8 & 93.2  & 126.4 \\\hline
        LA-pCN & 9.9 & 24.4  & 32.3 \\\hline
        DIS-mMALA & 2.8 & 6.8  & 9  \\\hline
        mMALA & 1 & 2.5  & 3.3 \\\hline
        \makecell{DA-NO-mMALA\\$n_t=4000$} & 5.5  & 13.5 & 17.9\\\hline
    \end{tabular}
    \caption{The effective sampling speedup of mMALA, DINO-mMALA, and DA-DINO-mMALA against other baseline MCMC methods for inference of a heterogeneous hyperelastic material property. The speedup measures the relative speed of generating effective samples for an MCMC method compared against another MCMC method; see \cref{eq:effectiv_sample_speed}.}
    \label{tab:hyperelasticity_cost_per_es}
\end{table}

In \cref{fig:hyper_baseline}, we visualize the diagnostics for the baseline MCMC methods in \cref{tab:mcmc_list}. The diagnostics show that mMALA produces Markov chains with the most effective samples and fastest mixing time among the baseline methods. When comparing methods with the same type of posterior geometry information (see \cref{tab:mcmc_list}), MALA is slightly inferior to pCN, and LA-pCN is much inferior to DIS-mMALA.

Comparing the median of ESS\%, mMALA produces 35 and 3 times more effective samples than pCN and DIS-mMALA. Notably, the ESS\% of DIS-mMALA is larger than that of mMALA for just $5$ DoFs out of $2145$. Moreover, due to the large number of iterative solves required for each PtO map evaluation, each Markov chain sample generated by mMALA is only around 1.18 and 1.05 times more computationally costly than pCN and DIS-mMALA. The effective sampling speedups \cref{eq:effectiv_sample_speed} of mMALA against other baseline MCMC methods are provided in \cref{tab:hyperelasticity_cost_per_es}.

\subsubsection{Delayed acceptance geometric MCMC with surrogate proposals}\label{subsubsec:hyperelasticity_da}

    \begin{figure}[tpb]
        \centering
        \includegraphics[width=0.5\linewidth]{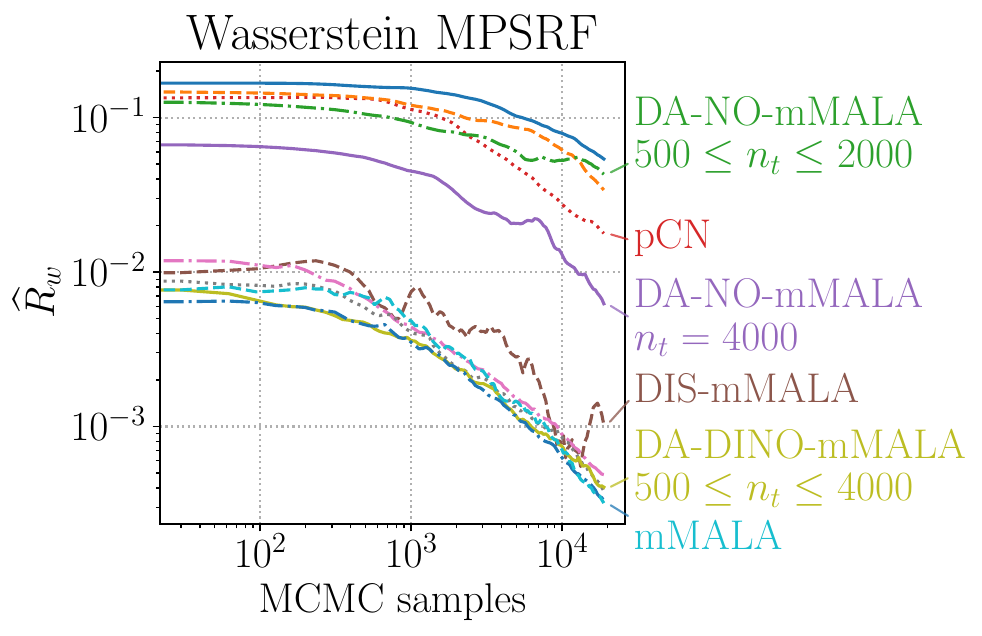}\includegraphics[width = 0.42\linewidth]{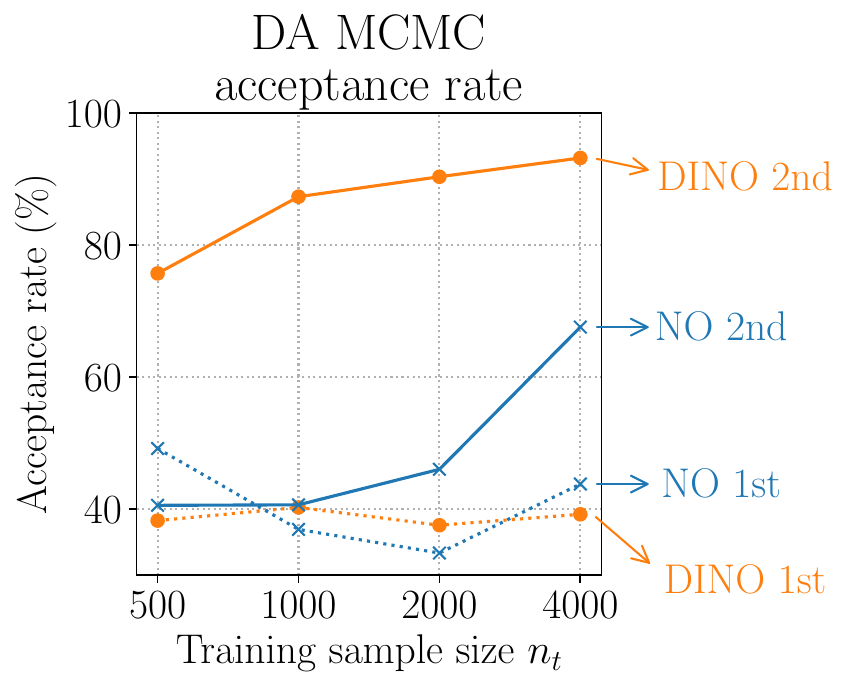}
        \caption{(\textit{left}) The Wasserstein MPSRF diagnostic (see \cref{subsubsec:mpsrf}) of Markov chains generated by DA-DINO-mMALA, DA-NO-mMALA, and other baseline MCMC methods for inference of a heterogeneous hyperelastic material property. (\textit{right}) The proposal acceptance rate in the first and second stages of the DA procedure as a function of training sample size. }
        \label{fig:hyper_da-dino-mpsrf}
    \end{figure}

        \begin{figure}[tpb]
        \centering
        \includegraphics[width=0.8\linewidth]{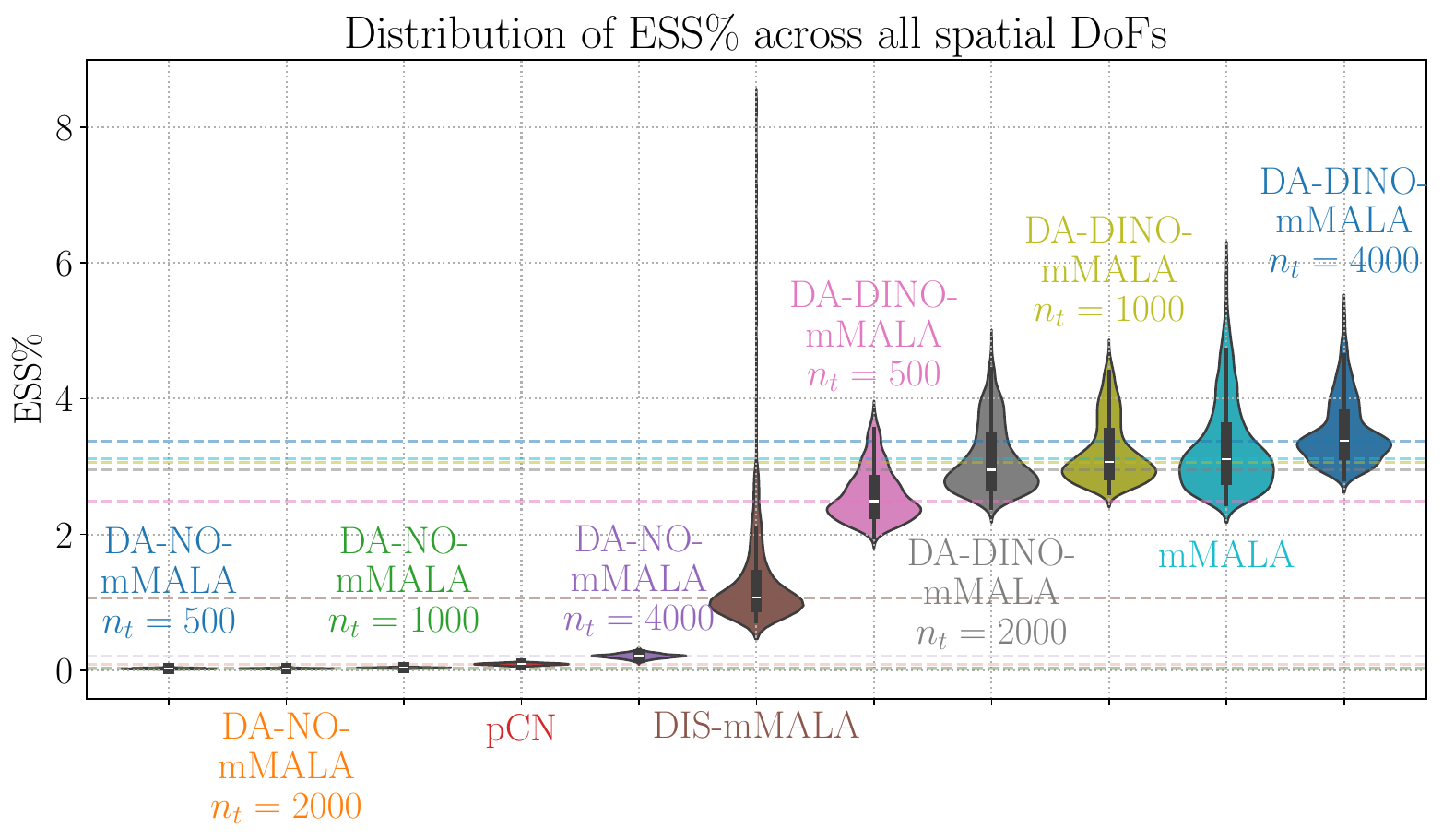}
        \caption{The ESS\% diagnostic (see \cref{subsubsec:ess}) of Markov chains generated by DA-DINO-mMALA, DA-NO-mMALA, and other baseline MCMC methods for inference of a heterogeneous hyperelastic material property. The symbol $n_t$ denotes the training sample size.}
        \label{fig:hyper_da-dino-ess}
    \end{figure}

    \begin{figure}[tpb]
        \centering
        \includegraphics[width=0.49\linewidth]{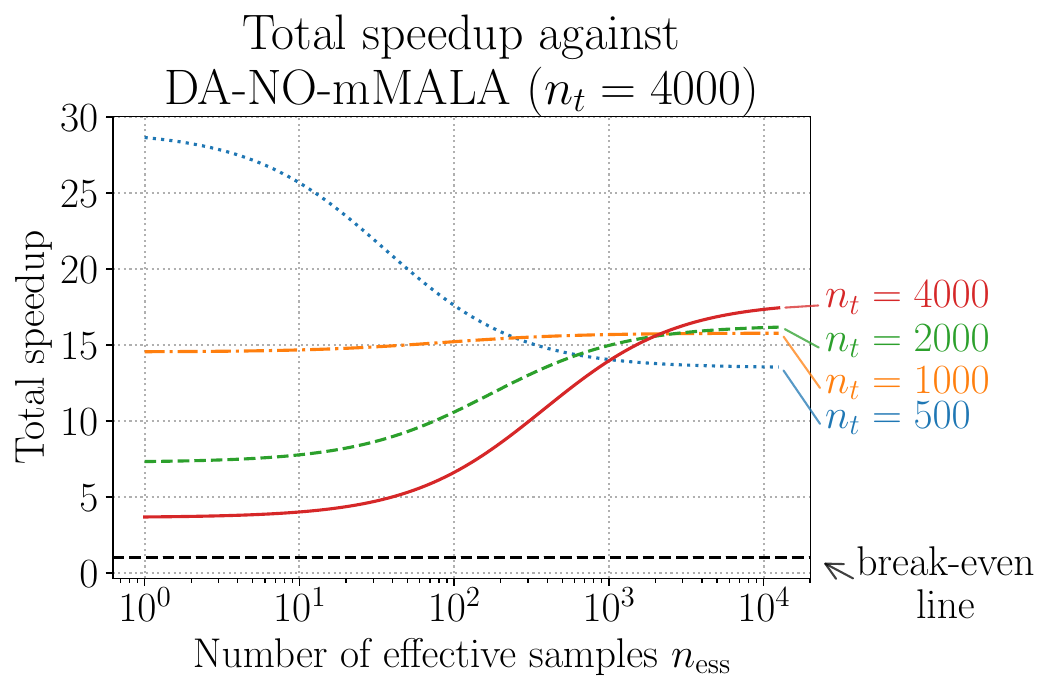}
        \includegraphics[width=0.49\linewidth]{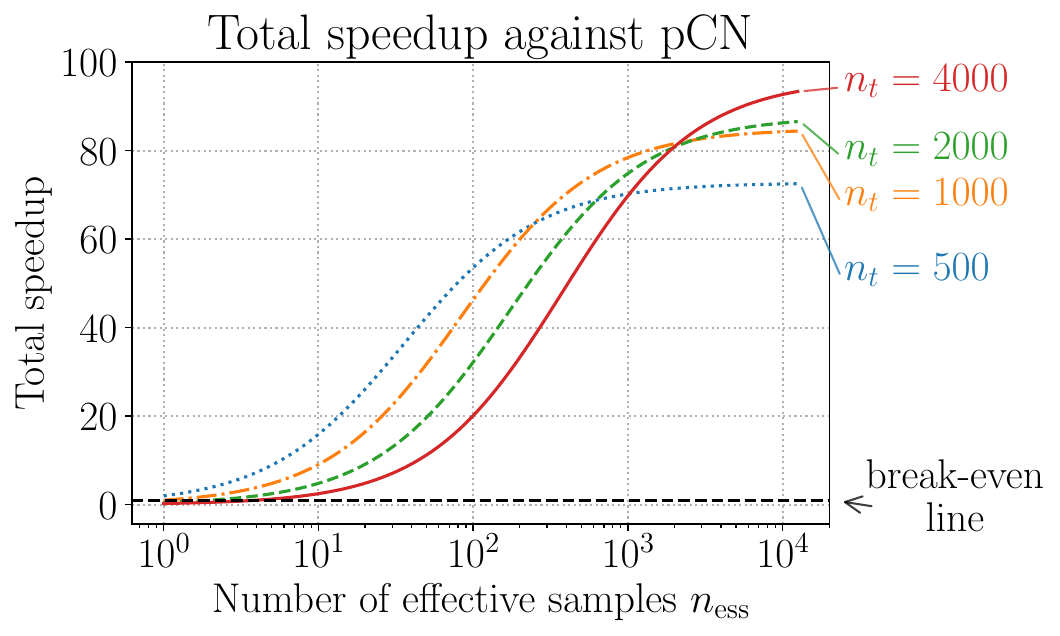}\\ 
        \includegraphics[width=0.49\linewidth]{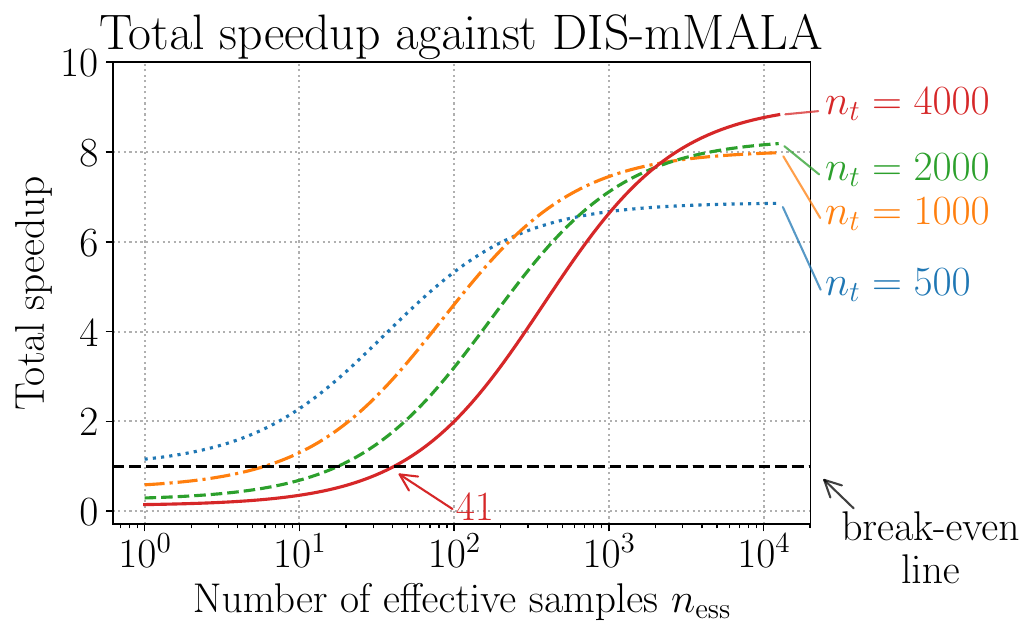}
    \includegraphics[width= 0.49\linewidth]{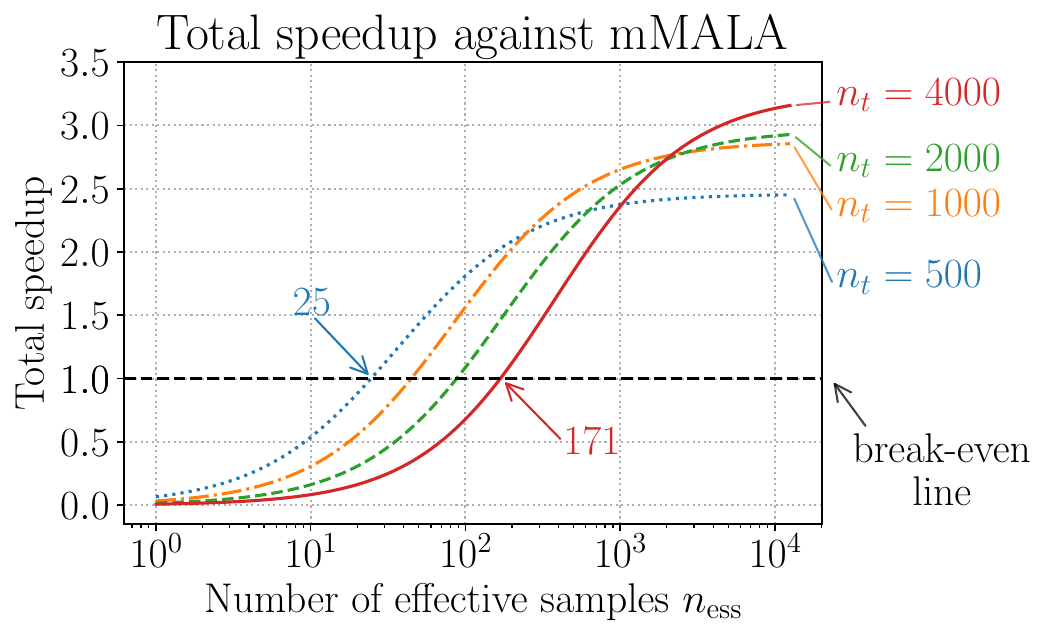}
        \caption{Total effective sampling speedups of DA-DINO-mMALA against DA-NO-mMALA, pCN, DIS-mMALA, mMALA as a function of effective sample size collected in an MCMC run for inference of a heterogeneous hyperelastic material property. The total speedup in \cref{eq:total_speedup} compares the relative speed of an MCMC method for generating effective samples when considering all computational costs, offline (e.g., training and MAP estimate) and online (MCMC). The break-even line indicates an equal total effective sampling speed of the two MCMC methods. The symbol $n_t$ denotes the training sample size.}
        \label{fig:hyper_total_speedup}
    \end{figure}

    In \cref{fig:hyper_da-dino-mpsrf} (\textit{left}) and \cref{fig:hyper_da-dino-ess}, we visualize the diagnostics of the Markov chains generated by the surrogate-driven geometric MCMC with DA. The diagnostics show that DA-DINO-mMALA at $n_t = 500$ and beyond outperforms DIS-mMALA regarding the effective sample size and mixing speed of MCMC chains. Furthermore, the ESS\% of DA-DINO-mMALA plateaued at around $n_t = 1000$, meaning the ESS\% fluctuation is dominated by (i) imperfect step size tuning, (ii) finite chain length, and (iii) finite Markov chain pool sizes. Accounting for the extra computational cost of mMALA at each MCMC sample and the cost reduction of the DA procedure, DA-DINO-mMALA at $n_t = 4000$ generates effective samples around 3.3 times faster than mMALA. 
    
     Driven by $L^2_{\mu}$-trained NOs, DA-NO-mMALA at $n_t = 4000$ surpasses pCN regarding the effective sample size and mixing speed. After including the cost reduction of the DA procedure, DA-NO-mMALA at $n_t = 4000$ generates effective samples 5 and 2 times faster than pCN and LA-pCN. However, it is still 2 and 5.5 times slower than DIS-mMALA and mMALA. Furthermore, it is 13.5 and 17.9 times slower than DA-DINO-mMALA at $n_t = 500$ and $4000$. See the speedups of DA-DINO-mMALA and DA-NO-mMALA against other methods in \cref{tab:ndr_cost_per_es}. 

    In \cref{fig:hyper_da-dino-mpsrf} (\textit{right}), we visualize the proposal acceptance rate in the first and second stages of the DA procedure for both DA-DINO-mMALA and DA-NO-mMALA. The plot shows that DA-DINO-mMALA has a high second-stage acceptance rate, improving consistently as the training sample size grows. The second-stage acceptance rate for DA-NO-mMALA is 1.3--2.3 times the rate for DA-DINO-mMALA. These results affirm that $H^1_{\mu}$-trained DINO leads to higher quality Markov chains for posterior sampling.

    In \cref{fig:hyper_total_speedup}, we plot the total effective sampling speedups of DA-DINO-mMALA against DA-NO-mMALA at $n_t=4000$, pCN, LA-pCN and mMALA. The total speedups against LA-pCN and mMALA show that if one aims to collect more than just 25 effective posterior samples, it is more cost-efficient to switch to DA-DINO-mMALA with $n_t=500$ rather than using mMALA. On the other hand, the asymptotic speedup at $n_t=500$ is relatively small. We achieve an asymptotic speedup of $3.3$ against mMALA, and one only needs to collect 171 effective posterior samples to break even the offline cost of surrogate training.

\section{Conclusion}
In this work, we propose deploying a neural operator surrogate of the PtO map to accelerate geometric MCMC and obtain fast and consistent solutions to infinite-dimensional Bayesian inverse problems. The method represents a synthesis of ideas from reduced basis DINO, DA MCMC, and dimension-independent geometric MCMC with the goal of designing an MCMC proposal that adapts to DINO-predicted posterior local geometry within a delayed acceptance procedure. Compared to conventional geometric MCMC, this surrogate-driven geometric MCMC method leads to significant cost reduction, as it requires no online forward or adjoint sensitivity solves, fewer model evaluations, and fewer instances of prior sampling. Our numerical results show that our proposed method is capable of producing high-quality Markov chains typical of a geometric MCMC method at a much lower cost, leading to substantial speedups in posterior sample generation. In particular, our numerical examples show that DA-DINO-mMALA generates effective posterior samples 60--97 times faster than pCN and 3--9 times faster than mMALA. Moreover, the training cost of DINO surrogates breaks even after collecting just 10--25 effective posterior samples when compared to mMALA.

The key to enabling surrogate acceleration of geometric MCMC is our derivative-informed operator learning formulation. We present an operator learning objective in $H^1_{\mu}$ Sobolev space with Gaussian measure that controls error in approximating the stochastic derivative of the PtO map. This formulation is naturally equipped with the Poincar\'e inequality for nonlinear mappings on function spaces, which is used to derive a $L^2_{\mu}$ approximation error bound for the reduced basis neural operator surrogate consisting of three sources of error: (i) neural network approximation of the optimal reduced mapping, (ii) basis truncation error, and (iii) sampling error when applicable. Our numerical examples show that derivative-informed $H^1_{\mu}$ operator learning achieves similar generalization accuracy in predicting the observable vector and the reduced Jacobian of the PtO map with at least 16--25 times fewer training sample generation cost than conventional $L^2_{\mu}$ operator learning. For coefficient inversion in a nonlinear diffusion--reaction PDE, we observe an estimated 166 times difference in training sample generation cost between derivative-informed $H^1_{\mu}$ and conventional $L^2_{\mu}$ operator learning for achieving an acceleration of mMALA.

\section*{Acknowledgments}

This work was partially supported by the National Science Foundation under awards OAC-2313033 and DMS-234643, and the U.S. Department of Energy, Office of Science, Office of Advanced Scientific Computing Research under awards DE-SC0021239 and DE-SC0023171, and the Air Force Office of Scientific Research under MURI grant FA9550-21-1-0084. The work of Lianghao Cao was partially supported by a Department of Defense Vannevar Bush Faculty Fellowship held by Andrew M. Stuart, and by the SciAI Center, funded by the Office of Naval Research (ONR), under Grant Number N00014-23-1-2729. This work benefited from discussions with Dingcheng Luo and Jakob Zech.
\appendix
\section{Stochastic G\^ateaux differentiability}\label{sec:stochastic_derivative}
From \cref{ass:stochastic_derivative} and \cref{def:gateaux_derivative}, it is clear that $\mu$-a.e.\ G\^ateaux differentiability implies stochastic G\^ateaux differentiability. Here, we provide a case where the reverse cannot be true. The following lemma uses the Cameron--Martin and Feldman--H\'ajek theorem \cite[Theorem 2.51]{sullivan2015intro} to establish that $\mu$-a.e.\ G\^ateaux differentiability requires a more regular forward operator $\bdmc{G}$ than the $\mu$-a.e.\ well-definedness specified in \cref{ass:well_posedness}.
\begin{lemma}\label{lem:sder}
    Assume $\bdmc{G}$ is well-defined $\mu$-a.e.\ and ill-defined on all sets $\scrA\subset \scrM$ with $\mu(\scrA) = 0$. Then, stochastic G\^ateaux differentiability does not imply $\mu$-a.e.\ G\^atueax differentiability. In particular, $\bdmc{G}$ is not G\^ateaux differentiable $\mu$-a.e.\ 
\end{lemma}
\begin{proof}
We focus on the term $\bdmc{G}(M + t\widehat{m})$ for $t>0$ and $M\sim\mu$ in the definition of the G\^ateuax ($\widehat{m}\in\scrM$) and stochastic G\^ateuax ($\widehat{m}\in\scrH_{\mu}$) derivative. Let $\scrN_{\mu}\coloneqq\{\scrA\subset\scrM\,\lvert\,\mu(\scrA) = 0\}$ be the null set of $\mu$ and $M + t\widehat{m}\sim \nu(\cdot;t\widehat{m})$. We have two scenarios listed as follows.
\begin{enumerate}[label = (\roman*)]
    \item $\scrN_{\mu} = \scrN_{\nu(\cdot;t\widehat{m})}$ for all $t>0$ if and only if $\widehat{m}\in\scrH_{\mu}$, i.e., the null sets are shift invariant.
    \item There exists a set $\scrE_t\subset\scrM$ such that $\mu(\scrE_t) = 0$ and $\nu(\scrE_t;t\widehat{m}) = 1$ for all $t>0$ and $\widehat{m}\in\scrM\setminus\scrH_{\mu}$, i.e., the shifted distributions have disjoint probability concentration.
\end{enumerate}
Due to (1), the term $\bdmc{G}(M+t\widehat{m})$ is well-defined a.s.\ for all $t>0$ and $\widehat{m}\in\scrH_{\mu}$, thus the limiting sequence within the stochastic derivative definition is well-defined. Due to (2), $M + t\hat{m}\in \mathscr{E}_t$ a.s.\ for $\widehat{m}\in\scrM\setminus\scrH_{\mu}$, in which case $\bdmc{G}(M+t\hat{m})$ is ill-defined and the limiting sequence within the G\^atueax derivative definition is ill-defined. 
\end{proof}

\section{The gradient and the Gauss--Newton Hessian of the data misfit} \label{appendix:gaussnewtonhessian}
In this section, we show the connection of our definitions of the ppGNH $\calH$ in \cref{eq:ppgh} and the ppg $\sder \Phi^{\by}(m)$ in \cref{eq:ppg} to the conventional definitions. Assume the G\^ateaux derivative $D\bdmc{G}(m)$ exists, then the following relation holds $\mu$-a.e.
\begin{equation}\label{eq:adjoint_relation}
    \sder\bdmc{G}(m) = D\bdmc{G}(m)|_{\scrH_{\mu}}\,,\quad
    \sder \bdmc{G}(m)^* = \cpr D\bdmc{G}(m)^*\bC_n^{-1}\,.
\end{equation}
We have the following $\scrM$-Risez representation of the gradient, i.e., G\^ateaux derivative of the data misfit:
\begin{equation*}
    D\Phi^{\by}(m) \coloneqq D\bdmc{G}(m)^*\bC_n^{-1}(\bdmc{G}(m)-\by)\,.
\end{equation*}
By the definition of the ppg and \cref{eq:adjoint_relation}
\begin{equation*}
    \sder\Phi^{\by}(m) \coloneqq \cpr D\Phi^{\by}(m) \implies \sder\Phi^{\by}(m)=  \sder \bdmc{G}(m)^*(\bdmc{G}(m)-\by)\,.
\end{equation*}
We now show that $\sder\Phi^{\by}(m)$ is indeed the $\scrH_{\mu}$-Riesz representation of the data misfit stochastic derivative. Suppose the stochastic derivative of the data misfit is given by $\calT\in\scrH_{\mu}'$, where $\scrH_{\mu}'$ is the dual space of $\scrH_{\mu}$. By the chain rule, for all $\delta m\in\scrH_{\mu}$ we have
\begin{align*}
    \calT(m)\delta m &= \left(\sder\bdmc{G}(m)\delta m\right)^T\bC_n^{-1}\left(\bdmc{G}(m)-\by\right) && (\text{Chain rule})\\
    & = \left\langle\sder\bdmc{G}(m)^*(\bdmc{G}(m)-\by), \delta m\right\rangle_{\cpr^{-1}} && (\text{Definition of adjoint on } \HS(\scrH_{\mu},\scrY))\\
    & = \left\langle\sder\Phi^{\by}(m), \delta m\right\rangle_{\cpr^{-1}}\,.
\end{align*}
Therefore, our definition of the ppg is identical to the conventional definition. Similarly, we have the following definition of the Gauss--Newton Hessian using the G\^ateaux derivative:
\begin{align*}
    \calH_{\text{GN}}(m) \coloneqq D\bdmc{G}(m)\bC_n^{-1}D\bdmc{G}(m)\,. && (\text{Gauss--Newton Hessian})
\end{align*}
By the definition of the ppGNH and \cref{eq:adjoint_relation}, we have
\begin{equation*}
    \calH(m) \coloneqq \cpr \calH_{\text{GN}}(m) \implies \calH(m)= \sder\bdmc{G}(m)^*\sder\bdmc{G}(m)\,.
\end{equation*}
To see that $\calH(m)$ is the Gauss--Newton Hessian under the stochastic derivative assumption, notice that the stochastic Hessian of the data misfit $\sder^2\Phi^{\by}(m)\in \HS(\scrH_{\mu})$ is given by
\begin{align}\label{eq:stochastic_hessian}
    \sder^2\Phi^{\by}(m)\delta m\coloneqq  \left(\sder^2\bdmc{G}(m)\delta m\right)^*\left(\bdmc{G}(m)-\by\right) + \calH(m)\delta m \quad \forall \delta m\in\scrH_{\mu}\,,
\end{align}
where $\sder^2\bdmc{G}(m)\in \HS(\scrH_{\mu}, \HS(\scrH_{\mu}, \scrY))$ is the stochastic Hessian of the forward operator. Assuming the data misfit term is relatively small in regions with high posterior probability, one may drop the term involving the Hessian of the PtO map and still retain a reasonable approximation to the data misfit Hessian. This makes the ppGNH $\calH$ an approximation to the stochastic Hessian of the data misfit.

\section{Delayed acceptance and neural operator approximation error}\label{app:da_and_error}

Recall from \cref{subsec:da} that the proposal acceptance rate in the second stage of the DA MCMC reflects the quality of the generated Markov chain and the second stage acceptance probability is closely related to the error in surrogate data misfit evaluation, denoted by $\calE_{\text{misfit}}:\scrM\to\R$:
\begin{equation*}
    \ln\rho^{(2)}(m_j, m^{\dagger}) = \calE_{\text{misfit}}(m_j) -\calE_{\text{misfit}}(m^{\dagger})\,,\quad \calE_{\text{misfit}}(m) \coloneqq \Phi^{\by}(m)- \widetilde{\Phi^{\by}}(m)\,\quad \mu\text{-a.e.}.
\end{equation*}
where $\rho^{(2)}$ is the transition rate ratio of DA MCMC defined in \cref{eq:da_second}. The arguments of $\rho^{(2)}$, namely $m_j$ and $m^{\dagger}$, are coupled through the proposal and the first stage of DA MCMC, thus analyzing $\rho^{(2)}$ is not straightforward and is not the focus of this work. However, the error analysis for surrogate data misfit evaluation provides insights into the behavior of $\rho^{(2)}$ and the efficiency of DA MCMC related to the operator surrogate approximation error. In particular, the $L^1_{\mu^{\by}}(\scrM)$ approximation error of the surrogate data misfit (averaged over the true posterior) is controlled by the surrogate approximation error (averaged over the prior):
\begin{align}\label{eq:error_misfit}
    \norm{\calE_{\text{misfit}}}_{L^1_{\mu^{\by}}(\scrM)} &\coloneqq \mathbb{E}_{M\sim\mu^{\by}}\left\lvert\Phi^{\by}(M)-\Phi^{\by}(M)\right\rvert\nonumber\\
    &\leq c_{\text{misfit}}(\bdmc{G}, \widetilde{\bdmc{G}}, \by, \bC_n^{-1} \mu)\norm{\bdmc{G}-\widetilde{\bdmc{G}}}_{L^2_{\mu}(\scrM;\scrY)}\,,
\end{align}
where $c_{\text{misfit}}>0$ is a constant given by
\begin{equation*}
    c_{\text{misfit}}(\bdmc{G},\widetilde{\bdmc{G}}, \by, \bC_n^{-1}, \mu) = \norm{\exp(-\widetilde{\Phi^{\by}}(\cdot))}_{L^{\infty}_{\mu}(\scrM)}\norm{\frac{1}{2}\left(\bdmc{G}(\cdot)+\widetilde{\bdmc{G}}(\cdot)\right)-\by}_{L^2_{\mu}(\scrM;\scrY)}\,.
\end{equation*}
See proof by \citet[Theorem 1]{cao2023residual}.

\section{Proofs of \cref{prop:dis_error,prop:kle_error}}\label{app:proof_error}
\begin{proof}[\cref{prop:dis_error}]
    We decompose the operator surrogate approximation error into two parts using a triangle inequality:
    \begin{equation*}
        \norm{\bdmc{G}-\widetilde{\bdmc{G}}}_{L^2_{\mu}(\scrM;\scrY)}\leq \norm{\bdmc{G}-\bdmc{G}_r}_{L^2_{\mu}(\scrM;\scrY)} + \norm{\bdmc{G}_r-\widetilde{\bdmc{G}}}_{L^2_{\mu}(\scrM;\scrY)}\,.
    \end{equation*}
    We examine the second term on the right-hand side of the inequality. First, we have $\Psi_r^*M\sim\mathcal{N}(\bzero, \bI_r)$ for any linear encoder $\Psi_r^*$ defined as in \cref{eq:projector}; see \cref{eq:encoded_white_noise}. Second, notice that $\bdmc{G}_r(m) \equiv \bdmc{G}_r(\widehat{\Psi_r^{\text{DIS}}}\widehat{{\Psi_r^{\text{DIS}}}}^*m)$. Consequently, the neural network error is given by
    \begin{align*}
        \norm{\widetilde{\bdmc{G}}-\bdmc{G}_r}_{L^2_{\mu}(\scrM;\scrY)}^2 &= \mathbb{E}_{\bM_r\sim \mathcal{N}(\bzero,\bI_r)}\left[\norm{\bV\nn(\bM_r) - \bdmc{G}_r\left(\widehat{\Psi_r^{\text{DIS}}}\bM_r\right)}_{\bC_n^{-1}}^2\right]\\
        & = \mathbb{E}_{\bM_r\sim \mathcal{N}(\bzero,\bI_r)}\left[\norm{\nn(\bM_r) - \bV^*\bdmc{G}_r\left(\widehat{\Psi_r^{\text{DIS}}}\bM_r\right)}^2\right]\,.
    \end{align*}
    We further decompose the first term on the right-hand side of the triangle inequality. 
    \begin{itemize}[leftmargin=0pt, label={}]
    \item \textit{Part I: Subspace Poincar\'e inequality.} We follow \citealt[Propositions 2.4]{zahm2020gradient}. Due to the Poincar\'e inequality for $H^1_{\mu}$ (\cref{thm:poincare} and \citealt[5.5.6]{bogachev1998gaussian}), for any $\calS\in H^1_{\mu}(\scrM)\coloneqq H^1_{\mu}(\scrM;\R)$ and any pair of encoder $\Psi_r^*$ and decoder $\Psi_r$ as defined in \cref{eq:projector}, we have:
    \begin{equation*}
        \norm{\calS-\calS_r}_{L^2_{\mu}(\scrM)}^2\leq \norm{(\calI_{\scrH_{\mu}}-\Psi_r\Psi_r^*)D_{\scrH}\calS}_{L^2_{\mu}(\scrM;\scrH_{\mu})}^2\,,
        \end{equation*}
    where $\calS_r$ is the $L^2_{\mu}(\scrM)$-optimal reduced mapping of $\calS$ for the given encoder and decoder, $\sder\calS\in L^2_{\mu}(\scrM;\scrH_{\mu})$ is the $\scrH_{\mu}$-representation of the stochastic derivative of $\calS$. The key to extend the results by \citealt[Propositions 2.4]{zahm2020gradient} is to show that the mapping for any $m'\in\scrM$
    \begin{equation*}
        f:m\mapsto \calS\left(\Psi_r\Psi_r^*m' + (\calI_{\scrM}-\Psi_r\Psi_r^*)m\right)
    \end{equation*}
     has a stochastic derivative of the following form via the chain rule:
    \begin{equation*}
        \sder f(m) = (\calI_{\scrH_{\mu}}-\Psi_r\Psi_r^*)\sder\calS(\Psi_r\Psi_r^*m' + (\calI_{\scrM}-\Psi_r\Psi_r^*)m)\,.
    \end{equation*}
    The $H^1_{\mu}(\scrM)$-Poincar\'e inequality is applied to $f$, which leads to the subspace Poincar\'e inequality
    \item \textit{Part II: Error upper bound.} We follow \citealt[Proposition 2.5]{zahm2020gradient} to our setting. Due to the subspace Poincar\'e inequality, for any pair of encoder $\Psi_r^*$ and decoder $\Psi_r$ as defined by \cref{eq:projector}, we have
    \begin{equation*}
        \norm{\bdmc{G}-\bdmc{G}_r}^2_{L^2_{\mu}(\scrM;\scrY)}\leq \text{Tr}_{\scrH_{\mu}}(\calH_A) - \text{Tr}\left(\Psi_r^*\calH_A\Psi_r\right)\,.
    \end{equation*}
    where $\bdmc{G}_r$ is the $L^2_{\mu}(\scrM;\scrY)$-optimal reduced mapping for the given $\Psi_r^*$ and $\Psi_r$.
    The key to extend the results by \citealt[Proposition 2.5]{zahm2020gradient} is to define $\calS^{(j)}\coloneqq \bv_j^T\bC_n^{-1}\bdmc{G}$ where $\{\bv_j\}_{j=1}^{d_y}$ is a $\scrY$-ONB.
    \begin{equation*}
        \norm{\bdmc{G}-\bdmc{G}_r}^2_{L^2_{\mu}(\scrM;\scrY)} = \sum_{j=1}^{d_y} \norm{\calS^{(j)}-\calS^{(j)}_r}^2_{L^2_{\mu}(\scrM)}\,.
    \end{equation*}
    Applying the subspace inequality to $\calS^{{j}}$ and the transformation between trace and HS norm \cite[Theorem 7.3]{gohberg2012traces}, we have
    \begin{align}
        \norm{\bdmc{G}-\bdmc{G}_r}^2_{L^2_{\mu}(\scrM;\scrY)}
        & \leq \sum_{j=1}^{d_y} \text{Tr}_{\scrH_{\mu}}\bigg((\calI_{\scrH_{\mu}}-\Psi_r\Psi_r^*)\\
        &\quad\mathbb{E}_{M\sim\mu}\left[\sder\bdmc{G}^*\bv_j\bv_j^T\bC_n^{-1}\sder\bdmc{G}\right](\calI_{\scrH_{\mu}}-\Psi_r\Psi_r^*)\bigg)\nonumber\\
        & = \text{Tr}_{\scrH_{\mu}}\left((\calI_{\scrH_{\mu}}-\Psi_r\Psi_r^*)\calH_A(\calI_{\scrH_{\mu}}-\Psi_r\Psi_r^*)\right)\,.\label{eq:orm_upper_bound}
    \end{align}
    \item \textit{Part III: Sampling error.} We follow a line of arguments presented in \cite{lam2020multifideility}. Let $\{\disev{j}\in\R_+, \disbasis{j}\in\scrH_{\mu}\}_{j=1}^{\infty}$ and $\{\widehat{\disev{j}}\in\R_+, \widehat{\disbasis{j}}\in\scrH_{\mu}\}_{j=1}^{\infty}$ denote the eigendeompositions of $\calH_A \coloneqq \mathbb{E}_{M\sim\mu}[\calH(m)]$ and $\widehat{\calH}$ with decreasing eigenvalues and $\scrH_{\mu}$-orthonormal basis. Let $\Psi_r^{\text{DIS}},\widehat{\Psi_r^{\text{DIS}}}\in \HS(\R^r, \scrH_{\mu})$ be the linear decoder defined using the first $r$ eigenbases. Then, we can deduce the optimal low-rank approximation of $\calH_A$ and $\widehat{\calH}$ using the Courant min--max principle:
        \begin{equation*} 
            \Psi_r^{\text{DIS}} \in \argmax_{\stackrel{\small\calU_r\in\HS(\R^r, \scrH_{\mu})}{\small\calU_r^*\calU_r = \bI_r}}\text{Tr}\left(\calU_r^*\calH_A\calU_r\right)\,,\quad\widehat{\Psi_r^{\text{DIS}}} \in \argmax_{\stackrel{\small\calU_r\in\HS(\R^r, \scrH_{\mu})}{\small\calU_r^*\calU_r = \bI_r}}\text{Tr}\left(\calU_r^*\widehat{\calH}\calU_r\right)\,.
        \end{equation*}
        Assume $\widehat{\calH}-\calH$ can be decomposed to $\calV\calD\calV^*$, where $\calV\in\HS(l^2,\scrH_{\mu})$ has columns consisting of $\scrH_{\mu}$-orthonormal eigenbases and $\calD \in l^2$ consists of eigenvalues, the cyclic property of trace leads to
        \begin{align*}
            \text{Tr}(\calU_r^*(\widehat{\calH}-\calH_A)\calU_r) &=  \text{Tr}(\calU_r^*\calV\calD\calV^*\calU_r) \\
            &\leq \norm{\calH-\widehat{\calH}}_{B(\scrH_{\mu})}\text{Tr}_{\scrH_{\mu}}(\calV^*\calU_r\calU_r^*\calV) = r\norm{\calH_A-\widehat{\calH}}_{B(\scrH_{\mu})}\,,
        \end{align*}
        where $\calU_r\in\HS(\R^r,\scrH_{\mu})$ is arbitrary and has columns consisting of $\scrH_{\mu}$-orthonormal reduced bases. Applying the inequality above twice, we have the following upper bound of the approximation error to the optimal reduced mapping given the pair of decoder and encoder $\widehat{\Psi_r^{\text{DIS}}}$ and $\widehat{\Psi_r^{\text{DIS}}}^*$:
        \begin{align*}
             \norm{\bdmc{G}-\bdmc{G}_r}^2_{L^2_{\mu}(\scrM;\scrY)}&\leq \text{Tr}_{\scrH_{\mu}}(\calH_A) - \text{Tr}\left(\widehat{\Psi_r^{\text{DIS}}}^*\calH_A\widehat{\Psi_r^{\text{DIS}}}^*\right)\\
             & \leq \text{Tr}_{\scrH_{\mu}}(\calH_A) - \text{Tr}\left(\Psi_r^{\text{DIS}}\calH_A(\Psi_r^{\text{DIS}})^*\right)  + 2r\norm{\calH_A-\widehat{\calH}}_{B(\scrH_{\mu})}\\
             &=\sum_{j=r+1}^{\infty} \disev{j} + 2r\norm{\calH_A-\widehat{\calH}}_{B(\scrH_{\mu})}\,.
        \end{align*}
    \end{itemize}
\end{proof}
\begin{proof}[\cref{prop:kle_error}]
    We follow the same arguments by \citealt[Porposition 3.1]{zahm2020gradient}. From \cref{eq:orm_upper_bound} and the definition of KLE, the approximation error for the optimal reduced mapping defined by the KLE reduced bases $\Psi_r^{\text{KLE}}\in \HS(\R^r,\scrH_{\mu})$ is given by
    \begin{align*}
        \norm{\bdmc{G}-\bdmc{G}_r}^2_{L^2_{\mu}(\scrM;\scrY)} &\leq \text{Tr}_{\scrH_{\mu}}\left((\calI_{\scrH_{\mu}}-\Psi_r^{\text{KLE}}{\Psi_r^{\text{KLE}}}^*)\calH_A(\calI_{\scrH_{\mu}}-\Psi_r^{\text{KLE}}{\Psi_r^{\text{KLE}}}^*)\right)\\
        &\leq \norm{\calH_A}_{B(\scrH_{\mu})}\mathbb{E}_{M\sim\mu}\left[\norm{\left(\calI_{\scrM}-\Psi_r^{\text{KLE}}{\Psi_r^{\text{KLE}}}^*\right)M}_{\scrM}^2\right]\\
        & =  \norm{\calH_A}_{B(\scrH_{\mu})}\sum_{j=r+1}^{\infty} \left(\kleev{j}\right)^2\,.
    \end{align*}
    By \cref{ass:stochastic_derivative}, we have the following bound $\mu$-a.e.
    \begin{align*}
        \norm{\sder\bdmc{G}(m)}_{B(\scrH,\scrY)}\leq\sup_{\stackrel{\delta m\in\scrH_{\mu}}{\norm{\delta m}_{\scrH_{\mu}} = 1}}\lim_{t\to 0} \norm{t^{-1}\left(\bdmc{G}(m+t\delta m)-\bdmc{G}(m)\right)}_{\bC_n^{-1}}\leq c_{\bdmc{G}}\,.
    \end{align*}
    Thus we have
    \begin{equation*}
        \norm{\calH_A}_{B(\scrH_{\mu})} = \sup_{\stackrel{\delta m\in\scrH_{\mu}}{\norm{\delta m}_{\scrH_{\mu}}=1}} \mathbb{E}_{M\sim\mu}\left[\norm{\sder\bdmc{G}(M)\delta m}^2_{\bC_n^{-1}}\right]\leq c_{\bdmc{G}}^2\,.
    \end{equation*}
    Due to \cref{eq:orm_upper_bound}, the minimized upper bound is achieved when $\Psi_r=\Psi_r^{\text{DIS}}$. Therefore, we have 
    \begin{equation*}
        \sum_{j=r+1}^{\infty}\disev{j}\leq \norm{\calH_A}_{B(\scrH_{\mu})}\sum_{j=r+1}^{\infty} \left(\kleev{j}\right)^2\leq c_{\bdmc{G}}^2 \sum_{j=r+1}^{\infty} \left(\kleev{j}\right)^2\,.
    \end{equation*}
\end{proof}

\section{Proofs of \cref{lemm:gauss_rv,prop:splitting}}\label{app:proof_splitting}
\begin{proof}[\cref{lemm:gauss_rv}]
    The statement on transforming Gaussian random elements is standard; see, e.g., \citealt[Proposition 1.2.3]{dapratto2002second}. Since the $\scrM$-adjoint of $\Psi_r^*$ is $\cpr^{-1}\Psi_r$, we have
    \begin{equation}\label{eq:encoded_white_noise}
        \Psi^*M\sim \mathcal{N}(0, \Psi^*\cpr\cpr^{-1}\Psi_r) = \mathcal{N}(0, \bI_r)\,.
    \end{equation}
    We focus on the statement on the independence of two random elements $M_r\perp M_{\perp}$ given by:
    \begin{equation*}
    \begin{cases}
        M_r = \Psi_r\Psi_r^*M\\
        M_{\perp}\coloneqq(\calI_{\scrM}-\Psi_r\Psi_r^*)M
    \end{cases}\,,\quad M\sim \mu\,.
    \end{equation*}
    We examine the characteristic function for the joint random element $X = (M_{r}, M_{\perp})$ and equip the product space $\scrM\times\scrM$ with an extended inner product:
    \begin{equation*}
        \left\langle(t_r, t_{\perp}), (s_r, s_{\perp})\right\rangle_{\scrM\times\scrM} = \left\langle t_r, s_r\right\rangle_{\scrM} + \left\langle t_{\perp}, s_{\perp}\right\rangle_{\scrM}\,.
    \end{equation*}
    We have the following form for the characteristic function of $X$:
    \begin{align*}
        &\quad\;\mathbb{E}_{M\sim\mu}\left[\exp\left(i\left\langle\Psi_r\Psi_r^*M, t_r\right\rangle_{\scrM}\right)\exp\left(i\left\langle\left(\mathcal{I}_{\scrM}-\Psi_r\Psi_r^*\right)M, t_{\perp}\right\rangle_{\scrM}\right)\right]\\
        &= \mathbb{E}_{M\sim\mu}\left[\exp\left(i\left\langle M, (\Psi_r\Psi_r^*)^*(t_r-t_{\perp}) + t_{\perp}\right\rangle_{\scrM}\right)\right] && (\text{Def. of $\scrM$-adjoint})\\
        & = \exp\left(\left\langle\cpr\left((\Psi_r\Psi_r^*)^*(t_r-t_{\perp}) + t_{\perp}\right), (\Psi_r\Psi_r^*)^*(t_r-t_{\perp}) + t_{\perp}\right\rangle_{\scrM}\right) && (\text{Def. of charac. func.} )\\
        & = \exp\left(\left\langle\Psi_r\Psi_r^*\cpr(t_r-t_{\perp}) + \cpr t_{\perp}, (\Psi_r\Psi_r^*)^*(t_r-t_{\perp}) + t_{\perp}\right\rangle_{\scrM}\right) && (\text{Explicit }\scrM\text{-adjoint})\\
        & = \exp\left(\left\langle\Psi_r\Psi_r^*\cpr t_r, t_r\right\rangle_{\scrM}\right)\exp\left(\left\langle\left(\calI_{\scrM}-\Psi_r\Psi_r^*\right)\cpr t_{\perp}, t_{\perp}\right\rangle_{\scrM}\right) && (\text{Cancel cross terms})
    \end{align*}
    Therefore, by the definition of the characteristic function for Gaussian measures, we have
    \begin{equation*}
        X\sim\mathcal{N}\left(\mathbf{0}, \begin{bmatrix}
            \Psi_r\Psi_r^*\cpr & 0\\
            0 & \left(\calI_{\scrM}-\Psi_r\Psi_r^*\right)\cpr
        \end{bmatrix}\right)\,,
    \end{equation*}
    and thus $M_r$ and $M_{\perp}$ are independently distributed.
\end{proof}

\begin{proof}[\cref{prop:splitting}]
    We show that the surrogate mMALA proposal $\widetilde{\mathcal{Q}_{\mmala}}(m,\cdot)$ at a given $m$ can be defined through a deterministic transformation of the prior following \cref{lemm:gauss_rv}. Consider $\widetilde{\calT}(m)\in B(\scrM)$ given by
    \begin{equation*}
        \widetilde{\calT}(m) = \calI_{\scrM} + \widetilde{\Psi_r}\left(\left((\widetilde{d}_j+1)^{1/2}-1\right)\delta_{jk}\right)\widetilde{\Psi_r}^*\,,
    \end{equation*}
    where $\widetilde{\Psi_r}$, $\widetilde{\Psi_r}^*$, and $\widetilde{d}_j$ are defined in \cref{subsec:surrogate_geometry}. The covariance of the local Gaussian approximation of the posterior in \cref{eq:surrogate_cpost} can be expressed as
    \begin{equation}\label{eq:posterior_trasnform}
        \widetilde{\cpo}(m) = \widetilde{\calT}(m)\cpr\widetilde{\calT}(m)^*.
    \end{equation}
   The key to validating \cref{eq:posterior_trasnform} is to take the adjoint of $\widetilde{\Psi_r}$ and $\widetilde{\Psi_r}^*$ in $\scrM$ when taking the adjoint of $\calT(m)$. In particular, the $\scrM$-adjoint of $\widetilde{\calT}(m)$ is given by
    \begin{equation*}
        \widetilde{\calT}(m)^* = \calI_{\scrM} + \cpr^{-1}\widetilde{\Psi_r}\left(\left((\widetilde{d}_j+1)^{1/2}-1\right)\delta_{jk}\right)\widetilde{\Psi_r}^*\cpr\,.
    \end{equation*}
    The covariance transformation given by $\widetilde{\calT}(m)$ leads to
    \begin{equation*}
        M^{\dagger}\sim \widetilde{\calQ_{\mmala}}(m,\cdot)\quad\text{and}\quad M^{\dagger} = sm + (1-s)\widetilde{\calA}(m) + \sqrt{1-s^2}\widetilde{\calT}(m)M\,,\quad M\sim\mu\,.
    \end{equation*}
    Since the $M$ can be independently decomposed into two parts using the encoder $\Psi_r^*$ and decoder $\Psi_r$ due to \cref{lemm:gauss_rv}, the proposal distribution can also be decomposed into two parts independently.
\end{proof}

\section{Diagnostics, tuning and initialization}\label{app:step_size}

\subsection{On the MPSRF diagnostic}

The MPSRF is typically defined as follows:
\begin{align*}
        \widehat{R} = \sqrt{\max_{\stackrel{m\in\scrM \text{ s.t.}}{\norm{m}_{\scrM}=1}}\frac{\left\langle m, \widehat{\calW}_sm\right\rangle_{\scrM}}{\left\langle m, \widehat{\calV}_sm\right\rangle_{\scrM}}}\,. && (\text{Conventional MPSRF})
\end{align*}
where $\widehat{\calW}_s$ and $\widehat{\calV}_s$ are defined in \cref{eq:covariance_estimator}. However, this quantity is not well-defined since the two empirical covariance operators are singular for a finite sample size $n_s$. During computation on a discretized parameter space, a pool of long MCMC chains is often needed to estimate this quantity. Moreover, the MPSRF characterizes Markov chains along a single slice of the parameter space, which is sufficient for monitoring convergence but insufficient for comparing the quality of Markov chains generated by different MCMC methods. 

\subsection{On tuning and initialization}

We provide the procedure for determining the step size parameter $\triangle t$ for a given Bayesian inverse problem. The procedure is consistent across all MCMC algorithms in \cref{tab:mcmc_list,tab:dino_mcmc_list}. It is designed to maximize the sampling performance of MCMC methods while maintaining uniform behaviors of the MCMC chain across different regions of the parameter space. 

First, we choose a list of candidate values for $\triangle t$. Then, we generate an MCMC chain $\{m_j\}_{j=1}^{n_s}$ with $n_s=5000$ samples (after burn-in) for each candidate value. Using these samples, we compute the acceptance rate (AR) and the mean square jump (MSJ) given by
\begin{equation*}
    \text{AR}\coloneqq \frac{n_{\text{accept}}}{n_{s}}\times 100\%\,,\quad\text{MSJ} \coloneqq \frac{1}{n_s-1}\sum_{j=1}^{n_s-1}\norm{m_{j+1}-m_j}^2_{\scrM}\,,
\end{equation*}
where $n_{\text{accept}}$ is the number of accepted proposal samples in the MCMC chain. We down-select the candidate values by choosing a maximum step size $\triangle t_{\text{max}}>0$ such that AR monotonically decreases and MSJ monotonically increases as a function of $\triangle t\in[0, \triangle t_{\text{max}}]$. Finally, we choose a tuned step size value from the remaining candidates that maximizes the median of the single chain version of the ESS\% in \cref{eq:ess}.

To make the step size tuning more efficient, we initialize the chains with samples from the Laplace approximation to reduce the number of burn-in samples. Once the step size tuning is complete, we initialize the subsequent MCMC runs using samples obtained from step size tuning.
\begin{remark}
    Due to the high non-linearity of Bayesian inversion in our numerical examples, the step size tuning procedure introduced above is often ineffective for MALA and LA-pCN. In particular, the ARs for MCMC chains initialized at different positions have large variations when the tuned step size is employed. In such cases, the MCMC chains are often stuck in local regions with low ARs; see \cref{fig:la-pcn_large_step}. We reduce the step size to eliminate this behavior until the AR variations are within $\pm 5\%$ of the averaged value.
\end{remark}
\begin{figure}[tbp]
    \centering
    \includegraphics[width = 0.4\linewidth]{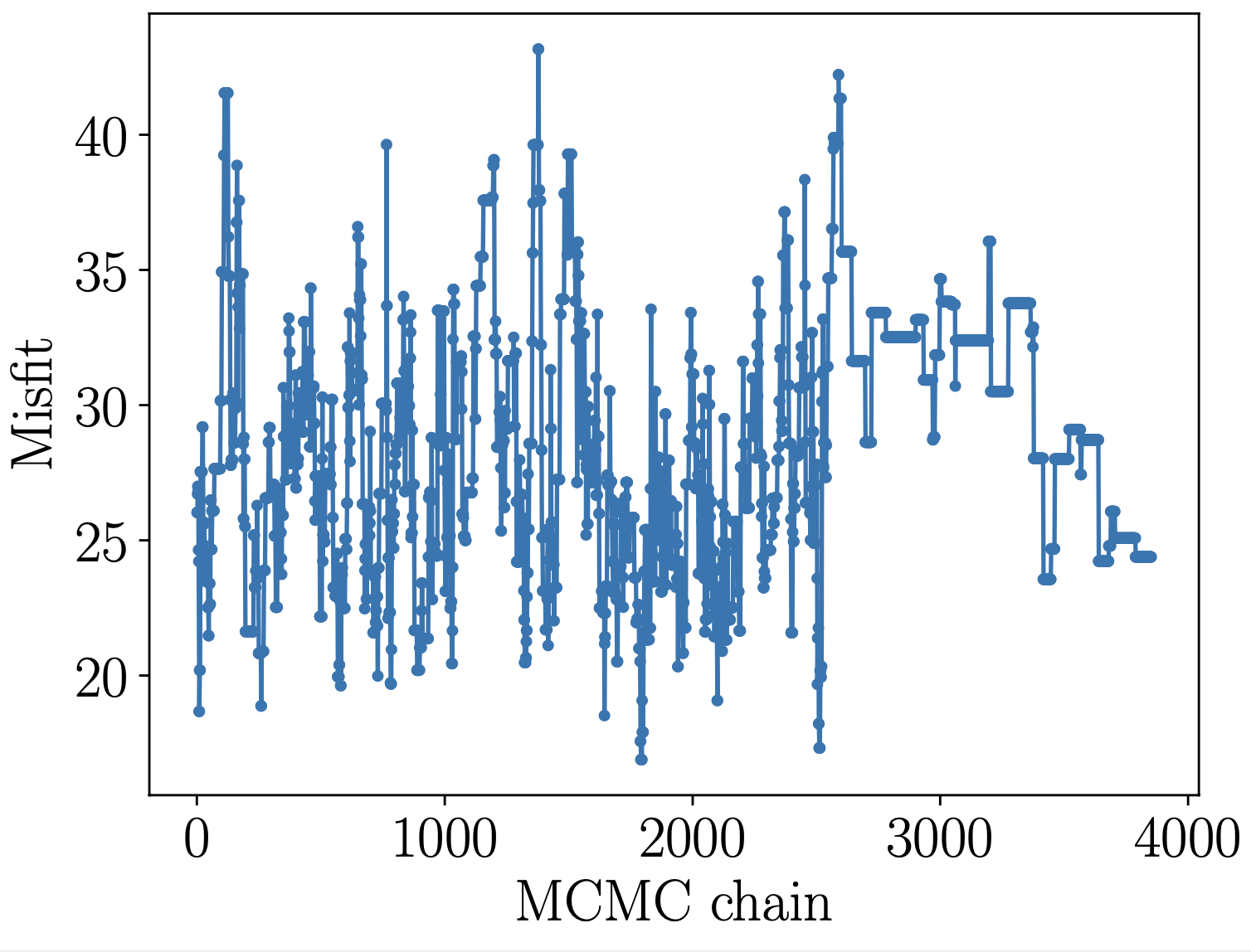}
    \caption{The trace plot of data misfit values along a Markov chain generated by LA-pCN at large step size for coefficient inversion in a nonlinear diffusion--reaction PDE. The Markov chain behaves drastically differently in different parts of the chain, leading to a large bias in posterior estimation.}
    \label{fig:la-pcn_large_step}
\end{figure}

\section{Supplementary materials for the numerical examples}\label{sec:supplementary}

In \cref{fig:ndr_basis}, we visualize selected DIS and KLE basis functions for coefficient inversion in a nonlinear diffusion--reaction PDE. In \cref{tab:ndr_mcmc_stats}, we list the statistics for the MCMC runs. In \cref{fig:ndr_posterior}, we visualize the posterior samples, mean estimate via MCMC using mMALA, the absolute error between MAP estimate and mean estimate, pointwise variance estimate via MCMC using mMALA, and the absolute error of pointwise variance estimate from LA and MCMC. The same visualization and statistics for hyperelastic material deformation is provided in \cref{fig:hyper_basis,fig:hyper_posterior,tab:hyper_mcmc_stats}. We emphasize that the large error between the MCMC mean estimate and the MAP estimate indicates a non-Gaussian posterior distribution.

\begin{figure}[tbp]
    \centering
    \begin{tabular}{|c| c | c | c | c|} \hline
       &\bf \#1  &\bf \#2 &\bf \#4 &\bf \#8 \\\hline
       \bf DIS & \raisebox{-.5\height}{\includegraphics[width = 0.15\linewidth]{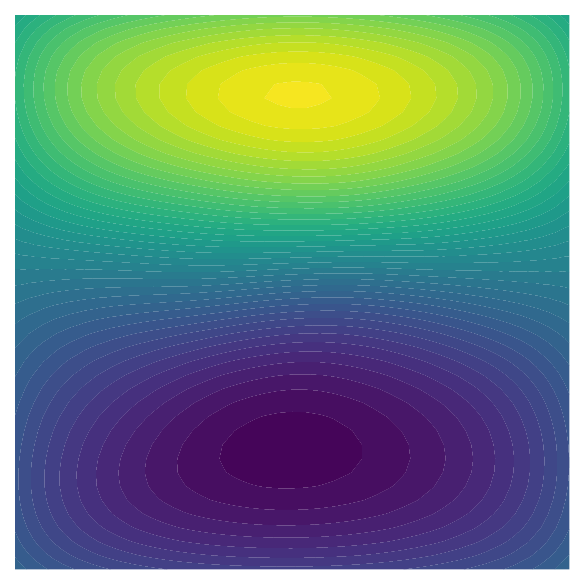}} & \raisebox{-.5\height}{\includegraphics[width = 0.15\linewidth]{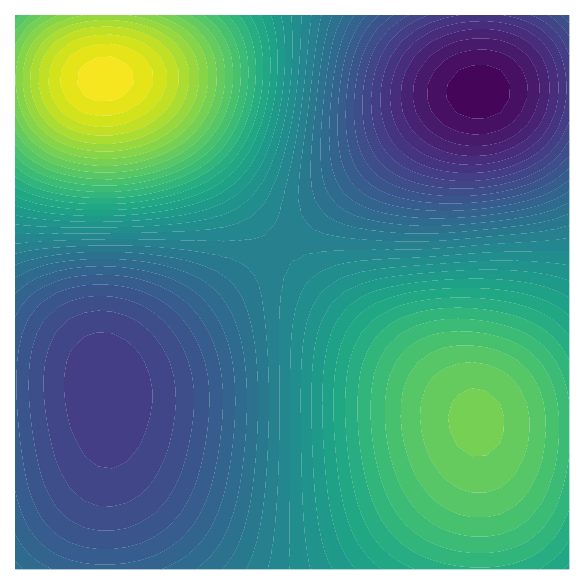}} & \raisebox{-.5\height}{\includegraphics[width = 0.15\linewidth]{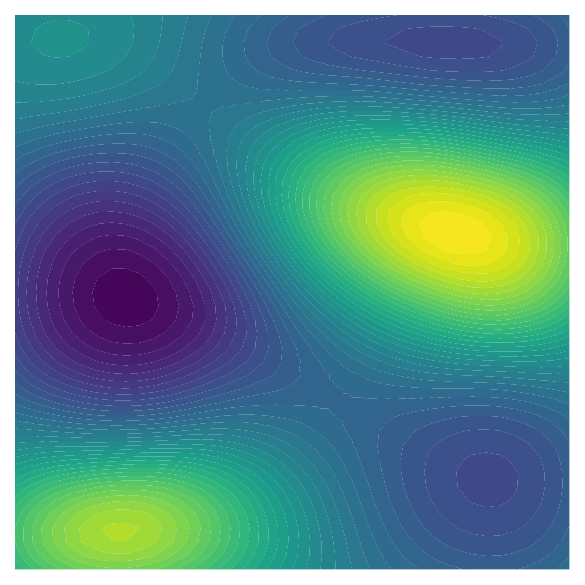}} & \raisebox{-.5\height}{\includegraphics[width = 0.15\linewidth]{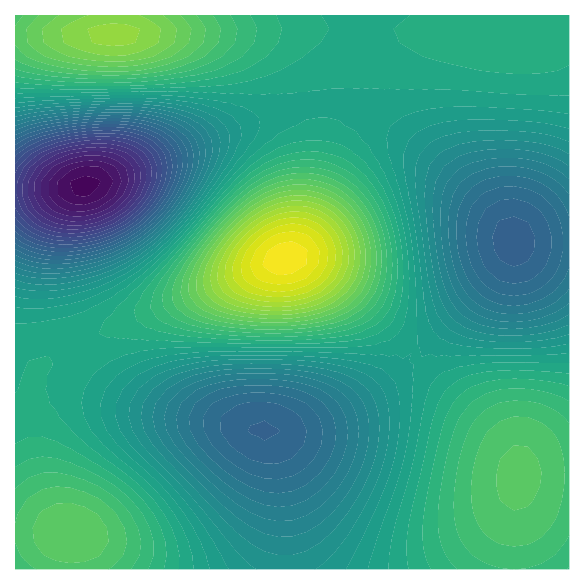}} \\\hline
       \bf KLE  & \raisebox{-.5\height}{\includegraphics[width = 0.15\linewidth]{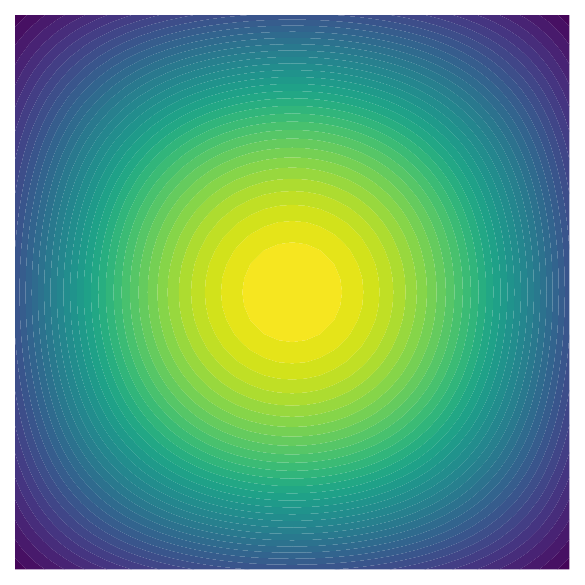}} & \raisebox{-.5\height}{\includegraphics[width = 0.15\linewidth]{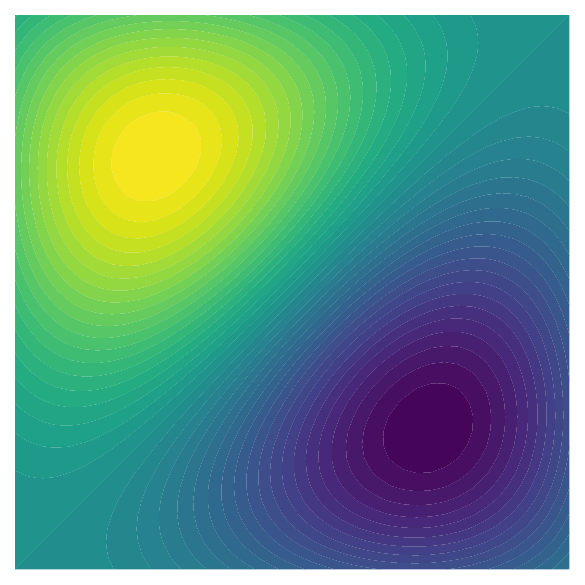}} & \raisebox{-.5\height}{\includegraphics[width = 0.15\linewidth]{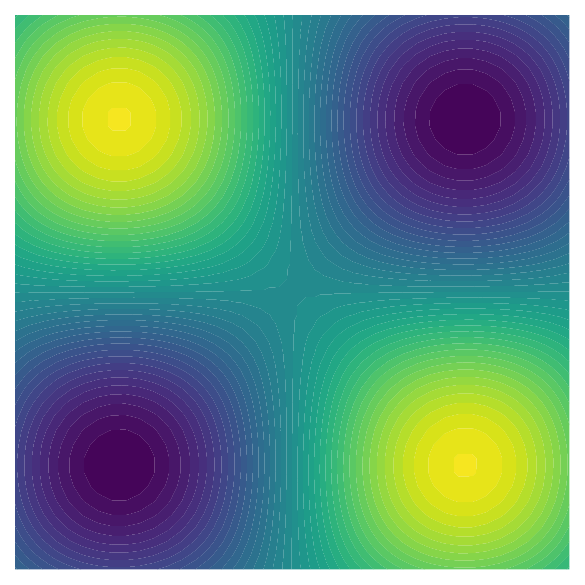}} & \raisebox{-.5\height}{\includegraphics[width = 0.15\linewidth]{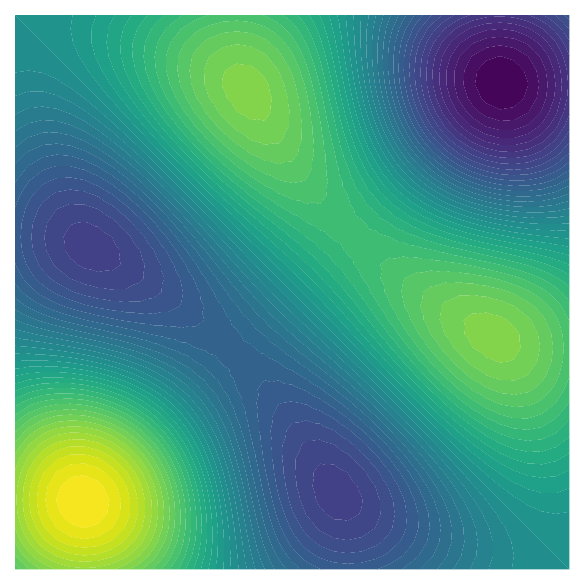}}  \\\hline
    \end{tabular}
    \caption{Visualization of selected DIS and KLE basis functions for coefficient inversion in a nonlinear diffusion--reaction PDE in \cref{sec:ndr}}
    \label{fig:ndr_basis}
\end{figure}

\begin{table}[tbp]
    \centering
    \begin{tabular}{|c | c | c | c | c |}\hline
       \bf Name  & \makecell{\bf Step size $\triangle t$\\ $\times 10^{-2}$} & \bf \makecell{Acceptance rate\\$\%$}& \bf \makecell{Mean square\\ jump $\times 10^{-3}$}  \\\hline
       pCN & 1 & 16 & 1 \\\hline
       MALA & 0.36 & 39 & 0.96 \\\hline
       LA-pCN & 2.3 & 60 & 6.0 \\\hline
       DIS-mMALA & 11 & 44 & 11 \\\hline
       mMALA & 15 & 20 & 12\\\hline
       \makecell{NO-mMALA\\$n_t=16\times 10^3$} & 1.8 & 26 & 2.1 \\\hline
       \makecell{DINO-mMALA\\$n_t= (2, 4, 8, 16)$\\$\times 10^{3}$} & $(12, 12, 13, 15)$ & $(19, 20, 20, 17)$  & $(9.9, 10, 11, 11)$ \\\hline
       \makecell{DA-NO-mMALA\\$n_t= (1, 2, 4, 8, 16)$\\$\times 10^{3}$} & \makecell{$(1, 1, 1.3,$ \\$1.5, 1.5, 1.8)$} & \makecell{$1$st: $(63, 67, 67, 70, 68)$\\$2$nd: $(11, 23, 28, 27, 35)$}  & \makecell{$(0.23, 0.45, 0.61,$\\ $0.87, 1.1)$} \\\hline
       \makecell{DA-DINO-mMALA\\$n_t= (1,2, 4, 8, 16)$\\$\times 10^{3}$} & $(11, 11, 12, 13, 14)$  & \makecell{$1$st: $(40, 29, 25, 22, 19)$\\$2$nd: $(42, 57, 67, 72, 76)$} & \makecell{$(4.6, 6.8, 7.9,$\\$ 8.7, 9.2)$} \\\hline
    \end{tabular}
    \caption{The statistics of MCMC runs for coefficient inversion in a nonlinear diffusion--reaction PDE in \cref{sec:ndr}.}
    \label{tab:ndr_mcmc_stats}
\end{table}

\begin{figure}[tbp]
    \centering
    \renewcommand{\arraystretch}{0.2}
    \addtolength{\tabcolsep}{-3pt}
    \begin{tabular}{c c c c}
    \multicolumn{4}{c}{\bf Posterior samples}\\
     \raisebox{-.5\height}{\includegraphics[width = 0.2\linewidth]{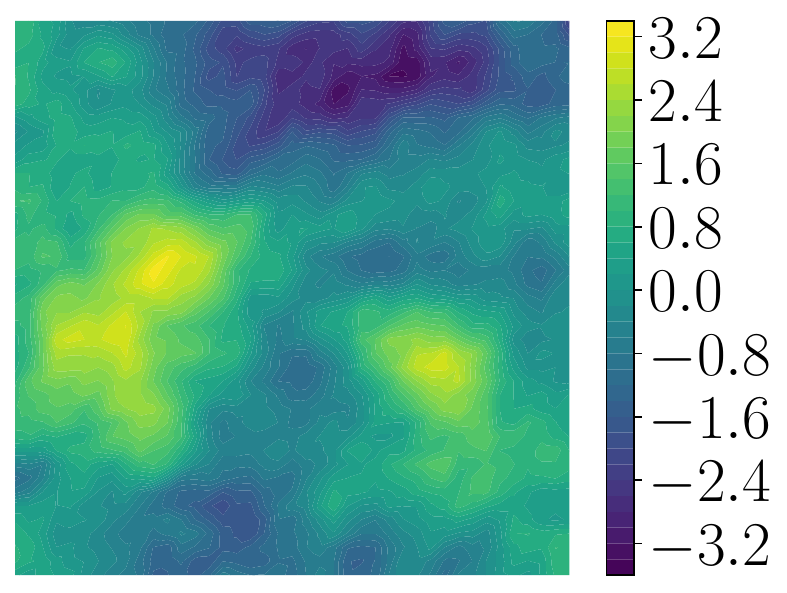}}    & 
     \raisebox{-.5\height}{\includegraphics[width = 0.2\linewidth]{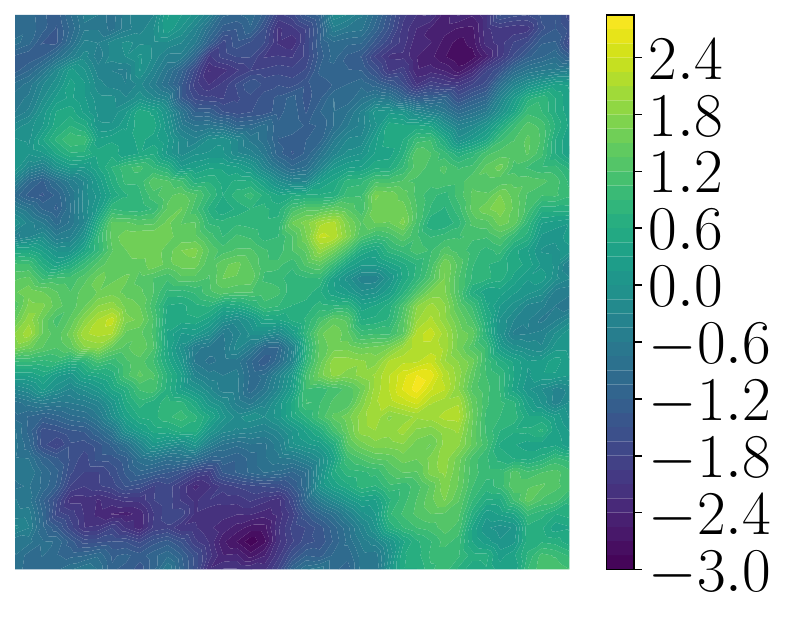}}  &
     \raisebox{-.5\height}{\includegraphics[width = 0.2\linewidth]{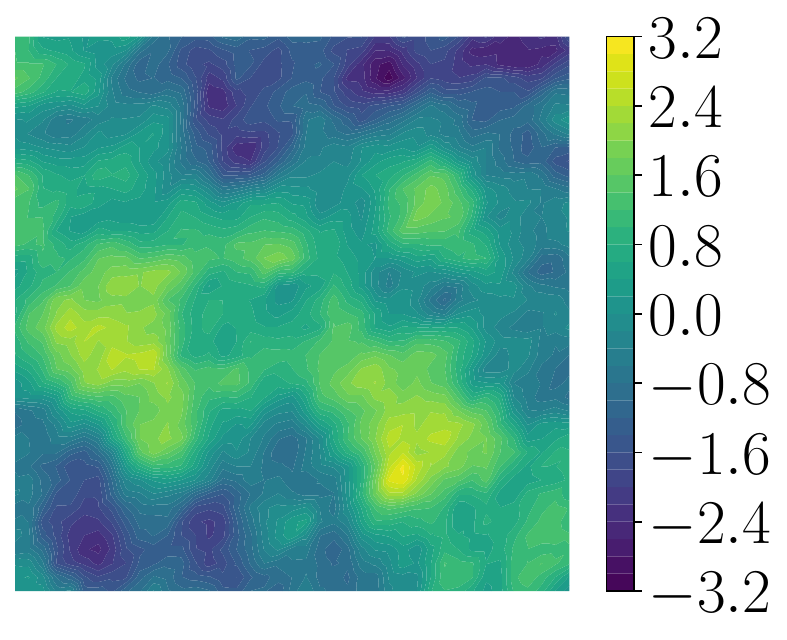}} &
     \raisebox{-.5\height}{\includegraphics[width = 0.2\linewidth]{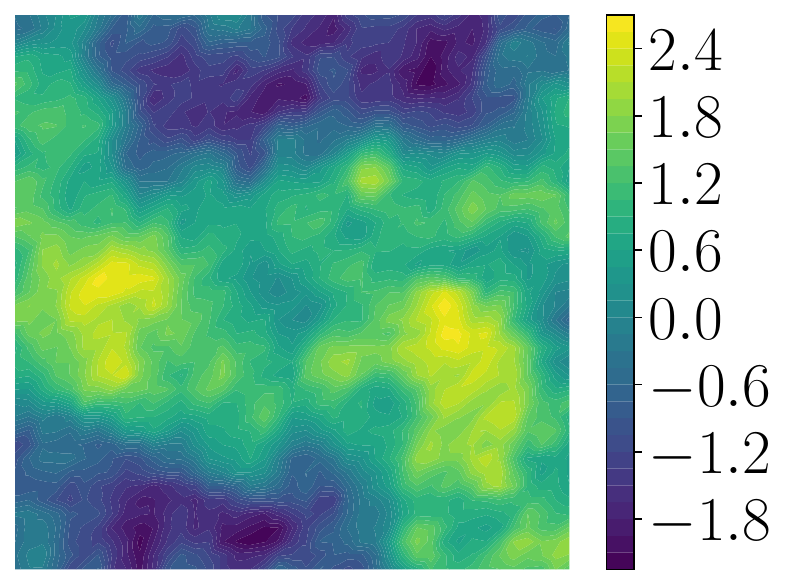}} \\
    
    \end{tabular}
    \begin{tabular}{c c c c}
        \bf \makecell{Mean estimate\\ via MCMC} & \bf \makecell{MAP--mean\\ absolute error} &  \bf \makecell{Pointwise variance\\estimate via MCMC} & \bf \makecell{LA--posterior\\ pointwise variance\\ absolute error}\\
        \raisebox{-.5\height}{\includegraphics[width = 0.2\linewidth]{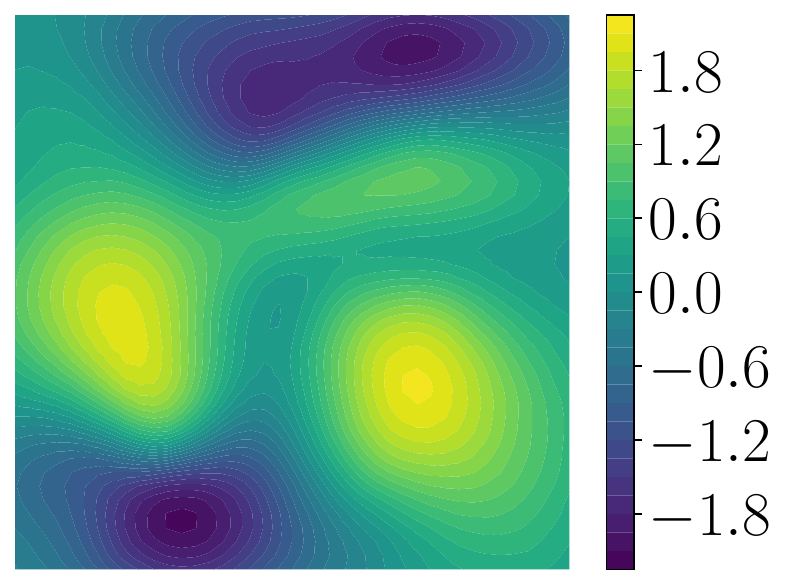}} & \raisebox{-.5\height}{\includegraphics[width = 0.2\linewidth]{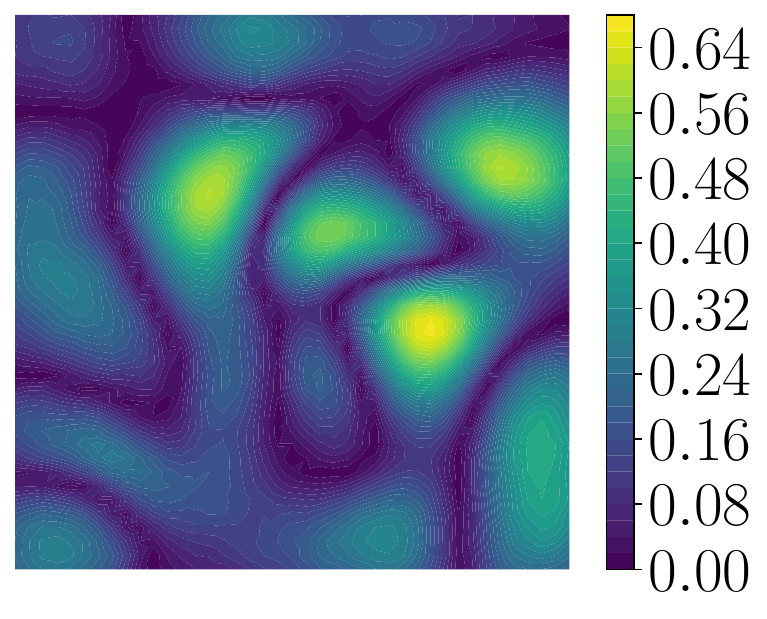}} & \raisebox{-.5\height}{\includegraphics[width = 0.2\linewidth]{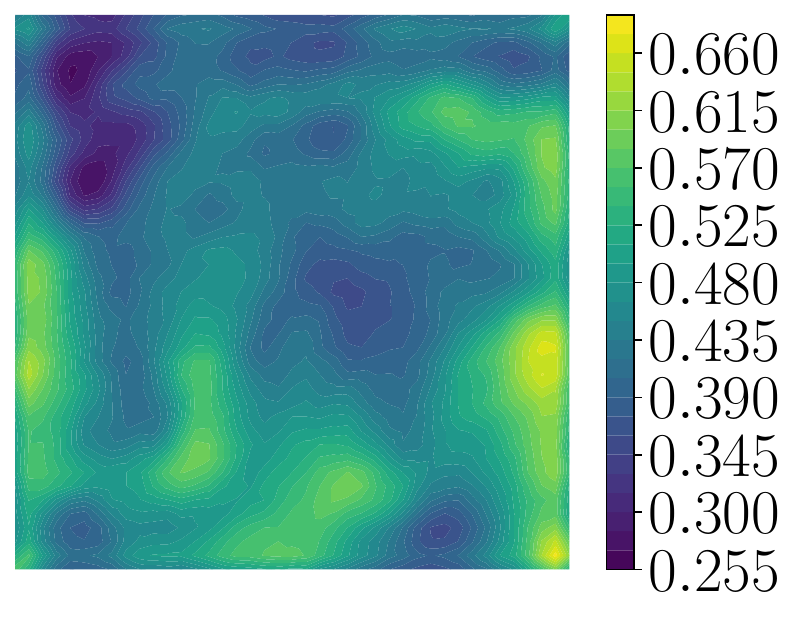}} & \raisebox{-.5\height}{\includegraphics[width = 0.2\linewidth]{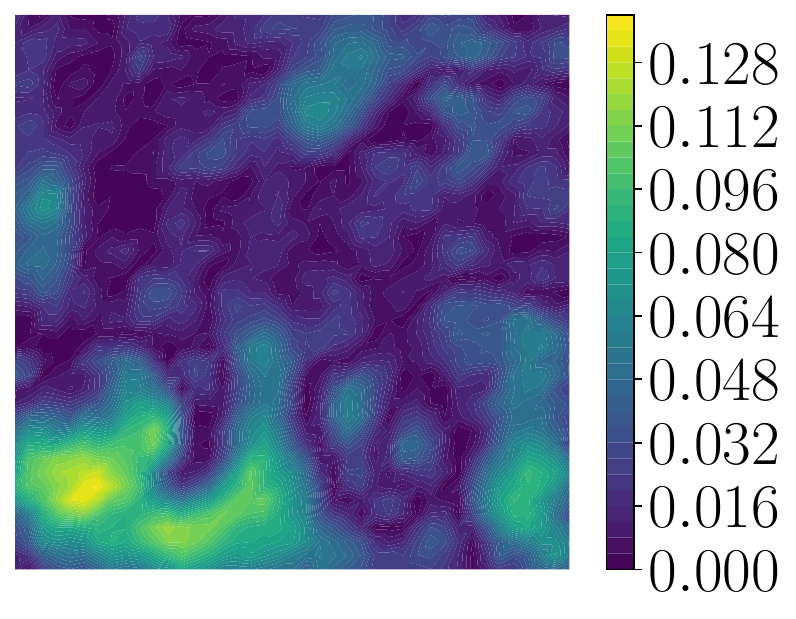}} 
    \end{tabular}
    \renewcommand{\arraystretch}{5}
    \addtolength{\tabcolsep}{3pt}
    \caption{Visualization of relevant statistics from Markov chains collected using mMALA for coefficient inversion in a nonlinear diffusion--reaction PDE in \cref{sec:ndr}.}
    \label{fig:ndr_posterior}
\end{figure}

\begin{figure}[tbp]
    \centering
    \begin{tabular}{|c| c | c | c | c|} \hline
      \bf  &\bf \#1  &\bf \#2 &\bf \#4 &\bf \#8 \\\hline
       \bf DIS & \raisebox{-.5\height}{\includegraphics[width = 0.2\linewidth]{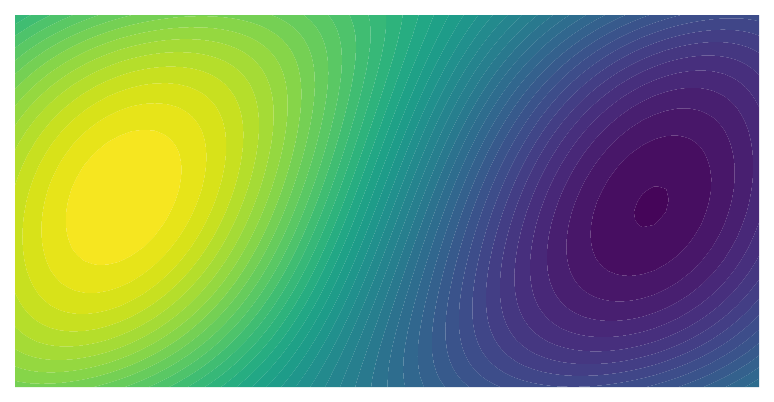}} & \raisebox{-.5\height}{\includegraphics[width = 0.2\linewidth]{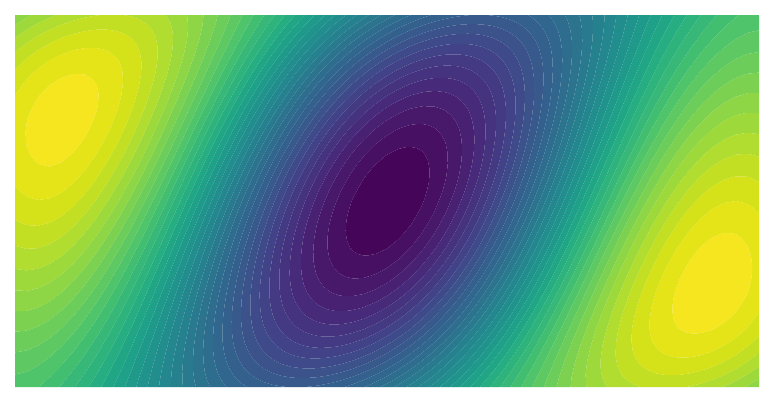}} & \raisebox{-.5\height}{\includegraphics[width = 0.2\linewidth]{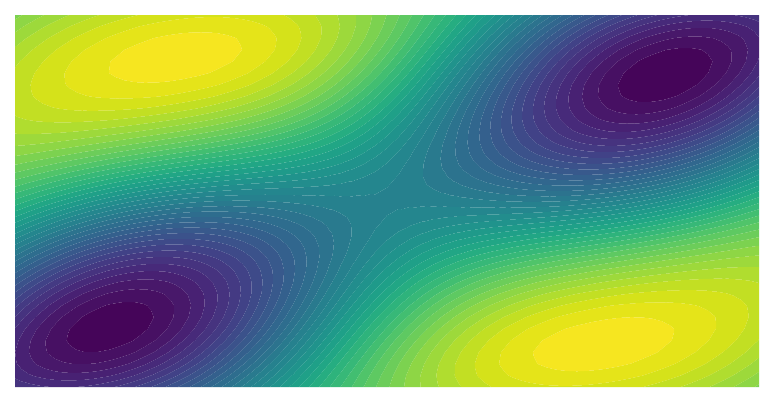}} & \raisebox{-.5\height}{\includegraphics[width = 0.2\linewidth]{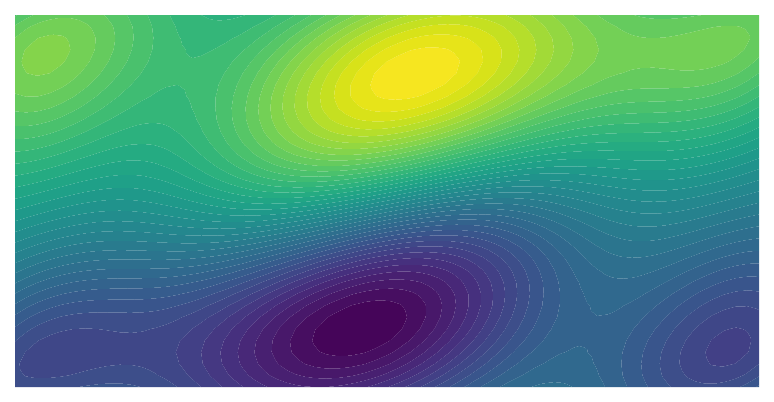}} \\\hline
       \bf KLE  & \raisebox{-.5\height}{\includegraphics[width = 0.2\linewidth]{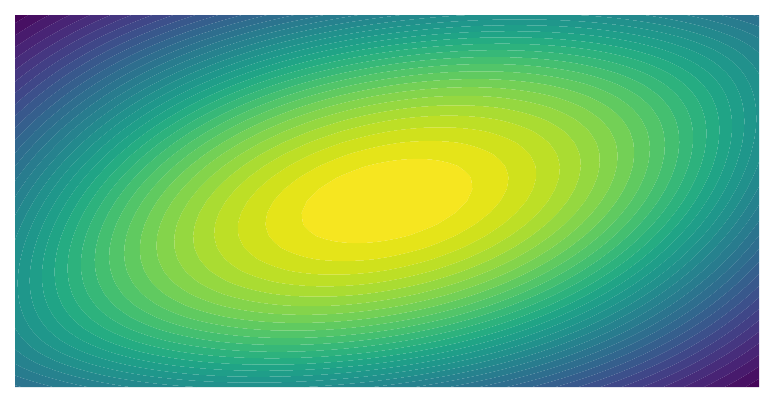}} & \raisebox{-.5\height}{\includegraphics[width = 0.2\linewidth]{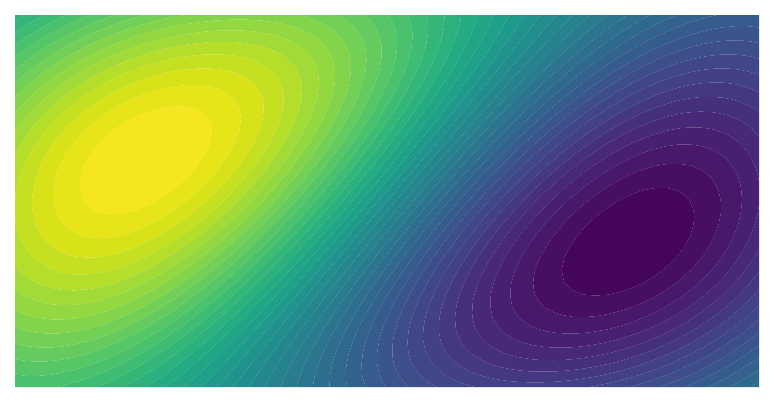}} & \raisebox{-.5\height}{\includegraphics[width = 0.2\linewidth]{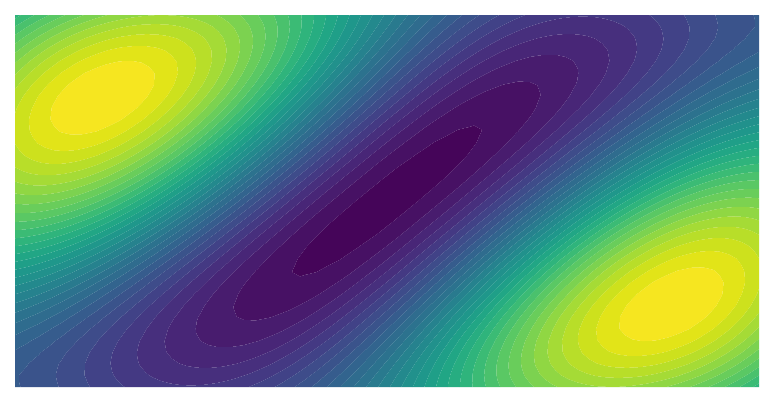}} & \raisebox{-.5\height}{\includegraphics[width = 0.2\linewidth]{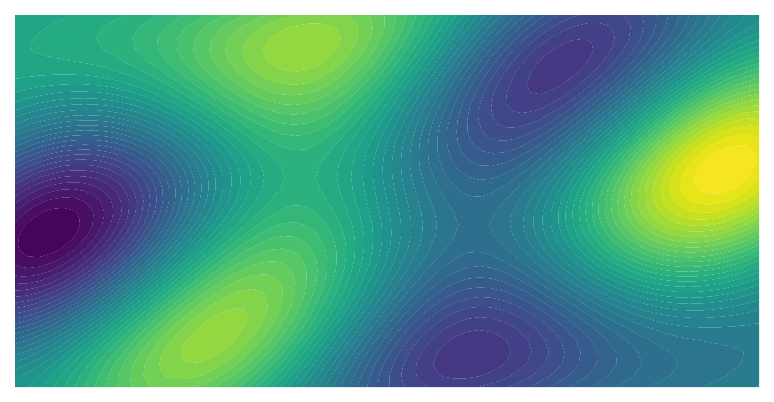}}  \\\hline
    \end{tabular}
    \caption{Visulaization of selected DIS and KLE basis functions for inference of a heterogeneous hyperelastic material property in \cref{sec:hyperelastic}}
    \label{fig:hyper_basis}
\end{figure}

\begin{table}[tbp]
    \centering
    \begin{tabular}{|c | c | c | c | c |}\hline
       \bf Name  & \makecell{\bf Step size $\triangle t$\\ $\times 10^{-2}$} & \bf \makecell{Acceptance rate\\$\%$}& \bf \makecell{Mean square jump\\ $\times 10^{-3}$}  \\\hline
       pCN & 1.0 & 30 & 0.43 \\\hline
       MALA & 0.075 & 78 & 0.40 \\\hline
       LA-pCN & 5.0 & 75 & 2.6 \\\hline
       DIS-mMALA & 150 & 64 & 26 \\\hline
       mMALA & 100 & 34 & 14 \\\hline
       \makecell{DA-NO-mMALA\\$n_t= (5, 10, 20, 40)\times 10^{2}$} & $(0.5, 1, 2, 4)$ & \makecell{$1$st: $(49, 37, 33, 44)$\\$2$nd: $(41, 41, 46, 68)$} & $(0.12, 0.16, 0.27, 0.90)$ \\\hline
       \makecell{DA-DINO-mMALA\\$n_t= (5, 10, 20, 40)\times 10^{2}$} & $(65, 65, 70, 70)$ & \makecell{$1$st: $(38, 40, 38, 39)$\\$2$nd: $(76, 87, 90, 93)$} & $(9.4, 11, 11, 12)$ \\\hline
    \end{tabular}
        \caption{The statistics of MCMC runs for inference of a heterogeneous hyperelastic material property in \cref{sec:hyperelastic}.}
    \label{tab:hyper_mcmc_stats}
\end{table}

\begin{figure}[tbp]
    \centering
    \renewcommand{\arraystretch}{0.2}
    \addtolength{\tabcolsep}{-5pt}
    \begin{tabular}{c c c c}
    \multicolumn{4}{c}{\bf Posterior samples}\\
     \raisebox{-.5\height}{\includegraphics[width = 0.24\linewidth]{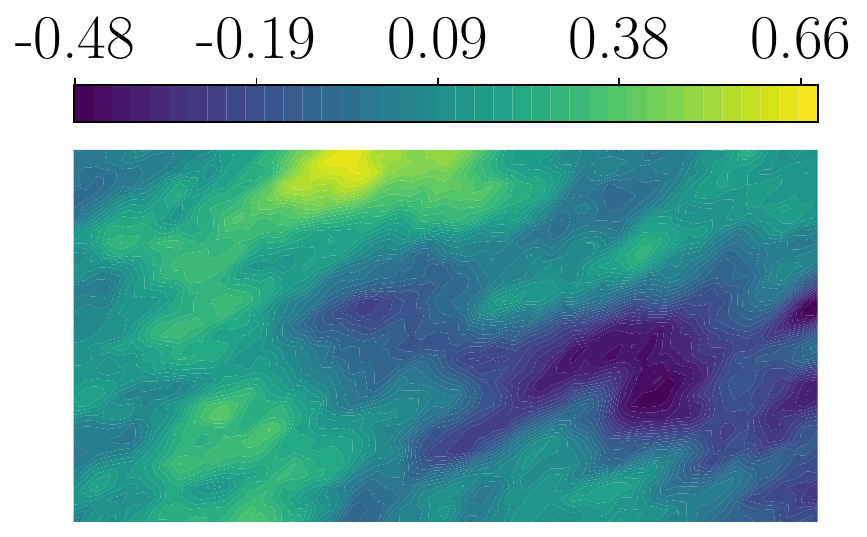}}    & 
     \raisebox{-.5\height}{\includegraphics[width = 0.24\linewidth]{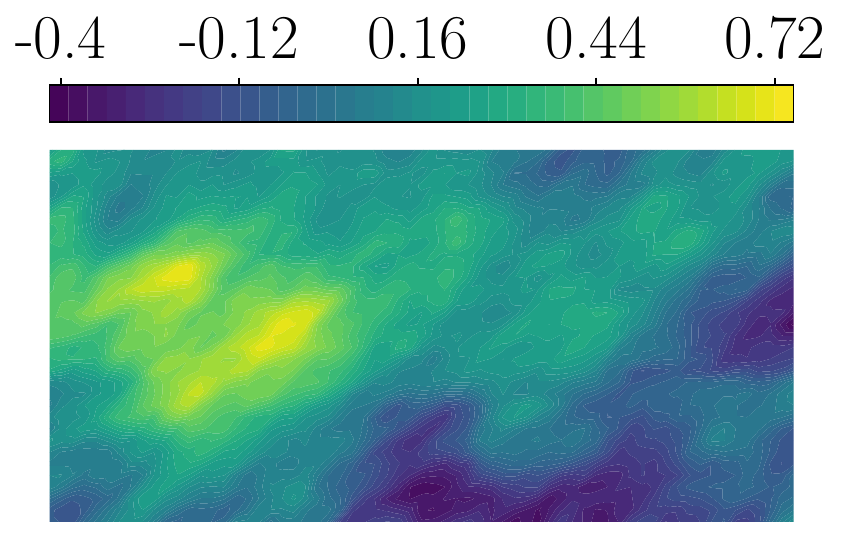}}  &
     \raisebox{-.5\height}{\includegraphics[width = 0.24\linewidth]{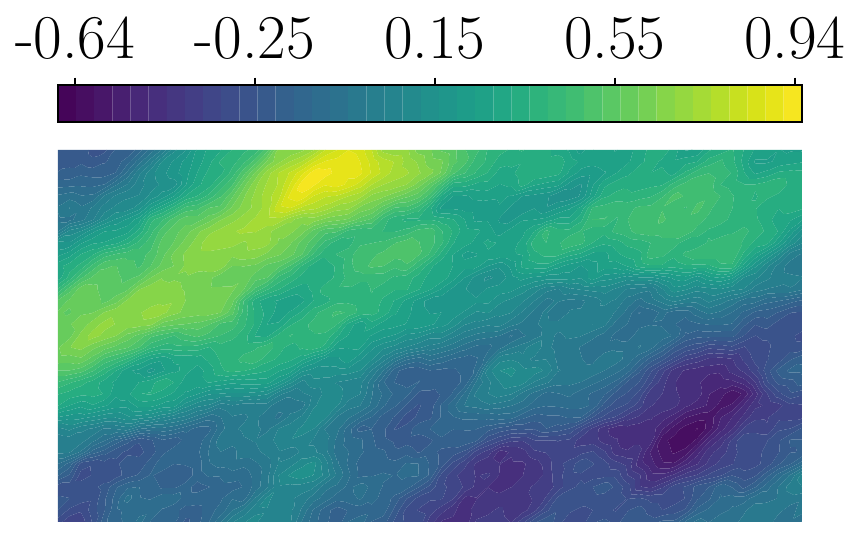}} &
     \raisebox{-.5\height}{\includegraphics[width = 0.24\linewidth]{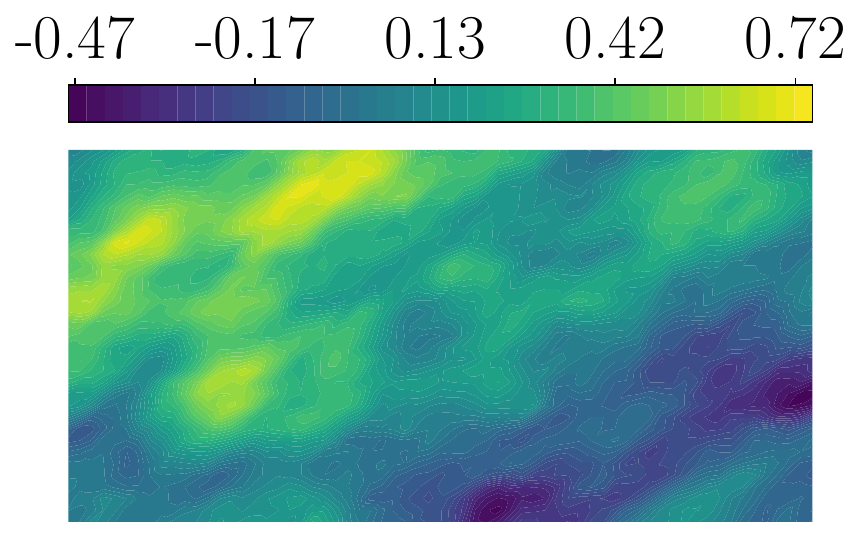}} \\
    \end{tabular}
    \begin{tabular}{c c c c}
        \bf \makecell{Mean estimate\\ via MCMC} & \bf \makecell{MAP--mean\\ absolute error} &  \bf \makecell{Pointwise variance\\estimate via MCMC} & \bf \makecell{LA--posterior\\ pointwise variance\\ absolute error}\\
        \raisebox{-.5\height}{\includegraphics[width = 0.24\linewidth]{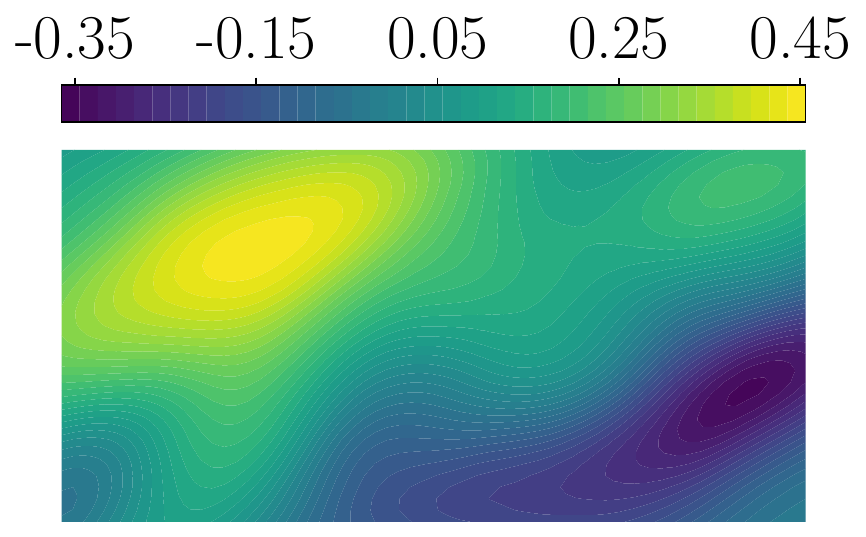}} & \raisebox{-.5\height}{\includegraphics[width = 0.24\linewidth]{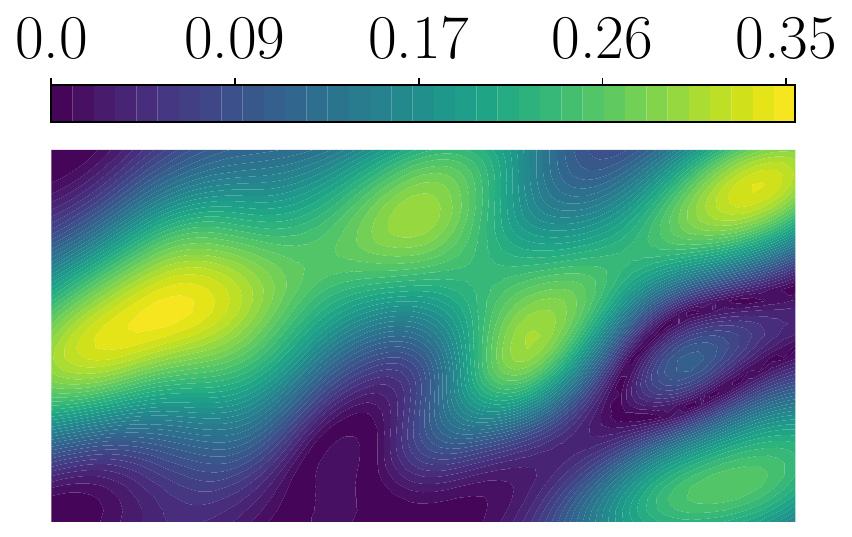}} & \raisebox{-.5\height}{\includegraphics[width = 0.24\linewidth]{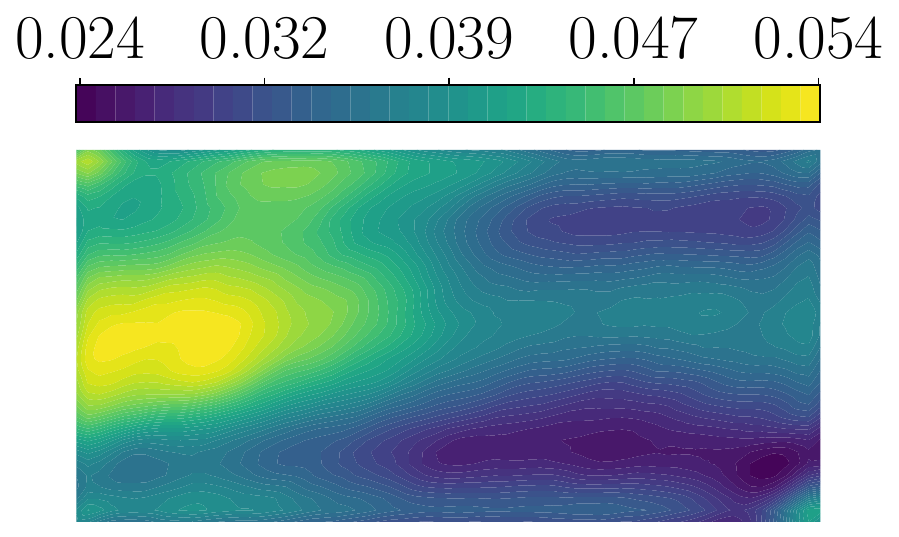}} & \raisebox{-.5\height}{\includegraphics[width = 0.24\linewidth]{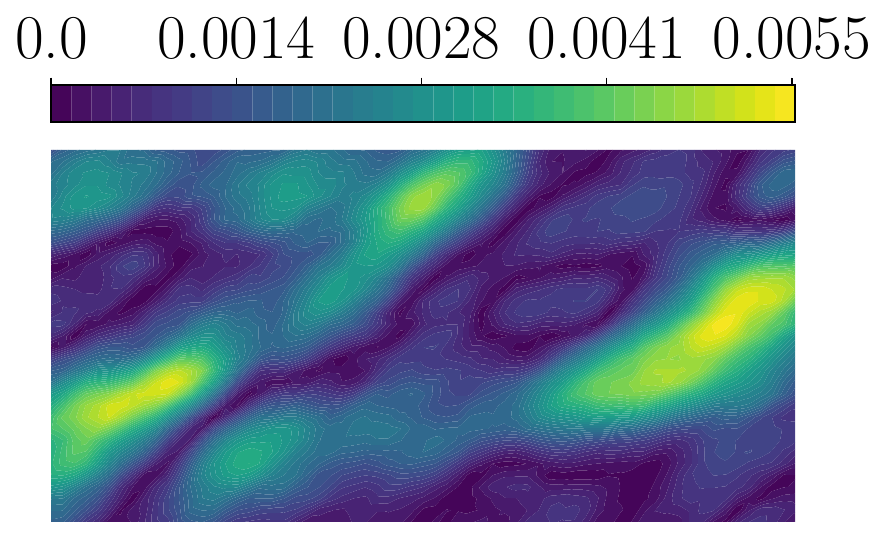}} 
    \end{tabular}
    \renewcommand{\arraystretch}{5}
    \addtolength{\tabcolsep}{5pt}
    \caption{Visualization of relevant statistics from Markov chains collected using mMALA for inference of a heterogeneous hyperelastic material property in \cref{sec:hyperelastic}.}
    \label{fig:hyper_posterior}
\end{figure}

\bibliography{main}

\end{document}